\documentclass[10pt, reqno]{amsart}
\usepackage{appendix}
\usepackage[a4paper, centering]{geometry}
\usepackage{xcolor}
\usepackage[utf8]{inputenc}
\usepackage[colorlinks=true,hyperindex,pagebackref,linktocpage=true]{hyperref}
\usepackage{cleveref}
\usepackage{amsfonts}
\usepackage{amssymb}
\usepackage{tikz-cd}
\usepackage[mathscr]{euscript}
\hypersetup{
	colorlinks,
	linkcolor={violet!60!black},
	citecolor={cyan!60!black},
	urlcolor={orange!60!black}
}
\usepackage{stmaryrd}
\usepackage{varwidth}
\theoremstyle{plain}
\newtheorem{theorem}{Theorem}[section]
\newtheorem{proposition}[theorem]{Proposition}
\newtheorem{lemma}[theorem]{Lemma}
\newtheorem{corollary}[theorem]{Corollary}
\newtheorem{conjecture}[theorem]{Conjecture}

\theoremstyle{definition}
\newtheorem{definition}[theorem]{Definition}

\newtheorem{remark}[theorem]{Remark}

\newcommand{\nc}{\newcommand}
\nc{\on}{\operatorname}

\nc{\Q}{\mathbb{Q}}
\nc{\Z}{\mathbb{Z}}
\nc{\cl}{\mathrm{cl}}

\nc{\fraka}{{\mathfrak a}} \nc{\bba}{{\mathbf a}}
\nc{\frakb}{{\mathfrak b}}
\nc{\frakc}{{\mathfrak c}}
\nc{\frakd}{{\mathfrak d}}
\nc{\frake}{{\mathfrak e}}
\nc{\frakf}{{\mathfrak f}}
\nc{\frakg}{{\mathfrak g}}
\nc{\frakh}{{\mathfrak h}}
\nc{\fraki}{{\mathfrak i}}
\nc{\frakj}{{\mathfrak j}}
\nc{\frakk}{{\mathfrak k}}
\nc{\frakl}{{\mathfrak l}}
\nc{\frakm}{{\mathfrak m}}
\nc{\frakn}{{\mathfrak n}}
\nc{\frako}{{\mathfrak o}}
\nc{\frakp}{{\mathfrak p}}
\nc{\frakq}{{\mathfrak q}}
\nc{\frakr}{{\mathfrak r}}
\nc{\fraks}{{\mathfrak s}}
\nc{\frakt}{{\mathfrak t}}
\nc{\fraku}{{\mathfrak u}}
\nc{\frakv}{{\mathfrak v}}
\nc{\frakw}{{\mathfrak w}}
\nc{\frakx}{{\mathfrak x}}
\nc{\fraky}{{\mathfrak y}}
\nc{\frakz}{{\mathfrak z}}
\nc{\frakA}{{\mathfrak A}}
\nc{\frakB}{{\mathfrak B}}
\nc{\frakC}{{\mathfrak C}}
\nc{\frakD}{{\mathfrak D}}
\nc{\frakE}{{\mathfrak E}}
\nc{\frakF}{{\mathfrak F}}
\nc{\frakG}{{\mathfrak G}}
\nc{\frakH}{{\mathfrak H}}
\nc{\frakI}{{\mathfrak I}}
\nc{\frakJ}{{\mathfrak J}}
\nc{\frakK}{{\mathfrak K}}
\nc{\frakL}{{\mathfrak L}}
\nc{\frakM}{{\mathfrak M}}
\nc{\frakN}{{\mathfrak N}}
\nc{\frakO}{{\mathfrak O}}
\nc{\frakP}{{\mathfrak P}}
\nc{\frakQ}{{\mathfrak Q}}
\nc{\frakR}{{\mathfrak R}}
\nc{\frakS}{{\mathfrak S}}
\nc{\frakT}{{\mathfrak T}}
\nc{\frakU}{{\mathfrak U}}
\nc{\frakV}{{\mathfrak V}}
\nc{\frakW}{{\mathfrak W}}
\nc{\frakX}{{\mathfrak X}}
\nc{\frakY}{{\mathfrak Y}}
\nc{\frakZ}{{\mathfrak Z}}
\nc{\bbA}{{\mathbb A}}
\nc{\bbB}{{\mathbb B}}
\nc{\bbC}{{\mathbb C}}
\nc{\bbD}{{\mathbb D}}
\nc{\bbE}{{\mathbb E}}
\nc{\bbF}{{\mathbb F}} \nc{\bbf}{{\mathbf f}}
\nc{\bbG}{{\mathbb G}}
\nc{\bbH}{{\mathbb H}}
\nc{\bbI}{{\mathbb I}}
\nc{\bbJ}{{\mathbb J}}
\nc{\bbK}{{\mathbb K}}
\nc{\bbL}{{\mathbb L}}
\nc{\bbM}{{\mathbb M}}
\nc{\bbN}{{\mathbb N}}
\nc{\bbO}{{\mathbb O}}
\nc{\bbP}{{\mathbb P}}
\nc{\bbQ}{{\mathbb Q}}
\nc{\bbR}{{\mathbb R}}
\nc{\bbS}{{\mathbb S}}
\nc{\bbT}{{\mathbb T}}
\nc{\bbU}{{\mathbb U}}
\nc{\bbV}{{\mathbb V}}
\nc{\bbW}{{\mathbb W}}
\nc{\bbX}{{\mathbb X}}
\nc{\bbY}{{\mathbb Y}}
\nc{\bbZ}{{\mathbb Z}}
\nc{\calA}{{\mathcal A}}
\nc{\calB}{{\mathcal B}}
\nc{\calC}{{\mathcal C}}
\nc{\calD}{{\mathcal D}}
\nc{\calE}{{\mathcal E}}
\nc{\calF}{{\mathcal F}}
\nc{\calG}{{\mathcal G}}
\nc{\calH}{{\mathcal H}}
\nc{\calI}{{\mathcal I}}
\nc{\calJ}{{\mathcal J}}
\nc{\calK}{{\mathcal K}}
\nc{\calL}{{\mathcal L}}
\nc{\calM}{{\mathcal M}}
\nc{\calN}{{\mathcal N}}
\nc{\calO}{{\mathcal O}}
\nc{\calP}{{\mathcal P}}
\nc{\calQ}{{\mathcal Q}}
\nc{\calR}{{\mathcal R}}
\nc{\calS}{{\mathcal S}}
\nc{\calT}{{\mathcal T}}
\nc{\calU}{{\mathcal U}}
\nc{\calV}{{\mathcal V}}
\nc{\calW}{{\mathcal W}}
\nc{\calX}{{\mathcal X}}
\nc{\calY}{{\mathcal Y}}
\nc{\calZ}{{\mathcal Z}}

\nc{\scrA}{{\mathscr A}}
\nc{\scrB}{{\mathscr B}}
\nc{\scrC}{{\mathscr C}}
\nc{\scrD}{{\mathscr{D} }}
\nc{\scrE}{{\mathscr E}}
\nc{\scrF}{{\mathscr F}}
\nc{\scrG}{{\mathscr G}}
\nc{\scrH}{{\mathscr H}}
\nc{\scrI}{{\mathscr J}}
\nc{\scrJ}{{\mathscr I}}
\nc{\scrK}{{\mathscr K}}
\nc{\scrL}{{\mathscr L}}
\nc{\scrM}{{\mathscr M}}
\nc{\scrN}{{\mathscr N}}
\nc{\scrO}{{\mathscr O}}
\nc{\scrP}{{\mathscr P}}
\nc{\scrQ}{{\mathscr Q}}
\nc{\scrR}{{\mathscr R}}
\nc{\scrZ}{{\mathscr{Z}}}
\nc{\D}{{\on{D}}}
\nc{\Div}{{\on{Div}}}

\nc{\bnu}{{\bar{ \nu}}}
\nc{\olO}{\bar{\calO}}

\nc{\al}{{\alpha}} 
\nc{\be}{{\beta}}
\nc{\ga}{{\gamma}} \nc{\Ga}{{\Gamma}}
\nc{\hGa}{\hat{\Gamma}}
\nc{\ve}{{\varepsilon}} 
\nc{\la}{{\lambda}} \nc{\La}{{\Lambda}}
\nc{\om}{\omega} \nc{\Om}{\Omega} 
\nc{\sig}{{\sigma}} \nc{\Sig}{{\Sigma}}
\nc{\dR}{{\mathrm{dR}}}
\nc{\Perf}{{\mathrm{Perf}}}
\nc{\Gm}{{\mathbb{G}_m}}
\nc{\colim}{{\on{colim}}}
\nc{\iwcentral}{\scrZ_{{\calI\calW}}}
\newcommand{\minus}{\scalebox{0.75}[1.0]{$-$}}
\nc{\pIserre}{\mathrm{Perv}_I^0}
\makeatletter
\DeclareFontEncoding{LS1}{}{}
\DeclareFontSubstitution{LS1}{stix}{m}{n}
\DeclareMathAlphabet{\rhomalpha}{LS1}{stixscr}{m}{n}
\makeatother

\nc{\Spa}{\on{{Spa}}}
\nc{\Spd}{\on{{Spd}}}
\nc{\tnb}{\psi_{\rm tame}}
\nc{\oM}{\overline{{M}}}
\nc{\op}{{\on{op}}}
\nc{\ad}{{\on{ad}}}
\nc{\alg}{{\on{alg}}}
\nc{\Ad}{{\on{Ad}}}
\nc{\Adm}{{\on{Adm}}} \nc{\aff}{{\on{af}}}
\nc{\Aut}{{\on{Aut}}}
\nc{\Bun}{{\on{Bun}}}
\nc{\cha}{{\on{char}}}
\nc{\der}{{\on{der}}}
\nc{\Der}{{\on{Der}}}
\nc{\diag}{{\on{diag}}}
\nc{\End}{{\on{End}}}
\nc{\Fl}{{\mathrm{Fl}}}
\nc{\Tr}{{\on{Transp}}}
\nc{\TR}{{\calT\!\calR}}
\nc{\Gal}{{\on{Gal}}}
\nc{\Gr}{{\on{Gr}}}
\nc{\Hk}{{\on{Hk}}}
\nc{\rH}{{\on{H}}}
\nc{\Hom}{{\on{Hom}}}
\nc{\IC}{{\on{IC}}}
\nc{\id}{{\on{id}}}
\nc{\Id}{{\on{Id}}}
\nc{\ind}{{\on{ind}}}
\nc{\Ind}{{\on{Ind}}}
\nc{\Lie}{{\on{Lie}}}
\nc{\Pic}{{\on{Pic}}}
\nc{\pr}{{\on{pr}}}
\nc{\Res}{{\on{Res}}}
\nc{\res}{{\on{res}}} \nc{\Sat}{{\on{Sat}}}
\nc{\spc}{{\on{sc}}}
\nc{\drv}{{\on{der}}}
\nc{\sgn}{{\on{sgn}}}
\nc{\Spec}{{\on{Spec}}}\nc{\Spf}{\on{Spf}} 
\nc{\Sph}{\on{Sph}}
\nc{\St}{{\on{St}}}
\nc{\tr}{{\on{tr}}}
\nc{\Mod}{{\mathrm{-Mod}}}
\nc{\Hilb}{{\on{Hilb}}} 
\nc{\Ext}{{\on{Ext}}} 
\nc{\vs}{{\on{Vec}}}
\nc{\ev}{{\on{ev}}}
\nc{\nO}{{\breve{\calO}}}
\nc{\tS}{{\tilde{S}}}
\nc{\spe}{{\on{sp}}}
\nc{\loc}{{\on{loc}}}
\nc{\pre}{{\on{pre}}}
\nc{\del}{\Delta}
\nc{\nab}{\nabla}
\nc{\ql}{\bar{\mathbb{Q}}_\ell}

\nc{\co}{\colon}
\nc{\dia}{{\diamondsuit}}

\nc{\nscrR}{{\mathscr{R}^{\on{nr}}}}

\nc{\GL}{{\on{GL}}}

\nc{\Gl}{\on{Gl}} 
\nc{\GSp}{{\on{GSp}}}
\nc{\gl}{{\frakg\frakl}}
\nc{\SL}{{\on{SL}}} 
\nc{\SU}{{\on{SU}}} 
\nc{\SO}{{\on{SO}}}
\nc{\PGL}{{\on{PGL}}}

\nc{\Conv}{{\on{Conv}}}
\nc{\Rep}{{\on{Rep}}}
\nc{\Dom}{{\on{Dom}}}
\nc{\red}{{\on{red}}}
\nc{\act}{{\on{act}}}
\nc{\nr}{{\on{nr}}}
\nc{\ctf}{{\on{ctf}}}

\nc{\str}{{\on{-}}} 
\nc{\os}{{\bar{s}}}
\nc{\oeta}{{\bar{\eta}}}

\nc{\et}{\textup{\'et}}

\nc{\hookto}{\hookrightarrow}
\nc{\longto}{\longrightarrow}
\nc{\leftto}{\leftarrow}
\nc{\onto}{\twoheadrightarrow}
\nc{\lonto}{\twoheadleftarrow}

\nc{\pot}[1]{ [\hspace{-0,5mm}[ {#1} ]\hspace{-0,5mm}] }
\nc{\rpot}[1]{ (\hspace{-0,7mm}( {#1} )\hspace{-0,7mm}) }

\nc{\CT}{\rm{CT}}

\numberwithin{equation}{section}

\begin{document}
	
	\title[Gaitsgory and Arkhipov--Bezrukavnikov in mixed characteristic]{Gaitsgory's central functor and the Arkhipov--Bezrukavnikov equivalence in mixed characteristic}
	
	\author[J. Ansch\"utz, J. Louren\c{c}o, Z. Wu, J. Yu]{Johannes Ansch\"utz, Jo\~ao Louren\c{c}o, Zhiyou Wu, Jize Yu}

	\address{Département de Mathématiques, Université Paris-Saclay, Bâtiment 307, 91405 Orsay Cedex, France}
	\email{janschuetz@math.cnrs.fr}

	\address{Mathematisches Institut, Universität Münster, Einsteinstrasse 62, Münster, Germany}
	\email{j.lourenco@uni-muenster.de}
	
	\address{Morningside Center of Mathematics, Chinese Academy of Sciences, No. 55, Zhongguancun East Road, Haidian District, Beijing, China}
	\email{wuzhiyou@amss.ac.cn}
	
	\address{Department of Mathematics, Rice University, Houston, TX 77005, U.S.A.}
	\email{jy120@rice.edu}
	
	\begin{abstract}
		We show that the nearby cycles functor for the $p$-adic Hecke stack at parahoric level is perverse t-exact, by developing a theory of Wakimoto filtrations at Iwahori level, and that it lifts to the $\bbE_1$-center. We apply these tools to construct the Arkhipov--Bezrukavnikov functor for $p$-adic affine flag varieties at Iwahori level, and prove that it is an equivalence for all classical groups and also exceptional groups of type $E_6$ and $E_7$. 
	\end{abstract}
	
	\maketitle
	\tableofcontents

	\section{Introduction} 
	\label{sec:introduction}
	
	Let $F$ be a $p$-adic field with ring of integers $O$ and residue field $k$, and let $G$ be a connected reductive $F$-group with a parahoric $O$-model $\calG$. The first goal of this paper is to define a $p$-adic analogue of the Gaitsgory central functor \cite{Gai01} sending perverse sheaves on the Hecke stack $\Hk_{G,C}$ over the completed algebraic closure $C$ of $F$ to central perverse sheaves on the Hecke stack $\Hk_{\calG,\bar k}$. The second goal of the paper is, when $\calG=\calI$ is Iwahori, to also define a $p$-adic analogue of the Arkhipov--Bezrukavnikov functor \cite{AB09} relating coherent sheaves on the dual Springer variety $\hat{\calN}_{\mathrm{Spr}}$ to constructible étale sheaves on $\Hk_{\calI,\bar k}$. During the introduction, we will assume for simplicity that $G$ is split and that the coefficients of our sheaves equal $\bar{\mathbb{Q}}_\ell$, as these hypotheses hold for most of the paper. We begin by recalling some of the representation-theoretic aspects at the level of Grothendieck groups.
	
	\subsection{Hecke algebras}
	We assume $G$ is a pinned split connected reductive $F$-group, i.e. equipped with a choice of Borel subgroup $B$, a maximal split torus $T\subset B\subset G$, and pinning isomorphisms for the root groups attached to positive simple roots with respect to $B$. The corresponding Iwahori--Weyl group $W=N(F)/T(O)$, where $N$ is the normalizer of $T \subset G$, admits a length function $\ell$ which makes $(W,\ell)$ into a quasi-Coxeter group, see \cite[page 7]{HR08} or \cite[Section 4.1.1]{AR}.

	Let $\bbH$ denote the affine Hecke algebra. Recall that the Iwahori-Matsumoto presentation, see \cite{IM65}, defines $\bbH$ as the $\mathbb{Z}[q^{\pm 1/2}]$-algebra generated by a basis $\{T_w \vert w\in W\}$ modulo relations $T_wT_{w'}=T_{ww'}$ if $\ell(w)+\ell(w')=\ell(ww')$, and $(T_{s}+q^{\minus 1/2})(T_s-q^{1/2})=0$ for all length $1$ elements $s$. Let $I:=\mathcal{I}(O)\subset G(F)$ be the Iwahori subgroup associated with $B$. Then, the Iwahori-Hecke algebra $\calH:=C_c(I\backslash G(F)/I)$ is the space of compactly
	supported smooth functions on $I\backslash G(F)/I$. Fixing a
	Haar measure on $G(F)$ so that $I$ has measure $1$, convolution of functions equips $\calH$ with the structure of a unital associative algebra. The affine Hecke algebra specializes to the Iwahori-Hecke algebra via the isomorphism
	$$
	\bbH\otimes_{\bbZ[q^{\pm 1}]}\bar\bbQ_\ell\simeq \calH,
	$$
	where $\bar \bbQ_\ell$ is regarded as a $\bbZ[q^{\pm 1/2}]$-algebra by mapping $q$ to the cardinality of $k$ (and thus choosing a square root of this integer in $\bar \bbQ_\ell$). Let $\bbH_\mathrm{f}\subset \bbH$ be the finite Hecke algebra associated with the finite Weyl group and the set of simple reflections. The \textit{anti-spherical module} is defined as $$
	\bbM^\mathrm{as}:=\bbZ[q^{\pm 1/2}]^\mathrm{sgn}\otimes_{\bbH_\mathrm{f}}\bbH,
	$$
	where $T_w\in \bbH_\mathrm{f}$ acts by multiplication by $(\minus 1)^{\ell({w})}q^{1/2}$ on $\bbZ[q^{\pm 1/2}]^\mathrm{sgn}$. We also define the anti-spherical module for $\calH$ to be $\calM^\mathrm{as}:=\bbM\otimes_{\bbZ[q^{\pm 1/2}]}\bar\bbQ_\ell $.
	
	According to Grothendieck's sheaf-function dictionary, the space of functions on the set of $\bbF_q$-points of a scheme has the category of complexes of coherent or constructible sheaves as its geometric counterpart. Let $\hat{G}$ denote the Langlands dual group of $G$ over $\bar{\bbQ}_\ell$, 
	and $\hat{U}\subset \hat{B}\subset \hat{G}$ be a Borel subgroup and its unipotent radical. Recall the Springer resolution of the nilpotent cone $\hat{\calN}\subset \mathrm{Lie}(\hat{G})$
	$$
	p_\mathrm{Spr}:\hat{\calN}_{\mathrm{Spr}}=\hat{G} \times^{\hat{B}} \mathrm{Lie}\,\hat{U}\rightarrow \hat{\calN}.
	$$
	The Steinberg variety is defined as $\hat{\mathrm{St}}:=\hat{\calN}_{\mathrm{Spr}}\times_{\hat{\calN}}\hat{\calN}_{\mathrm{Spr}}$.
	Kazhdan-Lusztig \cite{KL87} showed that the affine Hecke algebra is isomorphic to the Grothendieck group $ K_0([\hat{G}\times\bbG_m\backslash \hat{\mathrm{St}}])$, where the latter has an algebra structure induced by convolution.
	In particular, the anti-spherical module $\bbM^\mathrm{as}$ is identified with $K_0([\hat{G}\times\bbG_m\backslash\hat{\calN}_\mathrm{Spr}])$. If we forget the $\bbG_m$-equivariance, then we recover both the Iwahori-Hecke algebra $\calH$ and its anti-spherical module $\calM^{\mathrm{as}}$.
	

	
	On the other hand, it follows from the work of Iwahori--Matsumoto \cite{IM65} that the Iwahori-Hecke algebra coincides with $K_0(\scrP(\Hk_{\calI}))$ where $\scrP$ denotes the category of perverse sheaves on the Hecke stack $\Hk_{\calI}=L^+\calI\backslash LG/L^+\calI$. The natural action of $\calH$ on $\calM^\mathrm{as}$ induces a surjective map $\calH\rightarrow \calM^\mathrm{as}$ with kernel generated by the Kazhdan-Lusztig basis elements indexed by the $w\in W$, which are not minimal in their left $W_{\mathrm{fin}}$-coset. This leads us to consider the \textit{anti-spherical category} of perverse sheaves $\scrP_{\mathrm{as}}(\Hk_{\calI})$ given as the Serre quotient by the $\IC$-sheaves index by those $w$.
	Another approach to realize $\calM^{\mathrm{as}}$ is as the $I$-invariants of the compact induction to $G(F)$ of a generic character $\chi$ of the unipotent radical $I_u^\mathrm{op}$ of the opposite Iwahori subgroup. This is the so-called Iwahori--Whittaker model and its categorification plays an important role later in our arguments.
	
	We have observed that there is an abundance of spaces and sheaves that seem related to affine Hecke algebras and their anti-spherical modules. In the next sections, we will explain how to upgrade these isomorphisms to equivalences of stable $\infty$-categories. A guiding principle for this is the fact that there are certain objects which serve as building blocks for the various categories and we must track down where they get sent to. This is motivated by Bernstein's construction of translation elements $\theta_\nu$ in the affine Hecke algebra whose trace along the finite Weyl orbits are central, and also of an isomorphism between the spherical Hecke algebra and the center of the Iwahori-Hecke algebra.
	
	
	
	\subsection{The central functor}
	The first goal of this paper is to fully develop the central functor $\scrZ$ for $p$-adic groups, in analogy with Gaitsgory's central functor from \cite{Gai01} in the function field case. The correct geometric setup for this construction is naturally the world of diamonds or more generally v-sheaves on perfectoids of characteristic $p$. Indeed, for any parahoric $O$-model $\calG$, we have a Hecke stack $\mathrm{Hk}_{\calG}$ defined over $\mathrm{Spd}(O)$ and we can define a nearby cycles functor $R\Psi$ by push-pull along the geometric fibers over a complete algebraic closure $C$ of $F$ and the residue field $k$. Indeed, this was already partially exploited in \cite{AGLR22} to define certain complexes of sheaves $\scrZ(V)=R\Psi(\mathrm{Sat}(V))$, where $\mathrm{Sat}$ denotes the geometric Satake equivalence of Fargues--Scholze \cite{FS21}. It was proved in \cite{AGLR22} that these complexes are algebraic and constructible, that they carry certain centrality isomorphisms, and finally that they are supported at the $V$-admissible locus $ \calA_{\calG,V}$.
	
	Two very important properties of the functor $\scrZ$ remained however elusive in \cite{AGLR22}, namely verifying that the centrality isomorphisms of $\scrZ$ satisfy various expected compatibilities that make it into a central functor, and that it lands in the category of perverse sheaves. Our most important results in this direction can be resumed as follows:
	
	\begin{theorem}[\Cref{thm_E2_mor,thm_perversity_new}, \Cref{cor_Z_conv_exact,unipotent_new}]
		The functor $\scrZ\colon \mathrm{Rep}(\hat{G}) \to \scrD_{\mathrm{ula}}(\Hk_{\calG, k})$ lifts to an $\bbE_2$-monoidal functor towards the $\bbE_1$-center of the right side. Moreover, each $\scrZ(V)$ is perverse, and has unipotent monodromy. If $\calG=\calI$ is Iwahori, then $\scrZ(V)$ is convolution-exact and admits a Wakimoto filtration whose associated graded equals $\scrJ(V_{|\hat{T}})$.
	\end{theorem}
	
	Let us explain a bit of the notions and ingredients that go into the above theorem. Our treatment of the centrality of $\scrZ$ is to our best knowledge the only one that uses the machinery of stable $\infty$-categories, which entails additional higher coherent homotopies. An important technical tool is the notion of an abstract six-functor formalism in the sense of Mann \cite{Man22,Man22b}, which allows us to work at the level of the category of correspondences. Once we are there, we perform the usual fusion trick of looking at the disjoint locus of $(\mathrm{Spd}(O_C))^2$ and conclude the desired monoidality via full faithfulness of pullback away from the diagonal for those sheaves which are perverse over $(\mathrm{Spd}C)^2$. To obtain this full faithfulness, we apply a certain calculation of nearby cycles of kimberlites from \cite{GL24}.
	
	Trying to prove perversity of $\scrZ(V)$ was the genesis of this project. Contrary to the function field case, we cannot rely on Artin vanishing to provide us with this crucial fact. Instead, we first consider an Iwahori $\calI$ and look at the Wakimoto functor $\scrJ\colon \Rep(\hat{T}) \to \scrP(\Hk_\calI)$ following \cite{AB09}, but defined instead at the level of complexes as in \cite{AR}. The centrality of $\scrZ(V)$ implies that it lies in the full subcategory generated by the essential image of $\scrJ$ under extensions. Each graded piece can then be recovered by invoking geometric Satake and a certain orthogonality with respect to the constant terms $\mathrm{CT}_{B^{\mathrm{op}}}$, which proves perversity and the existence of a Wakimoto filtration all at once. This differs considerably from the strategy in \cite{AB09}, which exploits both perversity and convolution-exactness (known in equicharacteristic by Artin vanishing). 
	Perversity in the general parahoric case can be deduced from the Iwahori one, based on a suggestion of Achar, which we learned from Cass--Scholbach--van den Hove \cite{CvdHS24}, who adapted our argument for Iwahori models to their setting. 
	From the Wakimoto filtration, we can also deduce the convolution-exactness and unipotency of the monodromy operator induced by the Galois group (note that $\mathrm{Sat}(V)$ descends to $\mathrm{Spd}(F)$ with trivial inertia action).
	In the meantime, the Wakimoto filtration has been decisively used in \cite{GL24} to give a new proof of unibranchness (i.e., topological normality) of local models in complete generality.
	
	\subsection{The AB functor and Iwahori--Whittaker sheaves}
	
	The next part of our paper carries out the construction of the various functors from \cite{AB09} and, except for treating $\infty$-categorical questions carefully, does not significantly diverge from it. Recall that we work here with the Springer stack $[\hat{G}\backslash\hat{\calN}_{\mathrm{Spr}}]$ that resolves the corresponding nilpotent stack of $\hat{G}$. We have the following result:
	
	\begin{theorem}[\Cref{prop_constr_AB_funct_F}]
		There is a monoidal functor $\scrF\colon\mathrm{Perf}([\hat{G}\backslash \hat{\calN}_{\mathrm{Spr}}]) \to \scrD_{\mathrm{cons}}(\Hk_{\calI})$ extending $\scrZ\times\scrJ$.
	\end{theorem}
	
	Let us briefly describe the construction of $\scrF$ following \cite{AB09}. The main idea consists in defining an analogous functor starting from the quotient stack $[\hat{G}\times \hat{T}\backslash \hat{\calN}_{\mathrm{Spr}}^{\mathrm{af}}]$ containing $[\hat{G}\backslash \hat{\calN}_{\mathrm{Spr}}]$ as a dense open stack, where $\hat{\calN}_{\mathrm{Spr}}^{\mathrm{af}} $ is the affine hull of the canonical $\hat{T}$-torsor over $\hat{\calN}_{\mathrm{Spr}}$. The projective objects of the category of coherent sheaves on the enlarged stack can be mapped to $\scrP(\Hk_{{\calI}})$ by using the functors $\scrZ$, $\scrJ$, and the nilpotent monodromy endomorphism of the former, after verifying the Plücker relations. After deriving this functor to perfect complexes on the affine stack, we are reduced to showing vanishing on complexes supported at the complement of $[\hat{G}\backslash \hat{\calN}_{\mathrm{Spr}}]$.
	
	Next, we study the category $\scrD_{\mathrm{cons}}(\Hk_{\calI\calW})$ of Iwahori--Whittaker sheaves. This is the stable $\infty$-category of $(L^+\calI_{\mathrm{u}}^{\mathrm{op}},\calL)$-twisted equivariant constructible sheaves on $\Fl_{ \calI}$, where $\calL$ is a certain character sheaf obtained via the Artin--Schreier cover. It also carries a perverse t-structure and the category $\scrP(\Hk_{\calI\calW})$ is a highest weight category in the sense of Beilinson--Ginzburg--Soergel \cite{BGS96}, with simple and tilting objects indexed by $\bbX_\ast(T)$. We get a perverse t-exact averaging functor $\mathrm{av}_{\calI\calW}\colon \scrD_{\mathrm{cons}}(\Hk_{\calI}) \to \scrD_{\mathrm{cons}}(\Hk_{\calI\calW})$ given by left convolution against the simple object attached to $0\in \bbX_\ast(T)$ and we denote by $\scrZ_{\calI\calW}$, resp.~$\scrF_{\calI\calW}$, the composition $\mathrm{av}_{\calI\calW}\circ \scrZ$, resp.~$\mathrm{av}_{\calI\calW}\circ \scrF$.
	
	\begin{theorem}[\Cref{thm_tilting}]
		If $G$ has enough minuscules, then $\scrZ_{\calI\calW}(V)$ is tilting.
	\end{theorem}
	
	Notice that in \cite{AB09} there is no assumption on $G$. Unfortunately, we are missing a crucial ingredient replacing the $\mathbb{G}_{m,k}$-action given by rotating the uniformizer, which is impossible in the $p$-adic setting. Contrary to what was asserted in an initial version of this paper, the impact of this gap is very mild, and we get the result for all groups with enough minuscules, see \Cref{def_enough_minusc}, a class comprised of all classical groups, and also exceptional ones of types $E_6$ and $E_7$. This oversight on our side was brought to our attention by the work of Dhillon--Taylor \cite{DT25}. 
	
	Let us explain how the proof works, so the reader can better grasp the gap above. First of all, the tilting property propagates under convolution and can be verified on adjoint quotients, so we may assume $V$ is either minuscule or quasi-minuscule by a lemma of Ngô--Polo \cite{NP01}. In the minuscule case, all the weights are comprised in a finite Weyl orbit, so one can easily verify the given property. In the quasi-minuscule case, we must still handle the (co)restriction of $\scrZ_{\calI\calW}(V)$ to the weight $0$ Iwahori--Whittaker cell. Here, the vanishing can be achieved by calculating the alternating sum of the $\mathrm{Ext}$ groups via an argument on Grothendieck groups and finally bound the dimension of the $\bar{\bbQ}_\ell$-vector space $\mathrm{Hom}(\scrZ_{\calI\calW}(1),\scrZ_{\calI\calW}(V))$ accordingly. Now, this bound is achieved in \cite{AB09} via the theory of the regular quotient described below together with the fact that the monodromy operator is defined for every sheaf in $\scrP(\Hk_{\calI})$ as it is induced by the $\bbG_{m,k}$-action given by rotation. For groups with enough minuscules, we can realise every representation up to central isogeny as a direct summand of a tensor product of minuscule representations, so we still get the tilting property.
	
	The last step in proving that $\scrF_{\calI\calW}$ is an equivalence (now, necessarily assuming enough minuscules) revolves around the regular orbit $\calO_{\mathrm{r}} \subset \hat{\calN}$ inside the nilpotent cone. The Springer resolution is an isomorphism above this $\hat{G}$-orbit, and hence we should find a category of étale sheaves that plays a similar role. For this, we look at the Serre quotient $\scrP_0(\Hk_{\calI})$ of perverse sheaves on the Hecke stack obtained by modding out IC sheaves with positive-dimensional support.
	
	\begin{theorem}[\Cref{prop_tannakian_for_perv_I0}, \Cref{H_equals_the_full_centralizer}]
		If $G$ has enough minuscules, then there is a symmetric monoidal equivalence $\scrP_0(\Hk_{\calI}) \xrightarrow{\sim} \mathrm{Rep}(Z_{\hat{G}}(n_0))$, where $n_0$ is a regular nilpotent element.
	\end{theorem}
	
	Together with the tilting property, this result is key in order to prove the Arkhipov--Bezrukavnikov equivalence, as it induces certain injections of $\mathrm{Hom}$ maps. Let us remark that the most delicate point in the above theorem consists in showing that $n_0$ is regular. For this, we use the theory of weights by descending $\scrZ(V)$ to a mixed sheaf and calculating its monodromy filtration. In \cite{AB09}, one applies Gabber's local weight-monodromy theorem, see \cite{BB93}, stating that the weight filtration equals the monodromy filtration, and then calculates the former via the affine Hecke algebra. This is not available for our nearby cycles, unless $\mu$ is minuscule, by work of Hansen--Zavyalov \cite{HZ23} combined with the representability theorem in \cite{AGLR22}. Again, we can only easily reproduce this argument for groups with enough minuscules. It would be possible to adapt a different argument due to Bezrukavnikov--Riche--Rider \cite{BRR20}, but this would lead us into some detours that seem unnecessary, as we do not have the tilting property for other groups.
	
	Let us finish by stating the second main result of this paper, i.e. the AB equivalence for $p$-adic groups with enough minuscules. 
	
	\begin{theorem}[\Cref{thm_A-B}]
		If $G$ has enough minuscules, then the functor $\scrF_{\calI\calW}\colon \mathrm{Perf}([\hat{G}\backslash \hat{\calN}_{\mathrm{Spr}}])\to \scrD_{\mathrm{cons}}(\Hk_{\calI\calW})$ is an equivalence.
	\end{theorem}
	
	Let us remark that in \Cref{coro_av_IWas_is_an_equiv} we identify the anti-spherical Serre quotient $\mathscr{P}_{\mathrm{as}}(\mathrm{Hk}_{\mathcal{I}})$ with the category $\mathscr{P}(\mathrm{Hk}_{\mathcal{IW}})$ of Iwahori--Whittaker perverse sheaves, so upon passing to Grothendieck groups, the equivalence above really recovers the classical setup explained in the beginning of this introduction.
	We strongly believe that the theorem above must also hold for general split connected reductive groups $G$. Yun--Zhu have announced in conference talks regarding work of preparation of Hemo--Zhu, see also \cite{Zhu20}, a proof of the full Bezrukavnikov equivalence \cite{Bez16} for $p$-adic groups, that builds on a colimit presentation in terms of double quotients of parahoric jet groups due to Tao--Travkin \cite{TT20}. Recently, Bando \cite{Ban23} also gave a distinct proof of the Bezrukavnikov equivalence for $p$-adic groups by comparing constructible-étale sheaves in equi- and mixed characteristic via an ingenious geometric construction. However, these previous methods do not yield concrete knowledge about the central functor, whereas our paper places $\scrZ$ right at the center of it all. We also think that our functor $\scrZ$ will naturally appear in the picture if one studies étale sheaves on $p$-adic Hecke stacks, see, e.g., the unibranchness theorem of \cite{GL24}, and thus it must play a role in comparing the Zhu \cite{Zhu20} and the Fargues--Scholze \cite{FS21} variants of a categorical $p$-adic local Langlands correspondence. A natural task for the future will be to explain if and how all of the previous approaches fit together, namely by comparing a priori different central functors.
	
	\subsection{Acknowledgements}
	This paper owes a lot to the machinery developed in \cite{AGLR22}, so we are first of all thankful to Ian Gleason and Timo Richarz for their contribution to that project, and also further conversations involving this project. We benefited from a workshop on the Bezrukavnikov equivalence in Essen--Münster--Bonn organized by the first two authors together with Ulrich Görtz, Eugen Hellmann, and Konrad Zou, and we are thankful to all of the participants. This project has received funding from the Excellence Cluster of the Universität Münster, and the ERC Consolidator Grant 770936 via Eva Viehmann. The fourth author benefited from participating in the special academic year at the Institute for Advanced Study. He thanks Zurich Insurance Company for the support and organizers and all participants of the special year program. Part of the project was conducted when the fourth author was visiting the Max-Planck Institute for Mathematics, he is grateful to the institute for host. We would also like to acknowledge the relevant discussions we had or feedback we got from Pramod Achar, Robert Cass, Harrison Chen, Jean-François Dat, Gurbir Dhillon, Arnaud Étève, Thomas Haines, Linus Hamann, David Hansen, Xuhua He, Eugen Hellmann, Claudius Heyer, Tasho Kaletha, Arthur-César Le Bras, Lucas Mann, Cédric Pépin, Simon Riche, Peter Scholze, Jeremy Taylor, Thibaud van den Hove, Eva Viehmann, Yifei Zhao, Xinwen Zhu, Konrad Zou. We also thank the anonymous referee for carefully reading the paper and suggesting several improvements.
	
	\subsection{Notation}
	\label{section-notation}
	
	Unfortunately, we will have to use a lot of notations. Thus, let us address this once and for all and define the following objects, which will occur in the whole text.
	
	First, let us discuss scheme-theoretic notations.
	\begin{itemize}
		\item $p$ a prime,
		\item $F$ a finite extension of either $\mathbb{Q}_p$ or $\breve{\mathbb{Q}}_p$ with ring of integers $O$, and (perfect) residue field $k$.
		\item $\overline{F}$ an algebraic closure of $F$ and $\Gamma:=\mathrm{Gal}(\overline{F}/F)$ the absolute Galois group,
		\item $\breve{\Gamma}\subseteq \Gamma$ the inertia subgroup and $\Gamma^{\mathrm{un}}:=\Gamma/\breve{\Gamma}$ the unramified quotient,
		\item $\breve{F}$ the completion of the maximal unramified extension of $F$ in $\overline{F}$, $\breve{O}\subseteq \breve{F}$ its ring of integers, and $\overline{k}$ the residue field of $\breve{O}$,
		\item $G$ a quasi-split reductive group over $F$, 
		\item $S\subseteq T\subseteq B\subseteq G$ a maximal split torus $S\subseteq G$, $T$ its centralizer (a maximal torus in $G$ as $G$ is quasi-split), and a Borel $B\subseteq G$ containing $T$,
		\item $B^-\subseteq G$ is the opposite Borel of $B$,
		\item $N:=N_G(T)$ denotes the normalizer of $T$ in $G$,
		\item $\calS$ the connected  N\'eron model of $S$ over $O$,
		\item $\calT$ the connected N\'eron model of $T$ over $O$,
		\item If $H/F$ is a torus, then $\bbX_\bullet(H)$, resp.~$\bbX^\bullet(H)$ denote the groups of (geometric) cocharacters, resp.~characters of $H$,
		\item $\bbX_\bullet:=\bbX_\bullet(T),\ \bbX^\bullet:=\bbX_\bullet(T)$,
		\item $\bar{\bbX}_\bullet:=\bbX_{\bullet,\breve{\Gamma}}$, where the subscript $\breve{\Gamma}$ denotes the coinvariants,
		\item $\bbX_\bullet^+,\ \bbX^{\bullet,+}$ denote the dominant cocharacters resp.\ dominant characters of $T$ with respect to $B$,
		\item $\bbX_\bullet(S)^+,\ \bbX^{\bullet,+}(S)$ denote the dominant cocharacters resp.\ dominant characters of $S$ with respect to $B$, 
	\end{itemize}
	
	Next, let us introduce combinatorial notations.
	
	\begin{itemize}
		\item $W:=N(\breve{F})/\calT(\breve{O})$ the Iwahori-Weyl group of $T$, also called extended affine Weyl group,
		\item $\mathcal{A}(G,S)$ the appartment associated with $S$, identified with $\bbX_\bullet(S)_{\bbR}$ for pinned $G$,
		\item $\bba\subseteq \mathcal{A}(G,S)$ a fixed alcove,
		\item $\bbf \subseteq \mathcal{A}(G,S)$ a facet contained in the closure of $\bba$,
		\item $\bbS\subseteq W$ the set of reflections at the walls of $\bba$, also called the set of simple reflections,
		\item $W_\aff\subseteq W$ the affine Weyl group, which is the Coxeter group generated by the simple reflections,
		\item $\Omega_{\bba}$ the stabilizer of $\bba$, which yields an isomorphism
		\begin{equation}
			W\cong W_\aff\rtimes \Omega_{\bba}.
		\end{equation}
		\item $\ell\colon W\to \bbN_{\geq 0}$ the length function on $W$, i.e., the unique function $\ell(-)\colon W\to \mathbb{N}_{\geq 0}$, which extends the length function $\ell(-)\colon W_\aff\to \mathbb{N}_{\geq 0}$ on the Coxeter group $W_\aff$, such that $\ell(\tau)=0$ for $\tau\in \Omega_{\bba}$.
		\item $\leq$ is the Bruhat order on $W$, i.e., $w\leq w^\prime$ for $w=(w_\aff,\tau), w^\prime=(w^\prime_\aff,\tau^\prime)\in W\cong W_\aff\rtimes \pi_1(G)_{\breve{\Gamma}}$ if and only if $\tau=\tau^\prime$ and $w_\aff\leq w^\prime_\aff$ for the Bruhat order $\leq$ on $W_\aff$ coming from its Coxeter structure,
		\item $W_{\mathrm{fin}}=N(\breve{F})/T(\breve{F})$ the finite Weyl group, which sits in an exact sequence
		\begin{equation}
			1\to \bar{\bbX}_\bullet\to W\to W_{\mathrm{fin}}\to 1.
		\end{equation}
		\item $t_{\bar{\nu}}\in W$ is the translation element associated with $\bar{\nu}\in \bar{\bbX}_\bullet$. 
		\item $w_{\bar{\nu}}$ denotes the minimal length element in the coset $W_\mathrm{fin}t_{\bar{\nu}}$ for $\bar{\nu}\in\bar{\bbX}_\bullet$, upon choosing an origin for $\calA(G,S)$.
	\end{itemize}
	
	Now let us define notations related to affine flag varieties and perfect geometry.
	
	\begin{itemize}
		\item $\mathrm{Alg}^{\mathrm{perf}}_k$ the category of perfect $k$-algebras,
		\item for a scheme $X$ over $k$ we denote by $X^{\rm{pf}}$ its perfection, 
		\item for $R\in \mathrm{Alg}^{\mathrm{perf}}_k$ we set $W_O(R):=O\otimes_{W(k)}W(-)$, the ring of $O$-Witt vectors,
		\item if $X/O$ is an affine scheme of finite type, then $L^+X\colon (\mathrm{Alg}^{\mathrm{perf}})\to (\mathrm{Sets}),\ R\mapsto X(W_O(R))$ is the positive loop functor for $X$,
		\item if $Z/F$ is an affine scheme of finite type, then $LZ\colon (\mathrm{Alg}^{\mathrm{perf}})\to (\mathrm{Sets}),\ R\mapsto Z(W_O(R)[1/p])$ is the loop functor for $Z$,
		\item $\calI/O$ the Iwahori group scheme for $G$ associated with the alcove $\mathbf{a}$,
		\item $\calG/O$ the parahoric model of $G$ associated with the facet $\bbf$,
		\item the quotient of \'etale sheaves $\Fl_\calG:=LG/L^+\calG$ is the (partial) affine flag variety for $\calG$,
		\item the quotient stack $\Hk_{\calG}=[L^+\calG\backslash \Fl_{ \calG}]$ in the étale topology is the Hecke stack for $\calG$,
		\item the quotient stack $\Hk_{(\calI,\calG)}=[L^+\calI\backslash \Fl_{ \calG}]$ in the étale topology is the Hecke stack for the pair $(\calI,\calG)$.		
	\end{itemize}
	Next, let us introduce some cohomological notations.
	
	\begin{itemize}
		\item $\ell\neq p$ a prime,
		\item $\Lambda$ an algebraic extension of $\bbF_\ell$ or $\bbQ_\ell$.  
		\item $\scrD_\et(-):=\scrD_\et(-,\Lambda)$ denotes the (left-completed) $\infty$-derived category of ``\'etale sheaves of $\Lambda$-modules'' on a small v-stack on either perfect schemes or perfectoid spaces of characteristic $p$, compare with \cite[Definition 1.7]{Sch17} for $\Lambda=\bbZ/\ell^n \bbZ$. For $\Lambda=\bbQ_\ell$, we invert $\ell$ as in Section 26 in \textit{loc.cit}.
		\item $\scrD_\et(\Hk_{(\calI,\calG)}):=\scrD_\et(\Hk_{(\calI,\calG)},\Lambda)^{\mathrm{bd}}$ denotes the $\infty$-category of ``\'etale sheaves of $\Lambda$-modules'' on $\Hk_{(\calI,\calG)}$, whose support is a finite subset of the underlying topological space of $\Hk_{(\calI,\calG)}$.
	\end{itemize}
	Finally, we collect our notations for the ``coherent'' side. Note that we consider these objects usually under the assumption that $\Lambda$ is a field extension of $\bbQ_\ell$.
	\begin{itemize}
		\item $\hat{G}$ denotes the dual group of $G$ over $\Lambda$,
		\item $\hat{T}\subseteq \hat{G}$ denotes the dual torus to $T$, and we identify $\bbX^\bullet(\hat{T})\cong \bbX_\bullet(T)$,
		\item $\hat{G}^\prime:=\hat{G}\times\hat{T}$,
		\item $\hat{T}\subseteq \hat{B}\subseteq \hat{G}$ denotes the Borel subgroup with dominant characters identifying with $\bbX_\bullet(T)^+$,
		\item $\hat{U}\subseteq \hat{B}$ is the unipotent radical of $\hat{B}$ with Lie algebra $\mathrm{Lie}(\hat{U})$,
		\item $\hat{\mathfrak{g}}:=\mathrm{Lie}(\hat{G})$ denotes the Lie algebra of $\hat{G}$,
		\item $\hat{\calN}\subseteq \hat{\mathfrak{g}}$ is the nilpotent cone, i.e., the closed subscheme of nilpotent elements,
		\item $p_{\mathrm{Spr}}\colon \hat{\calN}_{\mathrm{Spr}}:=\hat{G}\times^{\hat{B}}\mathrm{Lie}(\hat{U})\to \hat{\calN}$ denotes the Springer resolution of the nilpotent cone,
		\item $\hat{\calN}_{\mathrm{Spr}}^{\mathrm{qaf}}:=\hat{G}\times^{\hat{U}} \mathrm{Lie}(\hat{U})\to \hat{\calN}_{\mathrm{Spr}}$ denotes the canonical $\hat{T}$-torsor over $\hat{\calN}_{\mathrm{Spr}}$,
		\item $\hat{\calX}:=\mathrm{Spec}(\mathcal{O}(\hat{G}/\hat{U}))$ is the affine closure of the quasi-affine scheme $\hat{G}/\hat{U}$,
	\end{itemize}
	
	\section{Geometry of the affine flag variety}
	\label{sec:geometry-affine-flag}
	
	In this section, we want to recall the geometry of the (Witt vector) partial affine flag variety $\Fl_\calG$, which was first considered as an algebraic space in \cite[Section 1.4]{Zhu17}. Its representability by an ind-(perfected projective $k$-scheme) was then proven in \cite[Corollary 9.6]{BS17} via reduction to $G=\GL_n$ and the construction of the determinant line bundle there.
	Let us note that the base change $\Fl_{\calG,\bar k}$ is the affine flag variety of the parahoric group $\calG \otimes_O \breve{O}$.
	Hence, geometric questions for $\Fl_\calG$ often reduce to the case $F=\breve{F}$. Our treatment will focus especially on $L^+\calI$-equivariant subvarieties of $\Fl_\calG$.
	
	\subsection{Schubert varieties and convolution}
	\label{sec:schub-vari-conv}
	
	During the entire paper, we will assume that the group $G$ is {\it residually split}. In fact, almost all of our arguments with sheaves take place when $F=\breve{F}$, except for a brief appearance of mixed sheaves, for which residual splitness is a lax enough assumption. This simplifies the Galois action on the Iwahori--Weyl group.
	
	\begin{lemma}
		\label{sec:geometry-affine-flag-1-residually-split}
		The following conditions are equivalent:
		\begin{enumerate}
			\item The $\Gamma$-action on $W$ is trivial.
			\item $\Gamma$ acts trivially on $\bar{\bbX}_\bullet$ (equivalently, $\Gamma^{\mathrm{un}}$ acts trivially on $\bar{\bbX}_\bullet$).
			\item $G$ is residually split, i.e., the reductive quotient $\calG_k^{\mathrm{red}}$ of the special fiber of every parahoric $O$-model of $G$ is split over $k$.
		\end{enumerate}
	\end{lemma}
	\begin{proof}
		By \cite[Proposition 9.10.10]{KP23} the group $G$ is residually split if and only if $\Gamma$ acts trivially on $\bar{\bbX}_\bullet$. If $\Gamma$ acts  trivially on $W$, then as well on $\bar{\bbX}_\bullet\subseteq W$. Assume now that $\Gamma$ acts trivially on $\bar{\bbX}_\bullet$. As $W$ is generated by $W_\aff$ and $\bar{\bbX}_\bullet$, it suffices to show that $\Gamma$ acts trivially on $W_\aff$. But $W_\aff$ embeds $\Gamma$-equivariantly into the group of affine transformations on $\bar{\bbX}_{\bullet,\bbR}$, and the $\Gamma$-action on the latter is trivial. 
	\end{proof}
	
	The geometry of the Iwahori orbits on the affine flag variety is summarized in the next lemma.
	
	\begin{lemma}
		\label{sec:geometry-affine-flag-1-stratification-of-affine-flag-variety}
		\begin{enumerate}
			\item The map $N(\breve{F})\to \Fl_{\calG},\ n\mapsto n\cdot L^+\calG$ induces a bijection
			\begin{equation}
				W/W_\bbf=\calT(\breve{O})\backslash N(\breve{F})/(N(\breve{F})\cap \calG(\breve{O}))\to \Hk_{(\calI,\calG)}
			\end{equation}
			on underlying topological spaces, i.e, the $L^+\calI$-orbits $\Fl_{(\calI,\calG),w}:=L^+\calI\cdot w\subset \Fl_{\calG}$ are indexed by $W/W_\bbf$.
			\item The $L^+\calI$-orbits on $\Fl_{\calG}$ form a stratification of $\Fl_{\calG}$, i.e., the closure $\Fl_{(\calI,\calG),\leq w}$ of a Schubert cell $\Fl_{(\calI,\calG),w}$ is a union of Schubert cells. More precisely, it is the unique closed perfect subscheme such that
			\begin{equation}
				\lvert \Fl_{(\calI,\calG),\leq w}\rvert=\bigcup\limits_{w^\prime\leq w}\lvert \Fl_{(\calI,\calG),w^\prime}\rvert
			\end{equation}
			for the Bruhat order on $W$.
		\end{enumerate}
	\end{lemma}
	\begin{proof}
		Statement (1) is essentially the Bruhat decomposition, see \cite[Section 6.5]{BT72} or \cite[Theorem 7.8.1]{KP23},
		\begin{equation}
			W/W_\bbf\cong \calI(\breve{O})\backslash G(\breve{F})/\calG(\breve{O})
		\end{equation}
		(applied over all formally unramified extensions $\breve{O}\subseteq \mathcal{O}'$ of discrete valuation rings). By étale descent, we have a natural map $W\to \Hk_{(\calI,\calG)}$, where we used the fact that $G$ is residually split to compute the left side. In particular, all the points of $\Hk_{(\calI,\calG)}$ are $k$ rational and are enumerated by $W/W_\bbf$.
		Then (2) follows by considering convolution and the Demazure resolutions, cf.\ \cite[Section 1.4]{Zhu17} for details.
	\end{proof}
	
	\begin{definition}
		\label{sec:geometry-affine-flag-5-definition-schubert-varieties}
		The (perfect) projective schemes $\Fl_{(\calI,\calG),\leq w}$ are called Schubert varieties, while their open and dense subschemes $\Fl_{(\calI,\calG), w}$ are called (Iwahori) Schubert cells.
	\end{definition}
	

	If $n\geq 1$, the contracted product
	\begin{equation}
		\underset{n-\textrm{factors}}{ \Fl_\calG\tilde{\times}\cdots \tilde{\times} \Fl_\calG}:=LG\times^{L^+\calG}\cdots \times^{L^+\calG} LG/L^+\calG
	\end{equation}
	is called the $n$-fold convolution product of $\Fl_\calG$. The convolution morphism
	\begin{equation}
		m:=m_{\Fl_{\calG}}\colon \Fl_\calG\tilde{\times} \cdots \tilde{\times} \Fl_\calG\to \Fl_\calG,\ \overline{(g_1,\ldots, g_n)}\mapsto g_1\cdots g_nL^+\calG
	\end{equation}
	has interesting geometric properties. If $X_1,\ldots, X_n\subseteq \Fl_{\calG}$ are (locally) closed $L^+\calG$-stable subschemes and $Y_1,\ldots, Y_n\subseteq LG$ their preimages, then we set
	\begin{equation}
		X_1\tilde{\times}\ldots \tilde{\times} X_n:= Y_1\times^{L^+\calG}\ldots \times^{L^+\calG} Y_n/L^+\calG,
	\end{equation}
	which is a (locally) closed subscheme of $\Fl_{\calG}\tilde{\times}\ldots \tilde{\times} \Fl_{\calG}$. A similar discussion applies the convolution product $\Fl_{\calI} \tilde{\times}\ldots \tilde{\times} \Fl_{\calI}\tilde{\times}\Fl_{ \calG}$, where the parahoric $\calG$ appears only in the last factor and all the other intermediate terms are given by the Iwahori $\calI$. We denote by $m$ the convolution morphism
	$$
	m: \Fl_{\calI} \tilde{\times}\ldots \tilde{\times} \Fl_{\calI}\tilde{\times}\Fl_{ \calG}\rightarrow \Fl_\calG
	$$
	by abuse of notation. 
	
	\begin{lemma}
		\label{sec:geometry-affine-flag-2-properties-of-convolution}
		Let $w_1,\ldots, w_n\in W$. Assume that $w_1\cdots w_n$ are right $W_\bbf$-minimal, i.e., it has minimal length in the right $W_\bbf$-coset, and $w_1\cdots w_n$ is a reduced expression.
		\begin{enumerate}
			\item The map
			\begin{equation}
				m: \Fl_{\calI, w_1}\tilde{\times}\cdots \tilde{\times} \Fl_{(\calI,\calG),w_n}\to \Fl_{\calG}
			\end{equation}
			has image in $\Fl_{(\calI,\calG), w_1\cdots w_n}$ and induces an isomorphism
			\begin{equation}
				\Fl_{\calI, w_1}\tilde{\times}\cdots \tilde{\times} \Fl_{\calI,w_n}\cong \Fl_{(\calI,\calG), w_1\cdots w_n}.
			\end{equation}
			\item We have
			\begin{equation}
				m(\Fl_{\calI,\leq w_1}\tilde{\times}\cdots \tilde{\times} \Fl_{(\calI,\calG),\leq w_n})\subseteq \Fl_{(\calI,\calG), \leq w_1\cdots w_n}
			\end{equation}
			and the map
			\begin{equation}
				m:\Fl_{\calI, \leq w_1}\tilde{\times}\cdots \tilde{\times} \Fl_{(\calI,\calG), \leq w_n}\to \Fl_{(\calI,\calG), \leq w_1\cdots w_n}
			\end{equation}
			is (perfectly) proper and birational.
			\item If $w\in W$ and $\tau\in \Omega_{\bba}$, then
			\begin{equation}
				\Fl_{(\calI,\calG), w}\to \Fl_{(\calI,\calG), \tau w},\ gL^+\calG\mapsto \tau g L^+\calG 
			\end{equation}
			is an isomorphism.
		\end{enumerate}
	\end{lemma}
	\begin{proof}
		It suffices to check the statements in the case that $F=\breve{F}$.
		Using induction on $n$, one reduces the first statement to the case that $n=2$. Now, 
		\begin{equation}
			\calI(\breve{O})w_1\calI(\breve{O})\cdot \calI(\breve{O})w_2\calG(\breve{O})\cong \calI(\breve{O})w_1w_2\calG(\breve{O})
		\end{equation}
		by the right $W_\bbf$-minimality and reducedness assumptions on $w_1w_2$: indeed, this follows from the theory of Tits systems, which we can apply by \cite[Section 6.5]{BT72} and \cite[Corollaire 4.6.7]{BT84}. This implies the claim when also applied over all formally unramified extensions of $\breve{O}$.
		The second statement follows from the first, and the third follows from the first and second as $\calI(\breve{O})\tau \calI(\breve{O})= \calI(\breve{O})\tau$ for $\tau\in \Omega_{\bba}$.
	\end{proof}
	
	For example, if $w=\tau s_1 \dots s_n $ is a right $W_\bbf$-minimal reduced word, with $s_i$ being simple reflections and $\tau$ stabilizing $\bba$, then
	$ \Fl_{\calI,\leq \tau}\tilde{\times}\Fl_{\calI,\leq s_1} \tilde{\times} \dots \tilde{\times} \Fl_{\calI,\leq s_n}  \to \Fl_{(\calI,\calG),\leq w}$ defines the Demazure resolution.
	Studying Demazure resolutions yields the following important geometric consequences.
	
	\begin{lemma}
		\label{sec:geometry-affine-flag-3-geometric-consequences-of-demazure-resolution}
		Let $w\in W/W_\bbf$ and denote by $w_{\mathrm{min}} \in W$ its right $W_\bbf$-minimal representative.
		\begin{enumerate}
			\item $\Fl_{(\breve{\calI},\breve{\calG}),w}\cong \bbA^{\ell(w_{\mathrm{min}}),\rm{pf}}_{k}$, in particular $\dim \Fl_{(\calI,\calG),w}=\ell(w_{\mathrm{min}})$.
			\item If $\ell(w_{\mathrm{min}})=0$, i.e., $w_{\mathrm{min}}=\tau\in \Omega_{\bba}$, then $\Fl_{(\calI,\calG), w}=\Fl_{(\calI,\calG),\leq w}$ is isomorphic to $\Spec(k)$.
			\item If $\ell(w_{\mathrm{min}})=1$, i.e., $w_{\mathrm{min}}=\tau s$ with $s$ a simple reflection and $\tau\in \Omega_{\bba}$, then
			\begin{equation}
				\Fl_{(\calI,\calG),\leq w}\cong \bbP^{1,\rm{pf}}_{k}.
			\end{equation}
		\end{enumerate}
	\end{lemma}
	\begin{proof}
		Using \Cref{sec:geometry-affine-flag-2-properties-of-convolution} one reduces to the case that $w=s$ is a simple reflection. In this case, one checks $\Fl_{\breve{\calI}, \leq s}\cong \bbP^{1,\rm{pf}}_{\overline{k}}$ by hand, cf.\ \cite[Section 1.4]{Zhu17}. The remaining assertions follow.
	\end{proof}
	
	Let us note that the morphism
	\begin{equation}
		(\mathrm{pr},m)\colon \Fl_{\calG}\tilde{\times} \Fl_{\calG}\to \Fl_{\calG} \times \Fl_{\calG},\ \overline{(g_1,g_2)}\mapsto (\overline{g_1}, \overline{g_1g_2}) 
	\end{equation}
	is an isomorphism, i.e., convolution products are secretly just products. Given now $L^+\calG$-stable locally closed perfect subschemes $X_1, X_2,Y\subseteq \Fl_{\calG}$ such that $m(X_1\tilde{\times} X_2)\subseteq Y$, we can factor $	m\colon X_1\tilde{\times} X_2\to Y$ as
	\begin{equation}
		X_1\tilde{\times} X_2\xrightarrow{(\mathrm{pr},m)} X_1\times Y\xrightarrow{\mathrm{projection}} Y.
	\end{equation}
	A similar discussion holds if we replace the $L^+\calG$-equivariant convolution by the $L^+\calI$-equivariant version. For the Iwahori convolutions in $\Fl_{\calG}$, we get the following important affineness statement.
	
	\begin{lemma}
		\label{sec:geometry-affine-flag-4-convolution-is-affine}
		Let $w\in W$ and $X\subset \Fl_{\calG}$ be a closed $L^+\calI$-stable perfect subscheme. Then the map
		\begin{equation}
			m\colon  \Fl_{\calI,w}\tilde{\times} X\to \Fl_{\calG}
		\end{equation}
		is affine. If $\calG=\calI$ is Iwahori, then the map
		\begin{equation}
			m\colon X \tilde{\times}\Fl_{\calI, w}\to \Fl_{\calI}
		\end{equation}
		is also affine. In particular, the inclusion $j_w\colon \Fl_{\calI, w}\to \Fl_{\calI}$ is affine for any $w\in W$.
	\end{lemma}
	More generally,	the affineness of $j_w \colon \Fl_{(\calI,\calG), w}\to \Fl_{\calG} $ for $w \in W/W_\bbf$ follows from \Cref{sec:geometry-affine-flag-3-geometric-consequences-of-demazure-resolution} because $\Fl_{(\calI,\calG),\leq w}\to \Fl_{\calG}$ is a closed immersion and $j_w\colon \Fl_{(\calI,\calG),w}\to \Fl_{(\calI,\calG), \leq w}$ is affine because the target is separated.
	\begin{proof}
		We may assume that $F=\breve{F}$. The proof of \cite[Lemma 4.1.6]{AR} applies, and we recall its argument. Using the above remarks, we can write $m\colon \Fl_{\calI, w}\tilde{\times} X\to \Fl_{\calG}$ as the composition
		\begin{equation}
			\Fl_{\calI, w}\tilde{\times} X\to \Fl_{\calI,w}\tilde{\times} \Fl_{\calG}\overset{(\mathrm{pr},m)}{\cong} \Fl_{\calI,w}\times \Fl_{\calG}\xrightarrow{\textrm{projection}} \Fl_{\calG}   
		\end{equation}
		of morphisms of ind-perfect schemes. The first morphism is affine as $X\subset \Fl_{\calG}$ is a closed immersion, and the second is affine by \Cref{sec:geometry-affine-flag-3-geometric-consequences-of-demazure-resolution}. 
		
		If $\calG$ is Iwahori, then the affineness of $m\colon X\tilde{\times}\Fl_{ \calI,w} \to \Fl_{ \calI}$ can be checked after passing to the quotient Hecke stacks
		\begin{equation}
			[L^+\calI \backslash \tilde{X} \times^{L^+\calI} \Fl_{ \calI,w}] \to \Hk_{\calI}
		\end{equation}
		where $\tilde{X}\subset LG$ is the natural $L^+\calI$-torsor over $X \subset \Fl_{ \calI}$.
		Forgetting about $L^+ \calI$-equivariance on the {\it right} and applying the inverse map of $LG$ and its subgroup $L^+\calI$, we arrive at the convolution map
		\begin{equation}
			Y\tilde{\times}\Fl_{ \calI,w^{-1}}\to \Fl_{ \calI}
		\end{equation}
		where $Y=\tilde{Y}/L^+\calI$ and $\tilde{Y}=\tilde{X}^{-1} \subset LG$. It suffices to prove this map is affine, but we already treated it in the first paragraph.
	\end{proof}
	
	We need the following result later on, but only at Iwahori level.
	
	\begin{lemma} \label{importantlemma}
		Let $X\subset \Fl_\calI$ be a locally closed $L^+\calI$-stable perfect subscheme of finite type. Then there exists a finite subset $S_X \subset W$ such that for any $w\in W$ we have
		\begin{equation}
			m(X\tilde{\times} \Fl_{\calI,w})\subset \cup_{x\in S_X}\Fl_{\calI,xw}, \quad m(\Fl_{\calI,w}\tilde{\times} X)\subset \cup_{x\in S_X} \Fl_{\calI,wx},
		\end{equation}
		where the union of certain $L^+\calI$-orbits stands for the unique locally closed perfect subscheme of $\Fl_{\calI}$ with those exact $L^+\calI$-orbits.
	\end{lemma}
	\begin{proof}
		The proof is via an induction argument on $\ell(w)$ which is similar to the equal characteristic setting, \cite[Lemma 4.4.2]{AR}. We sketch the proof here. Of course, it suffices to treat the case $X=\Fl_{\calI,w}$ for some $w\in W$. Let first $X=\Fl_{\calI,w}$ for some $w\in W$ with $\ell(w)=0$. Take $S_X=\{w\}$, then the statement holds by noting that $\ell(xw)=\ell(wx)=\ell(x)$ for any $x\in W$.

		Assume now $X=\Fl_{\calI,w}$ for some $w\in W$ with $\ell(w)>0$. Write $w=w_1s_1=s_2w_2$, where $\ell(w_1)=\ell(w_2)=\ell(w)-1$. The induction hypothesis implies that there exist finite subsets $S_{X_1}, S_{X_2}\subset W$ such that 
		\begin{equation}
			m(\Fl_{\calI,w_1}\tilde{\times} \Fl_{\calI,w'})\subset \cup_{x\in S_{X_1}} \Fl_{\calI,xw'}, \quad
			m(\Fl_{\calI,w'}\tilde{\times}\Fl_{\calI,w_2})\subset \cup_{x \in S_{X_2}}\Fl_{\calI,w'x},
		\end{equation}
		for any $w'\in W$.
		Note that for any $w''\in W$,
		\begin{equation}
			m(\Fl_{\calI,s_1}\tilde{\times}\Fl_{\calI,w''})\subset \Fl_{\calI,w''}\cup  \Fl_{\calI,s_1w''},\quad m(\Fl_{\calI,w''}\tilde{\times}\Fl_{\calI,s_2})\subset \Fl_{\calI,w''}\cup  \Fl_{\calI,w''s_2},
		\end{equation}
		and
		\begin{equation}
			\Fl_{\calI,w_1}\tilde{\times}\Fl_{\calI,s_1}\simeq \Fl_{\calI,s_2}\tilde{\times}\Fl_{\calI,w_2}\simeq \Fl_{\calI,w}
		\end{equation}
		by Lemma \ref{sec:geometry-affine-flag-2-properties-of-convolution}.
		We conclude that $S_X:=S_{X_1}\cup S_{X_1}s_1\cup S_{X_2}\cup s_2 S_{X_2}$ is the desired finite subset, thereby concluding the proof.
	\end{proof}
	
	\subsection{Constant terms and semi-infinite orbits}
	\label{sec:semi-infinite-orbits}
	
	Throughout this section, we assume $\calG=\calI$ is Iwahori. Let $U\subset B$ be the unipotent radical. Then we get the Iwasawa decompositions (\cite[Theorem 5.3.3]{KP23})
	\begin{equation}
		\label{eq:1-iwasawa-decomposition-for-f-breve}
		W=N(\breve{F})/\calT(\breve{O}) \simeq U(\breve{F})\backslash G(\breve{F})/\calI(\breve{O}),\ w\mapsto U(\breve{F})w\calI(\breve{O}).
	\end{equation}
	
	Geometrically, this yields the semi-infinite orbits.
	
	\begin{definition}
		\label{sec:constant-terms-semi-definition-semi-infinite-orbit}
		For $w\in W$ we set $\calS_w:=LU\cdot w\subset \Fl_{\calI}$.
	\end{definition}
	
	By \cite[Section 5]{AGLR22} the $\calS_w, w\in W,$ are represented by locally closed ind-(perfect schemes) and coincide with the connected components of the attractor $\Fl_{\calI}^+$ for a regular action by $\bbG_m$. More precisely, take a regular coweight $\chi\colon \bbG_{m,F}\to S$, i.e., such that the centralizer of $\chi$ in $G$ is $T$ (for example, the sum of all positive coroots). Then $B$ is the attractor locus 
	\begin{equation}
		G^+=\{g\in G\ |\ t\mapsto \chi(t)g\chi(t)^{-1} \textrm{ extends to } \bbA^1_F\}
	\end{equation}
	for the conjugation action of $\bbG_{m,F}$ on $G$. More generally, if $B\subseteq P\subseteq G$ is any parabolic subgroup, then there exists a character $\psi\colon \bbG_{m,F}\to S$ such that $P=G^+$ is the attractor locus for the $\bbG_{m,F}$-action on $G$ by conjugation. The centralizer $M$ of $\psi$ is then a Levi subgroup of $P$.
	
	The cocharacter $\chi$ extends to a group homomorphism $\bbG_{m,O} \to \calS$ by the universal property of connected Néron models, cf.\ \cite[Proposition 8.4]{KP23}. By conjugation, we deduce a $L^+\bbG_{m,O}$-action on $\Fl_{\calI}$ and we restrict it to $\bbG_{m,k}$ along the Teichmüller lift map. We get the decomposition
	\begin{equation}
		\Fl_{\calI}^+=\coprod\limits_{w\in W} \calS_w,\text{ where } \calS_w:=LU\cdot w\subset \Fl_{\calI},
	\end{equation}
	of the attractor locus, cf.\ \cite[Section 5]{AGLR22}. Similarly, the repeller locus $\Fl_{\calI}^-$, i.e., the attractor locus for the inverted $\Gm$-action, decomposes as
	\begin{equation}\label{equation_repellers}
		\Fl_{\calI}^-=\coprod\limits_{w\in W} \calS^-_w,\text{ where } \calS^-_w:=LU^-\cdot w\subset \Fl_{\calI},
	\end{equation}
	where $U^-$ denotes the unipotent radical of the opposite Borel $B^-$ of $B$. The semi-infinite orbits are relevant for computing constant term functors.
	Let
	\begin{equation}
		i^+\colon \Fl_{\calI}^+\to \Fl_{\calI},\ i^-\colon \Fl_{\calI}^-\to \Fl_{\calI}
	\end{equation}
	be the inclusions (=disjoint union of  locally closed immersions).
	Let $\Fl_{\calI}^0\subseteq \Fl_{\calI}$ be the fixed point locus of $\Gm$ and let
	\begin{equation}
		q^+\colon \Fl_{\calI}^+\to \Fl_{\calI}^0,\ q^-\colon \Fl_{\calI}^-\to \Fl_{\calI}^0
	\end{equation}
	be the natural morphism given by evaluating at $0\in \bbA^1_{k}$ resp.\ $\infty\in \bbA^1_{k}$.
	
	\begin{remark}
		\label{sec:constant-terms-semi-fixed-point-locus-not-equal-to-affine-flag-variety-for-levi}
		The natural inclusion $\Fl_{\mathcal{T}}\to \Fl_{\calI}^0$ induced by the $\Gm$-equivariant morphism $\calT\to \calI$, with $\Gm$ acting trivially on $\mathcal{T}$, is not an isomorphism. Namely, $\Fl_{\calT}$ is the perfect constant $k$-scheme associated with the set $\bar{\bbX}_\bullet$, while $\Fl_{\calI}^0$ is associated with the set $W$. 
	\end{remark}
	
	Following \cite[Section 6.3]{AGLR22} we can now define the constant term functor (associated with $B$ and $\calI$).
	
	\begin{definition}
		\label{sec:constant-terms-semi-definition-constant-term}
		We set
		\begin{equation}
			\mathrm{CT}_B:=Rq^+_!\circ i^{+,\ast}\colon \scrD_\et(\Hk_{\calI})\to \scrD_\et(L^+\calT\backslash\Fl_{ \calI}^0)
		\end{equation}
	\end{definition}
	
	By Braden's theorem the natural map
	$Rq^+_!\circ i^{+,\ast}\to Rq^-_\ast\circ Ri^{-,!}$
	is an isomorphism, see \cite[Theorem IV.6.5]{FS21} and \cite[Section 6]{AGLR22} in the analytic setting. We also refer the reader to \cite[Theorem B]{Ric19a} in the scheme-theoretic setting.  This implies excellent formal properties of the constant term functor.
	
	\begin{remark}
		\label{sec:constant-terms-semi-fibers-of-constant-terms}
		Let $A\in \scrD_\et(\Hk_{\calI})$ and $w\in W=\Hk_{\calI}(k)$. By proper base change the fiber of $\mathrm{CT}_B(A)$ over $w$ is calculated by $R\Gamma_c(\calS_w, A_{|\calS_w})$.
	\end{remark}
	
	To use the formula in \Cref{sec:constant-terms-semi-fibers-of-constant-terms} we establish the following lemma.
	
	\begin{lemma}\label{lem_schubert_cell_in_semi_infinite_orbit}
		Let $w\in W$ be such that $w(b)-b \in \bar{\bbX}_\bullet^+$, where $b$ denotes the barycenter of the standard alcove $\bba$.
		\begin{enumerate}
			\item $\Fl_{\calI,w}=L^+\calU\cdot w$, where $\calU\subseteq \calI$ denotes the scheme-theoretic closure of $U$ in $\calI$. 
			\item $\Fl_{\calI,w}=\calS_{w}\cap \Fl_{\calI, \leq w}$.
		\end{enumerate}
		
	\end{lemma}
	
	\begin{proof}
		The first claim is equation (5.11) in the proof of \cite[Lemma 5.3]{AGLR22}. Indeed, the root groups $L^+_k\mathcal{U}_{\alpha_i}$ of {\it loc.~cit.~} are contained in the $L^+\mathcal{U}$ here by the positivity condition imposed upon $w(b)-b$, matching the one in {\it loc.~cit.~}. The first claim implies $\Fl_{\calI,w}\subseteq \calS_w\cap \Fl_{\calI,\leq w}$ as $\calS_w=LU\cdot w$. Let $x\in \Fl_{\calI,\leq w}\setminus \Fl_{\calI,w}$. By (perfect) properness of $\Fl_{\calI,\leq w}$ the orbit map $\gamma\colon \bbG_{m}^\mathrm{pf}\to \Fl_{\calI,\leq w},\ t\mapsto \chi(t)x$ extends to a $\bbG_m^\mathrm{pf}$-equivariant map $\tilde{\gamma}\colon \bbA^{1,\mathrm{pf}}_{k}\to \Fl_{\calI,\leq w}$. As $\Fl_{\calI,w}$ is open in $\Fl_{\calI,\leq w}$ and $W\subseteq \Fl_{\calI}$ is exactly the set of $\bbG_{m}^\mathrm{pf}$-fixed points, \Cref{sec:geometry-affine-flag-1-stratification-of-affine-flag-variety} implies that $x\in \calS_{w^\prime}$ for $w^\prime:=\tilde{\gamma}(0)<w$ because $\calS_{w^\prime}$ is exactly the subscheme of points contracting to $w^\prime$ under the $\bbG_{m}^\mathrm{pf}$-action.
	\end{proof}
	
	Next we describe the closure relations for the stratification of $\Fl_{\calI}$ by the $\calS_w, w\in W$.
	As this is a geometric question, we may assume $F=\breve{F}$ for this. Then we have to define the dominant cocharacters $\bar{\bbX}_{\bullet}^+$ in $\bar{\bbX}_\bullet$. Recall that $\bbX_\bullet(S)^+$ denotes the ($B$-)dominant cocharacters for $S$. As we assumed $F=\breve{F}$, we get that
	\begin{equation}
		\bbX_\bullet(S)_\bbQ\cong \bar{\bbX}_{\bullet,\bbQ},
	\end{equation}
	and thus we can define $\bar{\bbX}_\bullet^+$ as the preimage of $\bbX_\bullet(S)_\bbQ^+$ under the map $\bar{\bbX}_\bullet\to \bar{\bbX}_{\bullet,\bbQ}$. 	Given $\bar{\bbX}_\bullet$ we can now define the semi-infinite Bruhat order $\preceq$ on $W$, which depends on $B$. Namely, set $w\preceq w^\prime$ if and only if for the Bruhat order $t_{n\bar{\nu}}w\leq t_{n\bar{\nu}}w^\prime$ for all $\bar{\nu}\in \bar{\bbX}_\bullet^+$ and $n\gg 0$.
	
	\begin{lemma}
		\label{closure-relations-of-semi-infinite-orbits}
		For $w,w^\prime\in W$ we have $\calS_w\subseteq \overline{\calS_{w^\prime}}$ if and only if $w\preceq w^\prime$.
	\end{lemma}
	\begin{proof}
		This is \cite[Proposition 5.4]{AGLR22}, where $\preceq$ is denoted by $\leq^{\frac{\infty}{2}}$.
	\end{proof}
	
	If $w=t_{\bar{\mu}}, w^\prime=t_{\bar{\nu}}$, then $w\preceq w^\prime$ if and only if $\bar{\nu}-\bar{\mu}$ lies in $\bar{\bbX}_\bullet^+$. 
	In particular, on $\bar{\bbX}_{\bullet}^+$ the two orders $\leq, \preceq$ agree. We will constantly use the equality $\ell(t_{\bar\nu})=\langle 2\bar\rho,\bar\nu\rangle$ for $\bar\nu$ dominant, see \cite[Lemma 9.1]{Zhu14}.

	\section{Cohomology of the affine flag variety}
	\label{sec:cohom-affine-flag}
	
	In this section, we want to study cohomology of constructible sheaves on the Hecke stack $\Hk_\calI$. In particular, we will
	\begin{enumerate}
		\item introduce Wakimoto-filtered complexes in mixed characteristic, following \cite{AB09}, \cite{AR} and \cite{Zhu14},
		\item calculate the constant terms of Wakimoto sheaves,
		\item show that central objects for convolution are Wakimoto-filtered.
	\end{enumerate}
	In this section we assume that $F=\breve{F}$, and thus in particular that $k$ is algebraically closed.
	
	\subsection{(Co)standard functors}
	\label{sec:costandard-functors}
	
	The considerations that are going to come have an easy shadow on Grothendieck groups. Recall that we always assume our ring of coefficients $\La$ to be a field.
	
	For $w\in W/W_\bbf$ let $j_w\colon \Fl_{(\calI,\calG),w}\to \Fl_{\calG}$ be the locally closed affine immersion. Note that $j_w$ is $L^+\calI$-equivariant, and hence descends to a morphism
	\begin{equation}\label{sheaf_standard}
		j_w\colon \Hk_{(\calI,\calG),w}\to \Hk_{(\calI,\calG)} 
	\end{equation}
	of stacks, where on the left side $\Hk_{(\calI,\calG),w}:=[L^+\calI\backslash \Fl_{(\calI,\calG),w}]$, and that we will usually denote in the same way.	
	Define the standard object
	\begin{equation}\label{sheaf_standard}
		\Delta_{(\calI,\calG),w}:=j_{w,!}(\Lambda)[\ell(w_{\mathrm{min}})] \in \scrD_{\mathrm{\acute{e}t}}(\Hk_{(\calI,\calG)})
	\end{equation}
	and the costandard object
	\begin{equation}\label{sheaf_costand}
		\nab_{(\calI,\calG),w}:=Rj_{w,\ast}(\Lambda)[\ell(w_{\mathrm{min}})] \in \scrD_{\mathrm{\acute{e}t}}(\Hk_{(\calI,\calG)})
	\end{equation}
	associated with $w\in W/W_\bbf$. When $\calI=\calG$, we use abbreviations $\Delta_{\calI,w}$ (resp.~ $\nab_{\calI,w}$) for $\Delta_{(\calI,\calI),w}$ (resp.~$\nab_{(\calI,\calI),w}$).
	
	Let
	\begin{equation}
		\scrD_{\mathrm{cons}}(\Hk_{(\calI,\calG)})\subset \scrD_{\et}(\Hk_{(\calI,\calG)})
	\end{equation}
	be the full subcategory consisting of objects with perfect stalks and bounded support (the last condition is a running assumption, as explained when we introduced the main notation of the paper).
	Let	$K_0(\Hk_{(\calI,\calG)})$ be the Grothendieck group of $\scrD_{\mathrm{cons}}(\Hk_{(\calI,\calG)})$. Since points in $\Fl_{ \calG}$ have connected stabilizers under the $L^+\calI$-action, $K_0(\Hk_{(\calI,\calG)})$ identifies with the Grothendieck group of the category of $L^+\calI$-equivariant sheaves on $\Fl_{\calG}$.
	Consequently, $K_0(\Hk_{(\calI,\calG)})$ is a free abelian group on the classes of the intersection complexes $\IC_{(\calI,\calG),w}:=j_{w,!\ast}(\Lambda)$ of $\Fl_{(\calI,\calG),\leq w}$.
	
	Via convolution $K_0(\Hk_{\calI})$ is naturally a ring, cf.\ \cite[Section 5.2]{AR}, and $K_0(\Hk_{(\calI,\calG)})$ is a left $K_0(\Hk_{\calI})$-module. In fact, this ring identifies with the integral group ring of $W$ as we recall now.
	
	\begin{lemma}
		\label{sec:k_0-i-equivariant-1-description-of-k-0-of-i-equivariant-sheaves}
		The maps
		\begin{equation}
			\theta\colon K_0(\Hk_{\calI})\to \bbZ[W],\ [\calF]\mapsto \sum\limits_{w\in W}(-1)^{\ell(w)} \chi(\Fl_{\calI,w}, j_w^\ast \mathcal{F})w  
		\end{equation}
		and
		\begin{equation}
			\vartheta\colon \bbZ[W]\to K_0(\Hk_{\calI}),\ w\mapsto (-1)^{\ell(w)}[\nab_{\calI,w}] 
		\end{equation}
		are inverse ring isomorphisms.
	\end{lemma}
	Here, $\chi(\Fl_{\calI,w},j_w^*\mathcal{F})$ denotes the Euler characteristic. In the equal characteristic case, a proof is given in \cite[Lemma 5.2.1]{AR}.
	\begin{proof}
		If $s\in W$ and $\ell(s)=1$, then by \Cref{sec:geometry-affine-flag-3-geometric-consequences-of-demazure-resolution}, we have $\Fl_{\calI, \leq s}\cong \bbP^{1,\rm{pf}}_{k}$ and we get distinguished triangles
		$$
		\IC_{\calI,e}\rightarrow \del_{\calI,s}\rightarrow\IC_{\calI,s}\xrightarrow{+1},\quad \IC_{\calI,s}\rightarrow \nab_{\calI,s}\rightarrow \IC_{\calI,e}\xrightarrow{+1}
		$$
		because $\IC_{\calI,s}$ identifies with the underived pushforward $R^0j_{s,\ast}(\Lambda)$.
		Thus, $\theta([\nab_{\calI,s}])=\theta([\del_{\calI,s}])=s$ and $\theta(\vartheta(s))=s$. It then follows from \Cref{djfejkjd;kjkel;} below that $\theta, \vartheta$ are ring homomorphisms and in fact $\theta([\del_{\calI,w}])=\theta([\nab_{\calI,w}])=w$ for any $w\in W$. Now, $\vartheta$ is surjective because the $[\del_{\calI,w}]=[\nab_{\calI,w}]$ generate $K_0(\Hk_{\calI})$. This finishes the proof.
	\end{proof}

	We will now study convolutions of standard and costandard sheaves. Before proceeding, we upgrade these objects to actual functors.
	Recall that $\scrD_{\mathrm{\acute{e}t}}(\Hk_\calI)$ has a natural monoidal structure in the sense of \cite[Definition 4.1.1.10]{Lur17} as we will see later on, see \Cref{sec:conv_fusion}. For now, it suffices to construct the underlying bifunctor as follows. First, consider the diagram
	\begin{equation}
		\begin{tikzcd}
			& {\Hk_{\calI}}\times {\Hk_{\calI}} \arrow[dl,"\mathrm{pr}_1"] \arrow[dr,"\mathrm{pr}_2",swap] &  {[L^+\calI\backslash LG\times^{L^+\calI} LG/L^+\calI]} \arrow[l,"p",swap] \arrow[r,"m"] & {\Hk_{\calI}}\\ 
			{\Hk_{\calI}} & & {\Hk_{\calI}}
		\end{tikzcd}
	\end{equation}
	of ind-$v$-stacks on perfect schemes, with $\rm{pr}_1,\rm{pr}_2$ the two projections onto the first and second factors, respectively, $p$ the natural  morphism, and $m$ the (quotient by the left $L^+\calI$-action of the) convolution morphism discussed in \Cref{sec:schub-vari-conv}. Now, for any $\mathcal{F}_1,\calF_2\in \scrD_\et(\Hk_{\calI})$,  define
	\begin{equation}
		\calF_1\ast \calF_2:=Rm_!(p^*(\calF_1{\boxtimes}^\mathbb{L}\calF_2)).
	\end{equation}
	Note that $m$ is (perfectly) ind-proper and thus $Rm_!=Rm_\ast$. The full subcategory $\scrD_{\mathrm{cons}}(\Hk_{\calI})$ is stable under convolution. A similar diagram can be used to define a convolution product $\ast$ that realizes $\scrD_{\mathrm{cons}}(\Hk_{(\calI,\calG)})$ as a left module of $\scrD_{\mathrm{cons}}(\Hk_{\calI})$.
	
	Note that the full subcategory of $\scrD_{\mathrm{cons}}(\Hk_{(\calI,\calG)})$ whose objects are supported at the origin identifies with $\scrD_{\mathrm{cons}}([*/L^+\calI])$. Convolution restricts to the usual tensor product on these sheaves which is symmetric monoidal. Furthermore, note that $L^+\calI$ is an extension of its reductive quotient, which is naturally isomorphic to the special fiber $S_k$ of the connected Néron model $\calS$ of $S$, by a connected pro-unipotent group. By \cite[Proposition VI.4.1]{FS21}, we can identify $\scrD_{\mathrm{cons}}([*/L^+\calI])$ with $\scrD_{\mathrm{cons}}([*/S_k])$ via pullback along $[\ast/L^+\calI]\to [\ast/S_k]$. This is convenient because we also have identification
	\begin{equation}\label{equation_identification}
		\scrD_{\mathrm{cons}}(\Hk_{ (\calI,\calG),w})\simeq    \scrD_{\mathrm{cons}}([*/S_k])     
	\end{equation}
	for any $w \in W/W_\bbf$, since $S_k$ maps isomorphically to the reductive quotient of the stabilizer group $L^+\calI \cap wL^+\calG w^{-1}$. In order to explain this assertion, we argue as follows: by the proof of \cite[Corollary 1.3]{Ric16}, the stabilizer identifies with the positive loop group $L^+\calG_{\bba \cup w\bbf}$ of the Bruhat--Tits $O$-group attached to the bounded subset $\Omega_w:=\bba \cup w\bbf$ of the apartment. For the reader's convenience, this relies on the combinatorial \cite[Lemma 7.7.4]{KP23} at the level of geometric points together with a Lie algebra calculation. The reductive quotient of $L^+\calG_{\Omega_w}$ identifies with that of its special fiber $\calG_{\Omega_w,k}$ and now \cite[Proposition 8.4.8]{KP23} tells us that its root system with respect to the image of $S_k$ is empty (as non-zero affine functionals cannot vanish on an entire alcove). In particular, the reductive quotient of the $w$-stabilizer identifies canonically with $S_k$.
	
	Moreover, the above abstract nonsense allows us to regard standard and costandard objects as functors by tensoring. Indeed, we identify $\scrD_{\mathrm{cons}}([\ast/S_k])$ with sheaves on the stratum $\Hk_{(\calI,\calI),1}$ and define  using convolution the \textit{standard} and \textit{costandard} functors:
	\begin{equation}
		\del_{(\calI,\calG),w}\colon \scrD_{\mathrm{cons}}([*/S_k])\to \scrD_{\mathrm{cons}}(\Hk_{(\calI,\calG)}),\ M \mapsto M\ast \del_{(\calI,\calG),w} ,
	\end{equation}
	\begin{equation}\nab_{(\calI,\calG),w} \colon \scrD_{\mathrm{cons}}([*/S_k])\to \scrD_{\mathrm{cons}}(\Hk_{(\calI,\calG)}),\ M \mapsto M\ast \nab_{(\calI,\calG),w} ,
	\end{equation}
	and one checks easily that there are isomorphisms $\del_{(\calI,\calG),w}(M)\simeq j_{w!}M[\ell(w)]$ and $\nab_{(\calI,\calG),w}(M)\simeq Rj_{w*}M[\ell(w)]$ of functors.  For $\calI=\calG$, we use abbreviations $\Delta_{\calI,w}$ (resp.~ $\nab_{\calI,w}$) for $\Delta_{(\calI,\calI),w}$ (resp.~$\nab_{(\calI,\calI),w}$).
	
	Since $j_w$ is an affine morphism by \Cref{sec:geometry-affine-flag-4-convolution-is-affine}, both functors are t-exact by \cite[Corollaire 4.1.3]{BBDG18} for the natural t-structure on $\scrD_{\mathrm{cons}}([\ast/S_k])$ and the perverse t-structure on $\scrD_{\mathrm{cons}}(\Hk_{(\calI,\calG)})$, cf.\ \cite[Definition 6.8]{AGLR22} for the latter.

	We start with the following lemma on the convolution of standard and costandard objects. 
	
	\begin{lemma} \label{djfejkjd;kjkel;}
		For any $w_1\in W$ and $w_2 \in W/W_\bbf$ such that $\ell(w_1)+\ell(w_{2,\mathrm{min}})=\ell((w_1w_2)_{\mathrm{min}})$, there exist canonical isomorphisms
		\begin{align*}
			&  \del_{\calI,w_1}*\del_{(\calI,\calG),w_2}\simeq \del_{(\calI,\calG),w_1w_2}\\
			& \nab_{\calI,w_1}*\nab_{(\calI,\calG),w_2}\simeq \nab_{(\calI,\calG),w_1w_2},
		\end{align*}
		satisfying the obvious associativity constraint.
	\end{lemma}
	
	\begin{proof}
		In equicharacteristic, this statement can be found in \cite[Lemma 8(a)]{AB09} with a proof given in \cite[Lemma 4.1.4 (1),(2)]{AR}. The same proof applies here. For any $w_1 \in W$ and $w_2\in W$ such that $w_1w_2$ is right $W_\bbf$-minimal and reduced, the convolution morphism $\Fl_{\calI, w_1}\tilde{\times}\Fl_{(\calI,\calG), w_2}\rightarrow \Fl_{(\calI,\calG), w_1w_2}$ is an isomorphism by \Cref{sec:geometry-affine-flag-2-properties-of-convolution}, so the constant complex $m_!\La$ identifies with $\underline{\La}$. This yields the desired isomorphisms after $!$- or $*$-extension and shifts. Indeed, $Rm_!(j_{w_1,!}\La\tilde{\boxtimes} j_{w_2,!}\La)\cong j_{w_1w_2,!}\La$ and $Rm_\ast(Rj_{w_1,\ast}\La\tilde{\boxtimes} Rj_{w_2,\ast}\La)\cong Rj_{w_1w_2,\ast}\La$. 
	\end{proof}
	
	Recall that in a monoidal category, an object is called left-invertible (resp. right-invertible) if multiplication on the left (resp. right) is an equivalence.
	
	\begin{lemma}\label{lem_invertible_conv_stand}
		For any $w\in W$, the objects $\nab_{\calI,w}$ and $\del_{\calI,w}$ are both left- and right-invertible in the monoidal category $\scrD_{\mathrm{cons}}(\Hk_\calI)$. More concretely, there exist canonical isomorphisms
		$$
		\nab_{\calI,w}*\del_{\calI,w^{-1}}\cong \delta_{\calI,e} \cong \del_{\calI,w}*\nab_{\calI,w^{-1}}.
		$$
		where $\delta_{\calI,e} : =\del_{\calI,e} = \nab_{\calI,e} $ with $e \in W$ is the identity element and $\delta_{\calI,e}$ is the unit object for convolution. 
	\end{lemma}
	
	\begin{proof}
		It is clear that $\delta_{\calI,e}$ is the unit for convolution.
		The existence of the desired isomorphisms is stated over a Laurent series field in \cite[Lemma 8(a)]{AB09} and proved in \cite[Lemma 4.1.4(3)]{AR}. Again the same proof applies. The desired canonical isomorphisms can be obtained by induction on $\ell(w)$ using \Cref{djfejkjd;kjkel;} provided we construct them for all simple reflections $s$ and $\tau\in \Omega_{\bba}$. The case for $\tau\in \Omega_\bba$ is clear as $\Fl_{\calI,\tau}=\Fl_{\calI,\leq \tau}$ by \Cref{sec:geometry-affine-flag-3-geometric-consequences-of-demazure-resolution}.
		
		Before discussing the the case for simple reflections, we construct a canonical isomorphism $(\Delta_{\calI,w}\ast\nab_{\calI,w^{-1}})_{eI}\cong \Lambda$ which gives rise to the desired canonical isomorphisms after proving that $\Delta_{\calI,w}$ is invertible.   Note that we have the following commutative diagram
		$$
		\begin{tikzcd}[row sep=huge]
			\Fl_\calI\tilde{\times}\Fl_\calI \arrow[r,"\mathrm{pr}_1\times m"]\arrow [d,"m",swap] & \Fl_\calI\times\Fl_\calI \arrow[dl,"\mathrm{pr}_2"]\\
			\Fl_\calI.
		\end{tikzcd}
		$$   
		The above commutative diagram implies that $m^{-1}(zI)$ can  be  identified  with $\Fl_\calI$ via $[g,g^{-1}zI]\mapsto gI $. Thus for any $y\in W$, we have isomorphisms
		\begin{align}  \label{equation_proof_lemma_3.3}
			\begin{split}        (\Delta_{\calI,w}\ast\nabla_{\calI,y})_{zI} & \simeq R\Gamma_c(m^{-1}(zI),\Delta_{\calI,w}\tilde{\boxtimes}\nabla_{\calI,y}) \\
				& \simeq R\Gamma_c(\Fl_{\calI},\Delta_{\calI,w}\otimes^{\bbL}j_*\Lambda[\ell(y)])\\
				&\simeq R\Gamma_c(\Fl_{\calI,w},j_{w}^*j_*\Lambda)[\ell(x)+\ell(y)].
			\end{split}
		\end{align}  
		The first isomorphism is obtained by proper base change. The map $j$ in the second isomorphism is the embedding $\{gI\in\Fl_\calI\vert zI\in gI yI\}\hookrightarrow \Fl_\calI$. The third isomorphism is derived from the projection formula. Let $y=w^{-1}$ and $z=e$, (\ref{equation_proof_lemma_3.3}) yields canonical isomorphisms
		\begin{align}\label{equation_canonical_isom}
			(\Delta_{\calI,w}\ast\nabla_{\calI,w^{-1}})_{eI}\cong \Lambda
		\end{align}
		since $\Fl_{\calI,w}$ is the perfection of an affine space.
		
		Let $s:=w=y=z\in\bbS$, we have an isomorphism $\Fl_{\calI,s}\simeq \bbP^{1,\mathrm{pf}}_k$ with $eI$ identified with $0\in\bbP^{1,\mathrm{pf}}$. Thus, arguments in \cite[Lemma 4.1.4.(3)]{AR}   apply here and we conclude the desired canonical isomorphisms for $w\in\bbS$. By the above discussion, we conclude \textit{non-canonical} isomorphisms $\nab_{\calI,w}*\del_{\calI,w^{-1}}\cong \delta_{\calI,e} \cong \del_{\calI,w}*\nab_{\calI,w^{-1}}$ for any $w\in W$. The canonical isomorphism (\ref{equation_canonical_isom}) provides us the desired canonical isomorphisms.
		
	\end{proof}
	
	\begin{remark}\label{isoms-are-canonical}
		Similar to the equicharacteristic setting as noted by \cite[Remark 4.1.5]{AR}, \Cref{Lemma_standatnd_convo_costand} induces more isomorphism of convolutions of standard sheaves and costandard sheaves. For example, there is a canonical isomorphism
		$$
		\Delta_{\calI,w}\ast \nabla_{\calI,v}\cong \Delta_{\calI,wv}
		$$
		if $\ell(w)-\ell(v)=\ell(wv)$.
	\end{remark}
	
	As usual one is interested in understanding what happens on the abelian subcategory $\scrP(\Hk_\calI)$ arising as the heart of the perverse t-structure. While it is not stable under the monoidal structure of $\scrD_{\mathrm{cons}}(\Hk_\calI)$, we can still benefit from the semiperversity below. We formulate it also for general parahorics $\calG$, because it plays a key role in the perversity of central sheaves.
	
	\begin{lemma}\label{Lemma_standatnd_convo_costand}
		\label{fn:1-convolution-of-standard-costandard-objects-perverse}
		For any $w \in W$, left convolution with $\Delta_{\calI,w} $ (resp.~$\nabla_{\calI,w} $) defines a left (resp.~right) exact endofunctor of $\scrD_\mathrm{cons}(\Hk_{(\calI,\calG)})$. If $\calG=\calI$ is Iwahori, the same holds for right convolution. In particular, for any other $v\in W$, we have $\del_{\calI,w}*\nab_{\calI,v}, \nab_{\calI,w}*\del_{\calI,v}\in \scrP(\Hk_\calI)$.
	\end{lemma}
	\begin{proof}
		Our proof parallels \cite[Lemma 4.1.7]{AR}. Consider the perverse $t$-structure 
		$$
		({}^p\scrD_{\mathrm{cons}}^{\leq 0}(\Hk_{(\calI,\calG)}),{}^p\scrD_{\mathrm{cons}}^{\geq 0}(\Hk_{(\calI,\calG)}) )
		$$ of $\scrD_{\mathrm{cons}}(\Hk_{(\calI,\calG)})$.
		Note that by definition ${}^p\scrD_{\mathrm{cons}}^{\leq 0}(\Hk_{(\calI,\calG)}) $ (resp.~${}^p\scrD_{\mathrm{cons}}^{\geq 0}(\Hk_{(\calI,\calG)}) $) is spanned by the non-negative (resp.~non-positive) shifts of the $\Delta_{(\calI,\calG),v}$ (resp.~$\nabla_{(\calI,\calG),v}$) for $v \in W/W_\bbf$. By \Cref{sec:geometry-affine-flag-4-convolution-is-affine}, the convolution map $m\colon \Fl_{\calI,w}\tilde{\times}\Fl_{(\calI,\calG),\leq v}\rightarrow \Fl_\calI$ is affine and proper. We have 
		\begin{equation}\label{equation_conv_stand_costand_is_perverse}
			\del_{\calI,w}\ast \nab_{(\calI,\calG),v}=Rm_!(\La[\ell(w)]\tilde{\boxtimes} \nab_{(\calI,\calG),v}).
		\end{equation}
		We refer to \cite[(2.2.2)]{AR} for the definition of this convolution and  \cite[Lemma 4.1.7]{AR} for elaboration on (\ref{equation_conv_stand_costand_is_perverse}). The sheaf $\La[\ell(w)]\tilde{\boxtimes} \nab_{(\calI,\calG),v}$ is perverse by our assumption and \Cref{sec:geometry-affine-flag-4-convolution-is-affine}. Thus $\del_{\calI,w}\ast \nab_{(\calI,\calG),v}$ is concentrated in non-negative perverse degrees because $!$-pushforward of affine morphisms is left exact for the perverse $t$-structure, cf.\ \cite[Corollaire 4.1.2]{BBDG18}.
		On the other hand, by a similar consideration for (\ref{equation_conv_stand_costand_is_perverse}),
		\begin{equation}
			\nab_{\calI,w}\ast \del_{(\calI,\calG),v}=Rm_\ast(\La[(\ell(w))]\tilde{\boxtimes}\Delta_{(\calI,\calG),v}),
		\end{equation}
		and is concentrated in non-positive perverse degrees by \cite[Th\'eorème 4.1.1]{BBDG18}. If $\calG=\calI$ is Iwahori, then by symmetry we can run the same arguments for the right convolution. This finishes the proof.
	\end{proof}
	
	During the remainder of this section, we will no longer need the general parahoric case. So we assume that $\calG=\calI$ is Iwahori and suppress it from the index of the standard and costandard sheaves.
	
	\begin{lemma} \label{wersd;kfdl;}
		For any $w_1,w_2\in W$, the perverse sheaves $\del_{w_1}*\nab_{w_2}, \nab_{w_1}*\del_{w_2}$ are both supported on $\Fl_{\calI,\leq w_1w_2}$, and restrict to $\Lambda[\ell(w_1w_2)]$ on $\Fl_{\calI,w_1w_2}$.
	\end{lemma}
	
	\begin{proof}
		The proof is similar to the equal characteristic case, cf.\ \cite[Lemma 4.1.10]{AR}, and we sketch it here. 
		The Euler characteristic
		\begin{equation}\theta\colon K_0(\Hk_\calI)\longrightarrow \mathbb{Z}[W],\quad [\mathcal{F}]\longmapsto \sum_{w\in W}(-1)^{\ell(w)} \chi(\Fl_{ \calI,w},j_w^*\mathcal{F})w
		\end{equation}
		defines a ring homomorphism. By the proof of \Cref{sec:k_0-i-equivariant-1-description-of-k-0-of-i-equivariant-sheaves} we know that $\theta([\del_w])=\theta([\nab_w])=w$ for any $w\in W$. Now let $w\in W$ be any element such that $\Fl_{\calI,w}$ is open in the support of $\del_{w_1}*\nab_{w_2}$. Then the coefficient of $w$ in $\theta([\del_{w_1}\ast \nab_{w_2}])\in \bbZ[W]$ does not vanish. By perversity, see \Cref{fn:1-convolution-of-standard-costandard-objects-perverse}, and $\calI$-equivariance it is a non-zero multiple of the Euler characteristic of the cohomology of $\Lambda[\ell(w)]$ on $\Fl_{\calI,w }\cong \bbA^{\ell(w),\mathrm{pf}}_k$.
		Now,
		\begin{equation}
			\theta([\del_{w_1}\ast \nab_{w_2}])=\theta([\del_{w_1}])\theta([\nab_{w_2}])=w_1w_2,
		\end{equation}
		and thus $w=w_1w_2$ and $\del_{w_1}\ast \nab_{w_2}$ is supported on $\Fl_{\calI, \leq w_1w_2}$. By perversity and $I$-equivariance, we must have that
		\begin{equation}
			j_{w_1w_2}^*(\del_{w_1}\ast \nab_{w_2})\cong \Lambda^{\oplus m}[\ell(w_1w_2)]
		\end{equation}
		for some $m\geq 1$. But the coefficient of $w_1w_2$ is $1$, so $m=1$ as desired. The statement for $\nab_{w_1}*\del_{w_2}$ follows similarly.
	\end{proof}
	
	If $\mathcal{F}\in \scrD_\et(\Hk_{\calI})$, then
	\begin{equation}
		\on{supp}(\mathcal{F})=\{w\in W\ |\ j^\ast_w\mathcal{F}\neq 0\}
	\end{equation}
	is the support of $\mathcal{F}$, and
	\begin{equation}
		\on{cosupp}(\mathcal{F})=\{w\in W\ |\ j^!_w\mathcal{F}\neq 0\}
	\end{equation}
	its cosupport.
	
	We need the following geometric consequence of  \Cref{importantlemma}. In equal characteristic this is \cite[Proposition 4.4.4.]{AR}.
	
	\begin{proposition} \label{prop_support_conv_standard}
		For any $\calF\in \scrD_{\mathrm{cons}}(\Hk_\calI)$, there exists a finite subset $A_\calF\subset W$ such that for any $w\in W$,
		\begin{enumerate}
			\item $ \on{supp}(\del_w*\calF) \subseteq w\cdot A_{\calF}$, and dually $\on{cosupp}(\nab_w*\calF)\subseteq w\cdot A_{\calF}$,
			\item $ \on{supp}(\calF*\del_w)\subseteq A_{\calF}\cdot w$, and dually $\on{cosupp}(\calF\ast \nab_w)\subseteq A_\calF\cdot w$.
		\end{enumerate}
	\end{proposition}
	\begin{proof}
		Let $X\subseteq \Fl_{\calI}$ be a closed finite union of $I$-orbits such that $\on{supp}(\calF)\subseteq X$. Set $A_\calF:=S_X$ with $S_X$ as in \Cref{importantlemma}, i.e.,
		\begin{equation}
			m(X\tilde{\times} \Fl_{\calI, w})\subseteq A_\calF\cdot w
		\end{equation}
		and
		\begin{equation}
			m(\Fl_{\calI, w}\tilde{\times} X)\subseteq w\cdot A_\calF
		\end{equation}
		for all $w\in W$. Now the proper base change theorem implies that $\on{supp}(\calF\ast \del_w)\subseteq A_\calF\cdot w$ and $\on{supp}(\del_w\ast \calF)\subseteq w\cdot A_\calF$ for any $w\in W$. Because we assumed that $X$ is closed, we can also use that $Rm_\ast$ commutes with $!$-restrictions (by the adjoint version of the proper base change theorem) to see that $\on{cosupp}(\calF\ast \nab_w)\subseteq A_\calF\cdot w$ and $\on{cosupp}(\nab_w\ast \calF)\subseteq w\cdot A_\calF$ for any $w\in W$. This finishes the proof.
	\end{proof}
	
	Regarding the products $\del_{w_1}\ast \del_{w_2},\ \nab_{w_1}\ast \nab_{w_2}$ for $w_1,w_2\in W$ we note the following. 
	
	\begin{lemma}\label{lemma_convolution_stds_and_costds}
		Let $w_1,w_2\in W$, then
		\begin{itemize}
			\item [(1)] $\Delta_{w_1}*\Delta_{w_2}$ lies in the smallest full subcategory of $\scrD_{\mathrm{cons}}(\Hk_{\calI})$, which is closed under extensions, and contains $\Delta_w[n]$ for $w\in W$ and $n\in\mathbb{Z}_{\leq 0}$.
			\item[(2)]$\nabla_{w_1}*\nabla_{w_2}$ lies in the smallest full subcategory of $\scrD_{\mathrm{cons}}(\Hk_{\calI})$, which is closed under extensions, and contains $\nab_w[n]$ for $w\in W$ and $n\in\mathbb{Z}_{\geq 0}$.
		\end{itemize}

	\end{lemma}
	\begin{proof}
		Given the results of this section the argument of \cite[Lemma 6.5.8]{AR} applies.
	\end{proof}
	
	Let $\Omega_{\bba} \subset W$ be the stabilizer of the fundamental alcove $\bba$, i.e., the subset of length $0$ elements.
	
	\begin{lemma}\label{lem_socle_top_stand_costand}
		Given $w \in W$, let $\tau \in \Omega_{\bba}$ be the unique element contained in $\Fl_{ \calI,\leq w}$. Then, the sheaf $\mathrm{IC}_\tau$ appears with multiplicity $1$ in the Jordan--Hölder series of $\nabla_w$ and equals its top. Dually, $\mathrm{IC}_\tau$ appears with multiplicity $1$ inside $\Delta_w$ as its socle.
	\end{lemma}
	
	\begin{proof}
		This follows from the same proof of \cite[Lemma 2.1]{BBM04}. The idea is to argue by induction on the length of $w$. Besides the combinatorics of Coxeter groups, one only has to know that $\Fl_{ \calI} \to \Fl_{ \calJ_s}$ is a $\bbP_k^{1,\mathrm{pf}}$-bundle locally for the étale topology that actually splits over Schubert cells (use root groups to see this latter property). Here, $\calI \to \calJ_s$ is the minimal parahoric fixing the wall of the fundamental alcove $\bba$ fixed by $s$.
	\end{proof}
	
	\subsection{Wakimoto sheaves}
	\label{sec:wakim-funct-calj}
	
	Let $w_1,w_2\in W$. In general,
	\begin{equation}
		\del_{w_1}\ast \del_{w_2}\ncong \del_{w_1w_2},\ \nab_{w_1}\ast \nab_{w_2}\ncong \nab_{w_1w_2}
	\end{equation}
	unless $\ell(w_1)+\ell(w_2)=\ell(w_1w_2)$, cf.\ \Cref{djfejkjd;kjkel;}. In this subsection, we want to remedy this fact by introducing objects
	\begin{equation}
		\scrJ_{\bar{\nu}}\in \scrD_{\et}(\Hk_{\calI})
	\end{equation}
	for $\bar{\nu}\in \bar{\bbX}_\bullet$ (recall the embedding $\bar{\bbX}_\bullet\to W,\ \bar{\nu}\mapsto t_{\bar{\nu}}$) such that
	\begin{equation}
		\scrJ_{\bar{\nu}_1}\ast \scrJ_{\bar{\nu}_2}\cong \scrJ_{\bar\nu_1+\bar\nu_2}
	\end{equation}
	for all $\bar\nu_1,\bar\nu_2\in \bar{\bbX}_\bullet$, and
	\begin{equation}
		\scrJ_{\bar{\nu}}\cong \del_{t_{\bar{\nu}}}
	\end{equation}
	if $\bar{\nu}\in -\bar{\bbX}_\bullet^+$. Note that by \Cref{lem_invertible_conv_stand} this already forces
	\begin{equation}
		\scrJ_{\bar{\nu}}\cong \nab_{t_{\bar{\nu}}}
	\end{equation}
	if $\bar{\nu}\in \bar{\bbX}_\bullet^+$. In fact, we must have
	\begin{equation}
		\scrJ_{\bar{\nu}}\cong \del_{t_{\bar\nu_2}}\ast \nab_{t_{\bar\nu_1}}
	\end{equation}
	if we write $\bar{\nu}=\bar\nu_1-\bar\nu_2$ with $\bar\nu_1,\bar\nu_2\in \bar{\bbX}_\bullet^+$ (which is always possible).
	Note that $\ell(t_{\bar\nu_1})+\ell(t_{\bar\nu_2})=\ell(t_{\bar\nu_1}t_{\bar{\nu}})$ if $\bar\nu_1,\bar\nu_2\in \bar{\bbX}_\bullet^+$. Hence, \Cref{djfejkjd;kjkel;} implies that the above formula for $\scrJ_{\bar{\nu}}$ is independent (up to isomorphism) of $\bar\nu_1,\bar\nu_2$.
	To get a more canonical construction of $\scrJ_{\bar{\nu}}$, we will adopt the definition from \cite[Section 4.2.1]{AR}.

	\begin{definition}
		\label{sec:wakim-funct-calj-1-definition-wakimoto-functor-for-nu}
		Let $\bar \nu\in \bar \bbX_\bullet$. The {\it Wakimoto sheaf }  $\scrJ_{\bar\nu} $ is the object in $\scrD_{\mathrm{cons}}(\Hk_\calI)$ corepresenting the functor
		\begin{equation}
			\calF \mapsto \on{colim}\Hom(\nab_{t_{\bar\nu_1}},\calF \ast \nab_{t_{\bar\nu_2}}),
		\end{equation}
		with the (filtered) colimit running over all pairs of $\bar{\nu}_1,\bar{\nu}_2 \in \bar \bbX_\bullet^+$ such that $\bar{\nu}=\bar{\nu}_1-\bar{\nu}_2$, where the order is given by $(\bar{\nu}_1,\bar{\nu}_2)\leq (\bar{\nu}_3,\bar{\nu}_4)$ if $\bar{\nu}_3-\bar{\nu}_1=\bar{\nu}_4-\bar{\nu}_2\in \bar \bbX_\bullet^+$. The transition morphisms in the colimit are given by convolution with $\nab_{t_{\bar{\nu}_3-\bar{\nu}_1}}=\nab_{t_{\bar{\nu_4}-\bar{\nu}_2}}$ (and using the canonical isomorphisms in  \Cref{djfejkjd;kjkel;}).
	\end{definition}
	
	Note that all the transition morphisms in the colimit are isomorphisms. In particular, we can conclude (by invertibility of $\nab_{t_{\bar{\nu_2}}}$, cf.\ \Cref{lem_invertible_conv_stand}) that
	\begin{equation}
		\scrJ_{\bar{\nu}}\cong \del_{t_{\bar\nu_2}}\ast \nab_{t_{\bar\nu_1}}
	\end{equation}
	as desired. More generally, we can use the fact that $\scrD_{\mathrm{cons}}([*/S_k])$ acts on $\scrD_{\mathrm{cons}}(\Hk_{\calI})$ to construct a functor 
	\begin{equation}
		\scrJ_{\bar\nu} \colon \scrD_{\mathrm{cons}}([*/S_k]) \to \scrD_{\mathrm{cons}}(\Hk_{\calI}), \, M \mapsto \scrJ_{\bar{\nu}} \ast M
	\end{equation}
	between the two categories via evaluation at the Wakimoto sheaf. This will be called the Wakimoto functor and still be denoted by $\scrJ_{\bar \nu}$ by abuse of notation.
	
	\begin{remark}
		\label{sec:wakim-funct-calj-1-remark-on-wakimoto-sheaves}
		The Wakimoto sheaves $\scrJ_{\bar\nu}$ were introduced by Mirkovi\'c for geometrizing Bernstein elements in the affine Hecke algebra, see \cite[Section 5.1]{AR}. 
	\end{remark}
	
	Given a subset $\Omega \subset \bar \bbX_\bullet$, it will also be convenient to define the $\Omega$-Wakimoto functor
	\begin{equation}
		\scrJ_\Omega=\oplus_{\bar \nu\in \Omega} \scrJ_{\bar \nu} \colon \scrD_{\mathrm{cons}}([\underline{\Omega}/S_k]) \to \scrD_{\mathrm{cons}}(\Hk_\calI)
	\end{equation}as the direct sum of the $\scrJ_{\bar\nu}$ for $\bar \nu \in \Omega$, where $\underline{\Omega}=\bigsqcup_{\Omega} \Spec\, k$ regarded as an ind-scheme, so that complexes of étale sheaves have compact support. If $\Omega=\bar \bbX_\bullet$ is the total set, then we simply write $\scrJ$ for $\scrJ_{\bar \bbX_\bullet}$, which is monoidal by \Cref{djfejkjd;kjkel;}. Indeed, we can identify $\scrD_{\mathrm{cons}}([\underline{\bar \bbX_\bullet}/S_k])$ with the full subcategory of compact objects of the product taken in $\mathrm{Cat}_\infty$ of the monoidal $1$-category $\bar \bbX_{\bullet}$ with the stable $\infty$-category $\scrD_{\mathrm{cons}}([*/S_k])$. We see that the first category maps monoidally to the abelian category $\scrP(\Hk_{\calI})$ via the Wakimoto sheaves $\scrJ_{\bar{\nu}}$, see \cite[Section 4.2.3]{AR}, whereas the second maps monoidally to the $\bbE_1$-center of $\scrD_{\mathrm{cons}}(\Hk_{\calI})$. This implies the claim by the universal property of centers, see \cite[Proposition 5.3.1.8]{Lur17}.

	\begin{lemma} \label{lem_wak_basic_prop}
		The Wakimoto functors satisfy the following properties:
		\begin{itemize}
			\item [(1)] For any $\bar \nu\in\bar\bbX_{\bullet}$, $\scrJ_{\bar \nu}$ is t-exact for the perverse t-structure.
			\item [(2)] For any $\bar \nu\in\bar\bbX_{\bullet}$, $\scrJ_{\bar \nu}$ is supported on $\Fl_{\leq \bar \nu}$ and $j_{\bar \nu}^*\scrJ_{\bar \nu}\simeq \La[\langle 2\bar\rho, \bar\nu\rangle]$.
			\item[(3)] For any $\bar\mu,\bar\nu\in\bar \bbX_{\bullet}$, there exists a canonical isomorphism $\scrJ_{\bar\mu} *\scrJ_{\bar\nu}\simeq \scrJ_{\bar\mu+\bar\nu}$.
			\item [(4)] For any $\bar\mu,\bar\nu\in\bar \bbX_{\bullet}$ with $t_{\bar \nu}\npreceq t_{\bar \mu}$, we have $R\mathrm{Hom}_{\scrD_\et(\Hk_{\calI})}(\scrJ_{\bar\mu},\scrJ_{\bar\nu})= 0$. 
		\end{itemize}
	\end{lemma}
	
	\begin{proof}
		The first statement follows from \Cref{fn:1-convolution-of-standard-costandard-objects-perverse}, the second from \Cref{wersd;kfdl;} and the third is implicit in the discussion of monoidality of $\scrJ$.
		Let us discuss the fourth statement. Using (3), the invertibity of $\scrJ_{\bar \nu}(\Lambda)$ and the definition of the semi-infinite Bruhat order $\preceq$ reduces by suitable convolution to the case that $\bar\nu,\bar\mu\in \bar{\bbX}_\bullet$ are dominant. Then, we have identities $\scrJ_{\bar\mu}(M)=\nab_{t_{\bar\mu}}(M)$ and $\scrJ_{\bar\nu}(N)=\nab_{t_{\bar\nu}}(N)$.
		Thus by  \Cref{sec:geometry-affine-flag-1-stratification-of-affine-flag-variety}
		\begin{equation}
			R\Hom_{\scrD_\et(\Hk_{\calI})}(\nab_{t_{\bar\mu}}(M),\nab_{t_{\bar\nu}}(N))\cong R\Hom_{\scrD_\et(\Hk_{\calI})}(j^\ast_{t_{\bar\nu}}\nab_{t_{\bar\mu}}(M),N[\ell(t_{\bar\nu})])=0 
		\end{equation}
		if $t_{\bar\nu}\nleq t_{\bar\mu}$ (for the Bruhat order $\leq$ or equivalently the semi-infinite order $\preceq$ as $\bar\nu,\bar\mu\in \bar{\bbX}_\bullet^+$). 
	\end{proof}
	
	For a stable $\infty$-category $\scrD$ and a set of objects $S \subset \mathrm{Ob}(\scrD)$, let $\langle S\rangle$ be the smallest full subcategory of $\scrD$ whose objects include $S$ and which is stable under cones and shifts.
	
	\begin{definition}
		Define the \textit{Wakimoto category} as the full subcategory 
		$$
		\on{Wak}:=\langle \scrJ \rangle \subset\scrD_{\mathrm{\acute{e}t}}(\Hk_\calI)
		$$ 
		generated by the essential image of $\scrJ$ under cones and shifts. An object $\calF \in \on{Ob}(\on{Wak})$ is called \textit{Wakimoto filtered}. Define the full subcategory 
		\[
		\scrP(\mathrm{Wak}) \subset \scrP(\Hk_\calI)
		\]
		consisting of those perverse sheaves that admit a filtration by perverse sheaves with grading in the essential image of $\scrJ$.
	\end{definition}

	\begin{remark}
		By  \Cref{lem_wak_basic_prop} the category $\on{Wak}\subseteq \scrD_\et(\Hk_{\calI})$ is stable under convolution. As it contains $\delta_e\cong \scrJ_{\bar\nu}$ it is thus itself monoidal. Our definition of $\scrP(\mathrm{Wak})$ agrees with the one in \cite[\S4.3.2]{AR} by taking $\Lambda=\bar{\bbX}_\bullet^+$. In the works \cite{AB09,AR}, their respective authors do not define the full subcategory 
		$\mathrm{Wak}\subset \scrD_{\mathrm{\acute{e}t}}(\Hk_\calI)$, but rather the full subcategory 
		$
		\scrP(\mathrm{Wak}) \subset \scrP(\Hk_\calI)
		$.
		Morally, one can try to think of $\scrP(\mathrm{Wak})$ as the heart of $\mathrm{Wak}$, but it is not an abelian category, only exact, and it is not true that every perverse sheaf that is Wakimoto filtered as a complex actually lies in $\scrP(\mathrm{Wak})$. Indeed, pick $\nu$ dominant with respect to $B$ and let $\Delta_0 \subset \Delta_{-\nu}$ be the socle by \cite[Lemma 2.1]{BBM04}. Then, the cokernel lies in $\mathrm{Wak}$, but its $0$-th graded piece equals $\Delta_0[1]$, which is not perverse.
	\end{remark}
	
	We give the following simple criterion for determing whether an object of $\scrD_{\mathrm{\acute{e}t}}(\Hk_\calI)$ lies in $\mathrm{Wak}$.
	
	\begin{proposition}\label{prop_wak_category_central}
		Let $\calF\in \scrD_{\mathrm{cons}}(\Hk_\calI)$. Then, the following are equivalent:
		\begin{enumerate}
			\item $\calF$ is Wakimoto filtered;
			\item $\on{supp}(\scrJ_{-\bar \nu} \ast \calF) \subset \{ t_{-\bar \mu} \colon \bar\mu\in\bbX_\bullet^+\}$ for all $\bar \nu \ll 0$;
			\item $\on{cosupp}(\scrJ_{\bar \nu} \ast \calF) \subset \{ t_{\bar \mu} \colon \bar\mu\in\bbX_\bullet^+\}$ for all $\bar\nu \gg 0$.
		\end{enumerate}
		In particular, if $\calF$ satisfies $\scrJ_{\bar\nu}\ast \calF\cong \calF\ast \scrJ_{\bar\nu}$ for all $\nu\in \bar\bbX_\bullet$, then it is Wakimoto filtered.
	\end{proposition}
	Here, the notation $\bar{\nu}\gg 0$ means that $\langle \bar \nu, \alpha \rangle \gg 0$ for all $B$-positive relative roots $\alpha$ of $G$, while $\bar\nu\ll 0$ means $-\bar\nu\gg 0$.
	\begin{proof}
		Assume that $\mathcal{F}$ is Wakimoto filtered, and let us check that it satisfies (2) and (3). We may then assume that $\mathcal{F}=\scrJ_{\bar\nu^\prime}$ for some $\nu^\prime\in \bar\bbX_\bullet$. If now $\nu\gg 0$, then
		\begin{equation}
			\scrJ_{-\bar\nu}\ast \mathcal{F}\cong \del_{t_{-\bar\nu+\nu^\prime}}
		\end{equation}
		and the support claim follows. Similarly, we can argue for (3).
		Let us now assume that $\calF$ satisfies (2). We want to show that $\calF$ is Wakimoto filtered. Replacing $\mathcal{F}$ by $\scrJ_{-\bar\nu}\ast \mathcal{F}$ for some suitable $\nu\gg 0$, we may assume that
		\begin{equation}
			\on{supp}(\mathcal{F})\subseteq \{t_{-\bar\mu} : \mu\in \bar\bbX_\bullet^+\}.
		\end{equation}
		Noting (\ref{equation_identification}), it is then formal that $\mathcal{F}$ lies in $\langle \del_{t_{-\bar\mu}} : \mu\in \bar\bbX_\bullet^+\rangle$, cf.\ \cite[Lemma 4.4.3]{AR} (we use here that the trivial representation $\Lambda\in \scrD_{\mathrm{cons}}([\ast/S_k])=\scrD_{\mathrm{cons}}(\mathrm{Hk}_{\mathcal{I},w})$ is a generator because $S_k$ is connected). But
		\begin{equation}
			\langle \del_{t_{-\bar\mu}} : \mu\in \bar\bbX_\bullet^+\rangle\subseteq \rm{Wak}
		\end{equation}
		by the construction of Wakimoto sheaves. The argument that (3) implies (1) is similar.
		
		For the last claim, let $A_{\calF}\subseteq W$ be as in \Cref{prop_support_conv_standard}, i.e.,
		\begin{equation}
			\on{supp}(\del_w\ast \mathcal{F})\subseteq w\cdot A_\calF,\quad \on{supp}(\calF\ast \del_w)\subseteq A_\calF\cdot w
		\end{equation}
		for all $w\in W$.
		As $\scrJ_{-\bar\nu}\ast \mathcal{F}\cong \calF\ast \scrJ_{-\bar\nu}$ for $\bar\nu\in \bar\bbX_\bullet$, we can conclude that for $\nu\gg 0$
		\begin{equation}
			\on{supp}(\scrJ_{-\bar\nu} \ast \calF)\subset t_{-\bar\nu}A_\calF \cap A_\calF t_{-\bar\nu}.
		\end{equation}
		Now, we claim that for $\nu \gg 0$
		\begin{equation}
			t_{-\bar\nu}A_{\calF}\cap A_\calF t_{-\bar\nu}\subseteq \{t_{-\bar \mu} \colon \bar\mu\in\bar\bbX_\bullet^+\},
		\end{equation}
		which would finish the proof. To check the claim let us recall that
		\begin{equation}
			\label{eq:1}
			\bar\bbX^+_\bullet\cong W_{\rm{fin}}\backslash W/W_{\rm{fin}}.
		\end{equation}
		If now $w\in t_{-\bar\nu}A_\calF\cap A_\calF t_{-\bar\nu}$, then we can write $w=t_{-\bar\nu}w_1=w_2t_{-\bar\nu}$ for $w_1,w_2\in A_\calF\subseteq W$, i.e.,
		\begin{equation}\label{eq:translations_equal}
			t_{\bar\nu}=w_1^{-1}t_{\bar\nu}w_2.
		\end{equation}
		Evaluating this equality on $W_{\mathrm{fin}}$, we find some $w_0 \in W_{\mathrm{fin}}$ such that $w_0^{-1}w_1$ and $w_0^{-1}w_2$ are both translations. We now see that the two translations in \eqref{eq:translations_equal} are in different Weyl chambers for all $\nu\gg 0$ (by finiteness of $A_{\mathcal{F}}$, unless $w_0=1$. Thus, we can conclude that for $\bar \nu\gg 0$ we have $w_1,w_2\in \bar\bbX_\bullet$, i.e.,
		\begin{equation}
			t_{-\bar\nu}A_\calF\cap A_\calF t_{-\bar\nu}\subseteq \bar\bbX_{\bullet}.
		\end{equation}
		Again using that $A_\calF$ is finite, we can conclude that for $\nu \gg 0$ we get
		\begin{equation}
			t_{-\bar\nu}A_\calF\cap A_\calF t_{-\bar\nu}\subseteq \{t_{-\bar\mu} : \mu\in \bar\bbX^+_\bullet\}
		\end{equation}
		as desired. 
	\end{proof}
	
	\begin{remark}
		In \cite[Proposition 5]{AB09} and \cite[Proposition 4.4.1]{AR}, it is shown that a central perverse sheaf whose convolution functor is perverse t-exact lies in the category $\scrP(\mathrm{Wak})$. The  proof given in those references is considerably more complicated, because of the need to ensure that the graded sheaves are actually perverse. Our proof is much simpler due to taking place in the derived setting, and later we will see how to recover the extra degree information required for perversity for the essential image of the Gaitsgory's central functor $\scrZ$.
	\end{remark}
	
	For an arbitrary subset $\Omega \subset \bar\bbX_\bullet$, we can also define the full subcategory $\on{Wak}_{\Omega}=\langle\scrJ_\Omega\rangle$.
	
	\begin{proposition}\label{prop_wak_filtr}
		If $\Omega \subset \bar\bbX_\bullet$ is a lower poset (for $\preceq$), the inclusion $\on{Wak}_{\Omega} \to \on{Wak}$ has a right adjoint $\on{Wak} \to \on{Wak}_{ \Omega},\ \calF \mapsto \calF_{\Omega}$ such that the cone $\calG$ of the adjunction unit $\calF_{\Omega} \to \calF$ lies in $\mathrm{Wak}$ and satisfies $\calG_{\Omega}=0$. 
	\end{proposition}
	
	\begin{proof}
		Given $\calF \in \on{Ob}(\on{Wak})$, we show the existence of a final morphism $\calF_{\Omega} \to \calF$, such that $\calF_{ \Omega} \in \on{Ob}(\on{  Wak}_{\Omega})$. In other words, we show that for $\calF\in \on{Wak}$ the functor $R\Hom(-,\calF)$ on $\on{Wak}_{\Omega}$ is representable by some object $\calF_{\Omega}\in \on{Wak}_\Omega$. This assertion is stable under cones, and thus reduces to $\calF=J_{\bar \nu}(M)$ for some $M\in \calD_{\mathrm{cons}}([\ast/S_k])$. If $\bar \nu\notin \Omega$, the functor in question is zero by \Cref{lem_wak_basic_prop}. If $\nu\in \Omega$, then the functor is representable because $J_{\bar\nu}(M)\in \on{Wak}_{\Omega}$.
	\end{proof}
	
	If $\Omega$ equals $\{\bar \nu \preceq \bar\mu\}$ resp. $\{\bar \nu \prec \bar\mu\}$ for some $\bar\mu\in \bar\bbX_\bullet$, we simply write $\on{Wak}_{\preceq \bar\mu}$ resp. $\on{Wak}_{\prec \bar\mu}$, instead of $\on{Wak}_{\Omega}$. We can now define the Wakimoto graded pieces for $\calF \in \on{Wak}$.
	
	\begin{definition}
		\label{sec:wakimoto-sheaves-graded-wakimoto-functor}
		For any $\bar\mu\in\bar\bbX_\bullet$, we define the endofunctor 
		\begin{equation}
			\mathrm{gr}_{\bar\mu}\colon \on{Wak} \to \on{Wak}, \calF\mapsto\on{cone}(\calF_{\prec\bar\mu}\to \calF_{\preceq\bar\mu}).
		\end{equation}
		We also define $\mathrm{gr}:=\oplus_{\bar\mu\in\bar\bbX_{\bullet}}\mathrm{gr}_{\bar\mu}\colon \on{Wak}\to \on{Wak}$.
	\end{definition}
	
	By construction, we see that $\calF\in \on{Wak}$ lies in $\calP({\on{Wak}})$ if and only if $\mathrm{gr}(\calF)$ is perverse. 
	Note that $\on{gr}_{\bar \nu}(\calF)$ lies in the essential image of the functor $\scrJ_{\bar\nu}\colon \scrD_{\mathrm{cons}}([*/S_k])\to \on{Wak}$. In the next subsection, we will show that this functor is fully faithful by explicitly constructing an inverse via constant terms of the opposite Borel, see \Cref{queiowuoiwuroe}. In particular, we can essentially uniquely lift $\on{gr}_{\bar\nu}(\calF)$ to an element of $\scrD_{\mathrm{cons}}([*/S_k])$ and can make the following definition. 
	
	\begin{definition}
		Let $\mathcal{F} \in \on{Wak}$, we define
		\begin{equation}
			\on{Grad}_{\bar\nu}(\calF) \in \scrD_{\mathrm{cons}}([*/S_k])
		\end{equation}
		to be the canonical object such that $\scrJ_{\bar\nu}(\on{Grad}_{\bar\nu}(\calF))$ identifies with $\on{gr}_{\bar\nu}(\calF)$.
	\end{definition}
	
	\subsection{Cohomology of Wakimoto filtered objects}
	\label{sec:cohomology-wak}
	
	We now analyze the cohomology of objects in $\on{Wak}$. First, we show that convolution with Wakimoto sheaves induces a shift.
	
	\begin{proposition} \label{owieuroieq[p[pp}
		For any $\bar \nu\in\bar \bbX_\bullet$ and $\calF\in \scrD_{\mathrm{cons}}(\Hk_\calI)$, there is a canonical isomorphism
		\begin{equation}
			R\Gamma(\Fl_\calI,\calF\ast \scrJ_{\bar\nu})\simeq R\Gamma(\Fl_\calI,\calF)[\langle 2\rho, \nu\rangle].
		\end{equation} 
		
	\end{proposition}
	\begin{proof}
		Let us first assume that $\bar \nu\in \bar\bbX_\bullet^+$, which implies
		\begin{equation}
			\scrJ_{\bar\nu}(\La)=\nab_{t_{\bar\nu}}=Rj_{w,\ast}(\La)[\ell(t_{\bar\nu})].
		\end{equation}
		The map $(\pr_1,m)\colon \Fl_\calI\tilde{\times} \Fl_\calI\to \Fl_\calI\times \Fl_\calI$ is an isomorphism. The second projection $\Fl_\calI\tilde\times \Fl_\calI\to \Hk_{\calI}$ is transformed to the map
		\begin{equation}
			\pi\colon \Fl_\calI\times \Fl_\calI\to \Hk_{\calI},\ (\bar g, \bar h)\mapsto \overline{g^{-1}h}.
		\end{equation}
		We note that this map is equivariant for the diagonal action of $LG$ by multiplication on the source and the trivial action on the target.
		By definition $\calF\ast \scrJ_{{\bar{\nu}}}(\La)\cong R\pr_{2,\ast}(\pr_1^\ast(\calF)\otimes^L_\La \pi^\ast \nabla_{t_{\bar{\nu}}})$ and thus we get
		\begin{align*}
			& R\Gamma(\Fl_\calI, \calF\ast \scrJ_{{\bar{\nu}}}(\La)) \\
			\simeq & R\Gamma(\Fl_\calI\times \Fl_\calI, \pr_1^\ast(\calF)\otimes^L_\La \pi^\ast \nabla_{t_{\bar{\nu}}}) \\
			\simeq & R\Gamma(\Fl_\calI,\calF \otimes^L_\La R\pr_{1,\ast} \pi^\ast \nabla_{t_{\bar{\nu}}}) 
		\end{align*}
		Using that $R\Gamma(\Fl_{\calI,\leq t_{\bar\nu}}, \nabla_{t_{\bar\nu}})\cong \Lambda[\ell(t_{\bar\nu})]$ one checks that the pullback of $R\pr_{1,\ast}\pi^\ast\nabla_{t_{\bar\nu}}$ along $LG \to \Fl_{ \calI}$ is isomorphic to $\La[\ell(t_{\bar\nu})]$. As the object $R\pr_{1,\ast}\pi^\ast\nabla_{t_{\bar\nu}}\in \scrD_\et(\Fl_\calI)$ is $LG$-equivariant (because $\pi$ and $\pr_1$ are, where $LG$ acts trivially on $\on{Hk}_{\mathcal{I}}$) we can conclude that $R\pr_{1,\ast}(\pi^\ast\nabla_{t_{\bar\nu}})\cong \Lambda[\ell(t_{\bar\nu})]$. Moreover, we normalize this isomorphism such that over $1\cdot I\in \Fl_{\calI}$ it reduces to the canonical isomorphism $R\Gamma(\Fl_{\calI,\leq t_{\bar\nu}}, \nabla_{t_{\bar\nu}})\cong R\Gamma(\Fl_{\calI, t_{\bar\nu}},\Lambda[\ell(t_{\bar\nu})])\cong \Lambda[\ell(t_{\bar\nu})]$ (induced by adjunction). With this convention, the resulting isomorphism
		\begin{equation}
			R\Gamma(\Fl_\calI, \calF\ast\scrJ_{\bar\nu})\cong R\Gamma(\Fl_\calI, \calF)[\ell(t_{\bar\nu})]
		\end{equation}
		for $\bar\nu\in \bbX_\bullet^+$ is additive in $\bar\nu$. Thus, it can be extended to the desired natural isomorphism
		\begin{equation}
			R\Gamma(\Fl_\calI, \calF\ast\scrJ_{\bar\nu})\cong R\Gamma(\Fl_\calI, \calF)[\langle 2\bar\rho, \bar\nu\rangle],
		\end{equation}
		using \cite[Lemma 9.1.]{Zhu14} to see that $\ell(t_{\bar\nu})=\langle 2\rho, \nu\rangle$ if $\bar\nu$ is dominant.
	\end{proof}

	We immediately deduce the following two corollaries.
	
	\begin{corollary} \label{queiowuoiwuroe} There is a canonical isomorphism
		\begin{equation}
			R\Gamma(\Fl_\calI,\scrJ_{\bar{\nu}}(M))\simeq M[\langle 2\bar\rho, \bar\nu\rangle]
		\end{equation}
		for $M\in \scrD([*/S_k])$ and $\bar\nu\in \bar\bbX_\bullet$.
	\end{corollary}
	\begin{proof}
		This follows from  \Cref{owieuroieq[p[pp} by setting $\mathcal{F}=\scrJ_0(\La)=\delta_e$.
	\end{proof}

	
	\begin{corollary}\label{coro_Grad_composes_Zisom_H}
		For any $\calF\in\scrP(\mathrm{Wak})$, there is a canonical isomorphism
		\begin{equation}
			H^n(\Fl_\calI,\calF)\simeq \bigoplus_{\langle 2\bar\rho, \bar\nu\rangle=-n}\mathrm{Grad}_{\bar{\nu}}(\calF).
		\end{equation}
	\end{corollary}
	\begin{proof}
		This result is analogous to \cite[Proposition 4.5.4]{AR}: please note that the Weyl chamber in \textit{loc.cit} that is denoted by $\Lambda$ (not to be confused with our coefficients) corresponds to $\bar \bbX^+_\bullet$ in our setup. The existence of a canonical isomorphism follows from  \Cref{queiowuoiwuroe} by using the filtration of $\mathcal{F}$ by Wakimoto sheaves. Note that the associated graded of $\calF$ is perverse, so we conclude that, whenever $\bar\nu_1 \prec \bar \nu_2$, then the cohomology complexes $R\Gamma(\Fl_\calI, \mathrm{gr}_{\bar\nu_i}(\calF)) $ sit in different degrees with the same parity. This implies that the connecting homomorphisms of the associated long exact sequences vanish.
	\end{proof}
	
	We wish to determine the $\La$-module $\on{Grad}_{\bar{\mu}}(\calF)$ in a functorial manner. For this we calculate constant terms.

	\begin{proposition}\label{prop_wakimoto_constant_terms}
		For any $\calF\in \scrD_{\mathrm{\acute{e}t}}(\Hk_\calI)$, $w\in W$ and $\bar\nu\in \bbX_\bullet$, there is a canonical identification
		\begin{equation}\label{equation_weight_functor}
			\mathrm{CT}_{B^-}(\scrJ_{\bar \nu} \ast \calF)_{t_{\bar\nu}w} \simeq \mathrm{CT}_{B^-}(\calF)_w[\langle 2\rho, \nu \rangle]
		\end{equation}
		between stalks of constant term complexes. 
	\end{proposition}
	
	\begin{proof}
		Take $\bar\lambda\in\bar{\bbX}_\bullet^+$ dominant enough such that $\bar\lambda-\bar\nu$ is also dominant, then $\scrJ_{\bar \lambda}(\Lambda)=\nabla_{t_{\bar \lambda}}$ and $\calJ_{\bar\lambda-\bar\nu}=\nabla_{\bar\lambda-\bar\nu}$. Note it follows from \Cref{lem_wak_basic_prop}  that $\nabla_{\bar\lambda-\bar\nu}\ast \calJ_{\bar\nu}\ast \calF\simeq\nabla_{\bar\lambda}\ast \calF$. Thus it is enough to consider $\bar \nu \in \bar\bbX_\bullet^+$.  
		
		Our proof follows its analogue in equicharacteristic, see \cite[Lemma 4.5.8]{AR}, and we sketch it here. 
		Similar to (\ref{equation_conv_stand_costand_is_perverse}),
		$$\nab_{t_{\bar \nu}}(\La)\ast \calF \simeq Rm_* (\La \tilde{\boxtimes}  \calF) [\langle 2\rho, \nu\rangle ].$$
		By Braden's theorem, the left side of (\ref{equation_weight_functor}) naturally identifies with cohomology supported at the corresponding $LU$-orbit $\calS_{t_{\bar \nu}w}$. Since $\calS_{t_{\bar \nu}}$ contains $\Fl_{\calI,t_{\bar \nu}}$ by \Cref{lem_schubert_cell_in_semi_infinite_orbit}, the pullback of $\calS_{t_{\bar \nu}w}$ along the convolution map $\Fl_{\calI, t_{\bar \nu}} \tilde{\times} \Fl_\calI\rightarrow \Fl_\calI$ identifies with the twisted product $L^+\calU x_{\bar\nu}L^+\calU\times^{L^+\calU} \calS_w $. Here, $x_{\bar \nu} \in LT(k)$ is a lift of the translation $t_{\bar \nu}$ regarded as a $k$-valued point in $\mathrm{Gr}_{\calT}$ (if $G$ is split, such a lift can be obtained by evaluating $\nu$ on a uniformizer $\pi$ of $F$): consequently, $L^+\calU x_{\bar\nu}L^+\calU$ denotes the obvious $L^+\calU$-torsor over $\Fl_{\calI,t_{\bar \nu}}$ and our identification matches the one appearing right above \cite[Equation (4.5.11)]{AR}. 
		By abuse of notation, we denote this perfect scheme by $\Fl_{\calI,t_{\bar \nu}} \tilde{\times} \calS_w$, even though the twisted product is not for the Iwahori group.

		Let $i_w:\calS_w\rightarrow \Fl_\calI$ denote the locally closed emebdding and recall $\calS^-_{t_{\bar \nu}w}$  defined in (\ref{equation_repellers}). Amassing all this information, we get that
		\begin{align}
			\begin{split}
				R\Gamma_c(\calS^-_{t_{\bar \nu}w}, \scrJ_{\bar \nu}(\Lambda) \ast \calF) & \simeq R\Gamma(\calS_{t_{\bar\nu}w},Ri_w^!( Rm_* (\La \tilde{\boxtimes}  \calF [\langle 2\rho, \nu\rangle ])) \\         
				& \simeq R\Gamma(\Fl_{\calI,t_{\bar\nu}}\tilde{\times} \calS_w, \La \tilde{\boxtimes} Ri_w^!\calF) [\langle 2\rho, \nu\rangle ].
			\end{split}
		\end{align}
		Here, the first isomorphism is via Braden's theorem as discussed in the above. We apply proper base change to commute $Ri_w^!$ and $Rm_*$ and obtain the second theorem. Because $\Fl_{\calI,t_{\bar \nu}}$ is an orbit under the pro-unipotent group $L^+\calU$, see again \Cref{lem_schubert_cell_in_semi_infinite_orbit}, the twisted product does not alter the cohomology complex, thereby yielding the desired claim. The isomorphism obtained in this manner does not depend on the choice of $x_{\bar \nu}$, as any other lift would fit into a corresponding commutative diagram.
	\end{proof}
	
	\begin{corollary}\label{coro_wakimoto_graded_ct} 
		We have a canonical isomorphism $\on{Grad}_{\bar \nu}(\calF)[\langle 2\rho, \nu \rangle ] \simeq \mathrm{CT}_{B^-}( \calF)_{t_{\bar \nu}}$.
	\end{corollary}
	\begin{proof}
		This follows by induction on $\bar \nu$, by considering the filtration $\calF_{\preceq\bullet}$ and applying \Cref{prop_wakimoto_constant_terms}. Note that, in particular, this proposition computes $\mathrm{CT}_{B^-}(\scrJ_{\bar \nu})$ by writing it as the convolution with the identity $\delta_e$, whose constant terms are concentrated at $e$ with value equal to $\Lambda$.
	\end{proof}
	
	\begin{remark}
		The corollary above tells us when $\on{Grad}_{\bar \nu}(\calF)$ is perverse with some ease for $\calF \in \on{Wak}$. This corollary together with geometric Satake and constant terms is what will allow us to show that the central functor $\scrZ$ actually factors through $\scrP(\mathrm{Wak})$, thus bypassing the strategy of \cite[Theorem 4, Proposition 5]{AB09} and \cite[Proposition 4.4.1]{AR}. Indeed, if we know that $\mathrm{CT}_{B^-}(\calF)$ is perverse, then its associated graded is perverse, and we can write the Wakimoto complex $\calF$ as an extension of perverse Wakimoto sheaves, so it lies in $\scrP(\mathrm{Wak})$.
	\end{remark}
	
	We end this section by discussing the monoidal structure of the functor $\mathrm{Grad}:=\bigoplus_{\bar \nu\in \bar\bbX_{\bullet}}\mathrm{Grad}_{\bar \nu}$ restricted to $\scrP(\mathrm{Wak})$.\footnote{We however don't discuss an $\mathbb{E}_1$-monoidal structure of the functor $\mathrm{Grad}$ on $\mathrm{Wak}$.} 
	
	\begin{lemma}
		The full subcategories $\scrP(\mathrm{Wak}) \subset \mathrm{Wak}$ are stable under convolution.
	\end{lemma}
	
	\begin{proof}
		By induction on the number of non-zero graded pieces, we reduce to the case of the convolution of two Wakimoto complexes, but this is \Cref{lem_wak_basic_prop}.    
	\end{proof}
	
	\begin{proposition} \label{gr-convol}
		For any $\calF,
		\calG \in \scrP({\rm Wak})$, 
		and $\bar\nu_1,\bar\nu_2\in \mathbb{X}_\bullet$ with $\bar\nu:=\bar\nu_1+\bar\nu_2$, there is a canonical morphism 
		\begin{equation}
			\beta_{\bar\nu_1,\bar\nu_2}:\mathrm{gr}_{\bar\nu_1}(\mathcal{F})\ast\mathrm{gr}_{\bar\nu_2}(\calG)\rightarrow \mathrm{gr}_{\bar\nu}(\calF*\calG).
		\end{equation}
		such that $\bigoplus_{\bar\nu_1+\bar\nu_2=\bar\nu}\beta_{\bar\nu_1,\bar\nu_2}$ is an isomorphism. 
	\end{proposition}
	
	\begin{proof}
		The statement is proved by induction on the number of non-vanishing $\on{Grad}_{\bar \nu}$, similarly to the equicharacteristic case \cite[Proposition 6a)]{AB09} and \cite[Lemma 4.7.4, Proposition 4.7.5]{AR}. We remind the reader that the main tool used in those proofs is the existence for $\mathcal{F}$ and $\mathcal{G}$ of a canonical filtration in $\mathcal{P}(\mathrm{Hk}_{\mathcal{I}})$ indexed by Wakimoto sheaves. We have produced such a filtration at the derived level, i.e., for objects in the $\infty$-category $\mathrm{Wak}$ in \Cref{prop_wak_filtr}, and it passes to $\mathscr{P}(\mathrm{Wak})$ essentially by definition of this full subcategory.
	\end{proof}

	\begin{corollary}\label{prop_grad_monoidal}
		For any $\calF,
		\calG \in \scrP({\rm Wak})$, 
		and $\bar \nu_1,\bar \nu_2\in \bar\bbX_\bullet$ with $\bar\nu:=\bar\nu_1+\bar\nu_2$, there is a canonical morphism 
		\begin{equation}
			\alpha_{\bar\nu_1,\bar\nu_2}:\mathrm{Grad}_{\bar\nu_1}(\mathcal{F})\otimes_\La\mathrm{Grad}_{\bar\nu_2}(\calG)\rightarrow \mathrm{Grad}_{\bar\nu}(\calF*\calG).
		\end{equation}
		such that $\bigoplus_{\bar\nu_1+\bar\nu_2=\bar\nu}\alpha_{\bar\nu_1,\bar\nu_2}$ is an isomorphism. Consequently, for any such $\calF$ and $\calG$, there is a canonical isomorphism \begin{equation}\mathrm{Grad}(\calF)\otimes_\La \mathrm{Grad}(\calG)\cong \mathrm{Grad}(\calF*\calG)
		\end{equation}  inside the category $\scrD_{\mathrm{\acute{e}t}}(\underline{\bar\bbX_\bullet})$.
	\end{corollary}
	\begin{proof}
		This follows directly from  \Cref{gr-convol}  and  \Cref{lem_wak_basic_prop} (3). 
	\end{proof}

	\section{Central functor}
	
	\subsection{Background}
	We introduce the spaces that underlie the construction of the Gaitsgory's central functor $\scrZ$ in mixed characteristic. The spaces are not of classical nature, and live in the world of v-stacks created by Scholze \cite{Sch17,SW20}. We recall their basic properties, following \cite{FS21,AGLR22}.  
	
	We use the notation introduced in \Cref{section-notation}, but additionally assume that $G$ is residually split, which implies that each $L^+\calI$-orbit in $\Fl_\calG$ is already defined over $\Spec(k)$, cf.\ \Cref{sec:geometry-affine-flag-1-residually-split}.	
	\begin{definition}
		The \textit{Hecke stack} $\Hk_{\calG,O}$ is the v-stack sending a characteristic $p$ affinoid perfectoid space $\Spa(R,R^+)$ 
		to the groupoid of 
		
		$\bullet$
		untilts $\Spa(R^{\sharp},R^{\sharp,+})$  of  $\Spa(R,R^+)$ over $O$,
		
		$\bullet$
		$\mathcal{G}$-torsors $\mathcal{P}_1$ and $\mathcal{P}_2$ on $\Spec(B_{\on{dR}}^+(R^{\sharp}))$ together with an isomorphism  
		\begin{equation}
			\gamma: 
			\mathcal{P}_1|_{\text{Spec}(B_{\text{dR}}(R^{\sharp}))} \cong 
			\mathcal{P}_2|_{\text{Spec}(B_{\text{dR}}(R^{\sharp}))}.
		\end{equation}
		We refer to \cite[Section 20.3]{SW20} for the definition of the rings $B_{\on{dR}}^+(R^\sharp)$ and $B_{\mathrm{dR}}(R^\sharp)$.
	\end{definition}

	An alternative way to define $\Hk_{\calG,O}$ is as the v-stack quotient
	\begin{equation}
		\Hk_{\calG,O} = [L_{O}^+\mathcal{G} \backslash \Gr_{\calG,O}]
	\end{equation}
	where $L_{O}^+\mathcal{G}$ is the jet group over $O$, i.e., the v-group sheaf over $\Spd(O)$ with value
	$\mathcal{G}(B_{\on{dR}}^+(R^{\sharp}))$
	on untilts $\Spa(R^{\sharp},R^{\sharp,+})$ over $O$, and 
	\begin{equation}
		\Gr_{\calG,O} := L_{O}\mathcal{G} /L_{O}^+\mathcal{G}
	\end{equation}
	with $L_{O}\mathcal{G}$ the loop group over $O$, i.e., the v-group sheaf over Spd$(O)$ with value $\mathcal{G}(B_{\text{dR}}(R^{\sharp}))$ on 
	Spa$(R^{\sharp},R^{\sharp,+})$, cf.\  \cite[Lemma 4.10]{AGLR22}. 
	
	Note that over the generic fiber  $\eta = \Spd(F)$, $\Gr_{\calG,O}$ identifies with the $B_{\on{dR}}$-affine Grassmannian $\Gr_{G,F}$. On the other hand, over the special fiber $s=\Spd(k)$, $\Gr_{\calG,O}$  becomes isomorphic to $\Fl_{\calG}^\diamondsuit$, the analytification\footnote{For more discussion on the analytification functor $\diamondsuit$ and other variants, we refer the reader to \cite[Subsection 2.2]{AGLR22}.} of the Witt vector affine flag variety $\Fl_\calG$, as defined in \cite[Section 27]{Sch17}. 
	
	Now we pick a complete algebraically closed extension $C$ of $F$ with residue field $\bar k$, and let $\bar{\eta} = \Spd(C)$, $\bar{s}= \Spd(\bar k)$. 
	Consider the natural diagram
	\begin{equation}
		\Hk_{\calG,C} \overset{j}{\hookrightarrow} \Hk_{\calG,O_C}
		\overset{i}{\hookleftarrow} \Hk_{\calG,\bar k},
	\end{equation}
	where $j$ is the open immersion of the generic fiber and $i$ the closed immersion of the special fiber.
	This induces a nearby cycles functor, see \cite[Section 6.5]{AGLR22},
	\begin{equation}
		R\Psi := i^*Rj_* :  \scrD_{\mathrm{\acute{e}t}}(\Hk_{\calG, C}) \longrightarrow 
		\scrD_{\mathrm{\acute{e}t}}(\Hk_{\calG,\bar k})
	\end{equation}
	between the stable $\infty$-categories of derived étale sheaves in the sense of Scholze \cite[Definition 14.13, Lemma 17.1]{Sch17} with bounded support as in Fargues--Scholze \cite[Chapter VI]{FS21}. For $\ell$-adic coefficients, we follow the same conventions of passing to the limit as in \cite[Section 27]{Sch17} and then inverting $\ell$, compare with \cite[Section 6.5]{AGLR22}. 
	
	An important property of $R\Psi$ is that it preserves universal local acyclicity in the sense of \cite[Section IV.2.1]{FS21}, see also \cite[Section 6]{AGLR22} for our convention for non-torsion coefficients $\La$. Below, we denote by $\scrD_{\mathrm{ula}}(X/S) \subset \scrD_{\mathrm{\acute{e}t}}(X)$ the full subcategory of universally locally acyclic sheaves (or, if the base is understood, simply $\scrD_{\mathrm{ula}}(X)$).
	
	\begin{proposition}\label{prop_nearby_cycles_ula}
		Nearby cycles $R\Psi$ restrict to a functor
		\begin{equation}		\scrD_{\mathrm{ula}}(\Hk_{\calG,C}) \rightarrow 
			\scrD_{\mathrm{ula}}(\Hk_{\calG, \bar k}).
		\end{equation}
	\end{proposition}
	
	\begin{proof}
		This is \cite[Corollary 6.14]{AGLR22}.
	\end{proof}
	
	Recall that in the previous sections of the paper, we introduced a Hecke stack $\Hk_{\calG}^{\mathrm{sch}}$ as a perfect $k$-stack. Its associated v-sheaf under the analytification functor $\diamondsuit$ of \cite[Section 27]{Sch17} is the fiber over $\mathrm{Spd} k$ of the analytic Hecke stack $\Hk_{\calG}^{\mathrm{an}}$ that we defined over $\mathrm{Spd}O$. There is a natural comparison map of sheaves due to \cite[Section 27]{Sch17}
	
	\begin{proposition} \label{092840dijf}
		The natural comparison functor
		\begin{equation}
			c \colon \scrD_{\mathrm{\acute{e}t}}(\Hk_{\calG,\bar k}^{\mathrm{sch}} ) \to \scrD_{\mathrm{\acute{e}t}}(\Hk_{\calG, \bar k}^{\mathrm{an}})
		\end{equation}
		is an equivalence carrying $\scrD_{\mathrm{cons}}(\Hk_{\calG,\bar k}^{\mathrm{sch}} ) $ to $\scrD_{\mathrm{ula}}(\Hk_{\calG, \bar k}^{\mathrm{an}})$.
	\end{proposition}
	\begin{proof}
		For the definition of the comparison functor, we refer to \cite[Section 27]{Sch17} and \cite[Appendix A]{AGLR22}. The above result is \cite[Propositions 6.7 and A.5]{AGLR22}.  
	\end{proof}
	
	This result also highlights the importance of ula sheaves as singling out constructible sheaves over a base field.
	From now on, we will not make a stark distinction between $\Hk_{\calG}^{\mathrm{sch}}$ and $\Hk_{\calG,k}^{\mathrm{an}}$ and will simply omit the superscript when writing down its derived category of sheaves. 
	
	\begin{definition}
		\label{sec:background-definition-central-functor}
		We define the central functor (for the Witt vector affine flag variety) as the composition
		\begin{equation}
			\scrZ\colon \mathrm{Rep}(\hat{G}) {\xrightarrow{\sim}} \scrP_{\mathrm{ula}}(\Hk_{G,C}) {\xrightarrow{R\Psi}} \scrD_{\mathrm{ula}}(\Hk_{\calG, \bar k}) 
		\end{equation}
	\end{definition}
	\noindent Here, the first arrow comes from the geometric Satake equivalence of \cite[Chapter VI]{FS21}, with the Satake category consisting of ula perverse sheaves on $\Hk_{G,C}$. The second arrow is just nearby cycles which respect the ula property by \Cref{prop_nearby_cycles_ula}. 
	Often below, we will find it convenient to still abusively denote by $\scrZ$ the nearby cycles functor $R\Psi \colon \scrD_{\mathrm{ula}}(\Hk_{G,C})\to \scrD_{\mathrm{ula}}(\Hk_{\calG,\bar k})$. 
	
	\begin{remark}
		\label{sec:background-central-functor-with-galois-action} As explained in \cite[Section 8]{AGLR22} the nearby cycles functor is Galois equivariant. More precisely, given $A\in \scrP_{\mathrm{ula}}(\Hk_{G,E})$ for some finite extension $E/F$, then $R\Psi(A_C)$ has a natural $\Gamma_E$-action that is equivariant with respect to the residual action of $\Gamma_{k_E}$. Here, $\Gamma_E\subseteq \Gamma$ denotes the Galois group of $E$, and $\Gamma_{k_E}$ the one for the residue field $k_E$ of $E$.
	\end{remark}
	
	\subsection{Convolution and fusion}\label{sec:conv_fusion}
	
	In this section, we are going to discuss in detail the convolution and fusion products.
	
	\begin{definition}
		\label{sec:convolution-fusion-definition-convolution-hecke-stack}
		Given a finite linearly ordered set $J=\{j_1<\dots <j_n\}$, we define the convolution Hecke stack $\Hk_\calG^{J}$ to be the v-sheaf over $\Spd O$ which classifies successive modifications of $\calG$-torsors over $B_{\mathrm{dR}}^+$, indexed by the elements $j_i \in J$. More precisely, for a given $f\colon S\to \Spd O$ the groupoid $\Hk_\calG^J(S)$ is given by $\calG$-torsors $\calP_{j_1}, \ldots, \calP_{j_n}$ on $B^+_{\mathrm{dR}}(S)$ with modifications $\calP_{j_1}\dashrightarrow \calP_{j_2},\ldots, \calP_{j_{n-1}}\dashrightarrow\calP_{j_n}$ defined on $B_{\mathrm{dR}}(S)$.
	\end{definition}
	
	One often finds the expression $\Hk_\calG^{J}=\Hk_\calG^{j_1} \tilde{\times} \dots \tilde{\times} \Hk_\calG^{j_n}$ to denote the convolution Hecke stack. We have already seen that there is a natural correspondence with $n=2$ inducing the convolution product $\ast : \scrD_{\mathrm{\acute{e}t}}(\Hk_{\calG,S}) \times \scrD_{\mathrm{\acute{e}t}}(\Hk_{\calG,S}) \to \scrD_{\mathrm{\acute{e}t}}(\Hk_{\calG,S})$ for any $S\to \Spd O$, see \Cref{sec:costandard-functors}.  We wish to enhance this operation to a monoidal structure of $\infty$-categories in the sense of \cite[Definition 4.1.1.10]{Lur17}. This will be quite technical, and we recommend the unaccustomed reader to try to ignore the heavy language at first, and focus on the geometry at hand. After each categorical proof, we also provide an explanation of our constructions at the level of $1$-morphisms of correspondences, which should prove helpful.
	
	Let us recall some of the notions from \cite[Section 4.1]{Lur17}. First, we have the (nerve of) the 1-category 
	$\mathrm{Comm}^{\otimes}$ (also denoted by $\mathrm{Fin}_*$ or $\bbE^\otimes_\infty$ in \cite[Notation 2.0.0.2, Notation 5.1.1.6]{Lur17}) whose objects are finite pointed sets $\langle n \rangle=\{0,1,\dots, n\}$ with base point $0$ and whose morphisms $\langle n \rangle \to \langle m \rangle$ preserve $0$. A symmetric monoidal $\infty$-category $\scrC$ is given by a cocartesian fibration $\scrC^\otimes \to \mathrm{Comm}^\otimes$ of $\infty$-operads, see \cite[Example 2.1.2.18, Definition 2.1.1.10]{Lur17}, in particular, $\scrC^\otimes_{[n]} \simeq \scrC^n$ in a natural manner.
	Similarly, we have the $\infty$-operad $\mathrm{Assoc}^{\otimes}$ (which is equivalent to some other common $\infty$-operads denoted by $\bbA_\infty^\otimes$ or $\bbE_1^\otimes$ in \cite[Definition 4.1.1.3, Example 5.1.0.7]{Lur17}) given as the 1-category whose objects are pointed finite sets $\langle n \rangle$ and morphisms $\langle n \rangle \to \langle m \rangle$ are pointed maps equipped with a total order on the non-pointed fibers and composition is given by the lexicographical order, see \cite[Remark 4.1.1.4]{Lur17}. A monoidal $\infty$-category is a cocartesian fibration $\scrC^\otimes \to \mathrm{Assoc}^\otimes$ of $\infty$-operads and its underlying $\infty$-category $\scrC$ is the fiber of the fibration over $\langle 1\rangle$. When the cocartesian fibration is clear from the context, we will often abuse language and refer to $\scrC$ as a monoidal $\infty$-category. Such a datum induces by \cite[Propositions 2.4.1.7, 2.4.2.5]{Lur17} a map $\mathrm{Assoc}^\otimes \to \mathrm{Cat}_\infty^\times$ that preserves inert morphisms in the sense of \cite[Definitions 2.1.1.8, 2.1.2.3]{Lur17} and also cocartesian morphisms, in particular $\scrC^\otimes \to \mathrm{Assoc}^\otimes$ induces an associative algebra in $\mathrm{Cat}_{\infty}$. 
	Here, $\mathrm{Cat}^\times_\infty$ denotes $\mathrm{Cat}_\infty$ with its cartesian symmetric monoidal structure, cf.~\cite[Construction 2.4.1.4, Proposition 2.4.1.5]{Lur17}. Note that in general maps of $\infty$-operads are not necessarily monoidal, but rather only lax-monoidal, see \cite[Definition 2.1.3.7]{Lur17}.
	
	In order to produce the desired map that will induce a monoidal structure on $\scrD_{\mathrm{\acute{e}t}}(\Hk_{\calG,S})$, we recall that following \cite[Definition A.5.2]{Man22} we dispose of a symmetric monoidal $\infty$-category $\mathrm{Corr}(\mathrm{vSt})$ of correspondences on v-stacks. The 6-functor formalism defined in \cite{Sch17} for torsion coefficients can be reinterpreted as in \cite[Definition A.5.6]{Man22} thanks to \cite[Theorem 5.11]{Man22b} via an operadic map \begin{equation}\scrD_{\mathrm{\acute{e}t}}^\otimes\colon \mathrm{Corr}^\otimes(\mathrm{vSt})_{\ell\text{-}\mathrm{fine}} 
		\to \mathrm{Cat}_\infty^\times,\end{equation} where the $\ell$-fine subscript indicates that we restrict to the full subcategory of $\mathrm{Corr}(\mathrm{vSt})$ whose correspondences have $\ell$-fine maps to the right. We extend it to $\ell$-adic coefficients via the naive construction of taking limits and tensoring with $\bbQ$, instead of using nuclear $\ell$-adic sheaves, compare with \cite[Section 26]{Sch17} and \cite[page 6]{Man22b}. Note also that $\scrD_{\mathrm{\acute{e}t}}$ is a map of $\infty$-operads, and not symmetric monoidal (only lax symmetric monoidal). Despite this, all the maps obtained below between $\infty$-operads of either correspondences or sheaves will turn out to be monoidal.
	
	Now, we are going to enhance $\ast$ to a monoidal structure on the $\infty$-category, by constructing a map $\scrH^\otimes_S \colon \mathrm{Assoc}^\otimes \to \mathrm{Corr}^{\otimes}(\mathrm{vSt})$ of $\infty$-operads that commutes with the maps towards $\mathrm{Comm}^\otimes$ and recovers the convolution $\ast$ on $\Hk_{\calG,S}$ via evaluation on the active morphism $\langle 2 \rangle \to \langle 1 \rangle$ (with order $1<2$). We were crucially assisted in this task by discussions with Heyer, Mann, and Zhao. A similar construction of the monoidal structure appeared later in \cite[Example 8.7 and Corollary 8.11]{Zhu25}. Note that there is an obvious isomorphism
	\begin{equation}
		\Hk_{\calG,S}:= L_S^+\calG \backslash L_S\calG /L_S^+\calG \simeq [*/L_S^+\calG] \times_{[*/L_S\calG]} [*/L^+_S\calG]
	\end{equation}
	One can therefore realize $\Hk_{\calG,S}$ as the internal endomorphism object of $[*/L_S^+\calG] $ in the symmetric monoidal\footnote{Note that the symmetric monoidal structure of  $\mathrm{Corr}(\mathrm{vSt}_{/[*/L_S\calG]})$ is induced by finite products in $\mathrm{vSt}_{/[*/L_S\calG]}$, i.e. fiber products in $\mathrm{vSt}$ over the base $*/L_S\calG$.}
	$\infty$-category $\mathrm{Corr}(\mathrm{vSt}_{/[*/L_S\calG]})$ (see \cite[Proposition 2.4.1]{HM24}),\footnote{To enhance it to an internal endomorphism in $\mathrm{Corr}^\otimes(\mathrm{vSt}_{/[*/L_S\calG]})_{\ell\text{-}\mathrm{fine}}$, 
		we need to replace it with the finite dimensional truncations as in \cite[Definition 5.1.2, Subsection 5.1.7]{XZ17}, and apply \cite[Proposition 2.3.9]{HM24}.
	}
	equipping it with a monoid structure in the category of correspondences over $[*/L_S\calG]$ and hence it inherits a natural $\infty$-monoidal structure in the category of correspondences over $*$ by forgetting the slice over $[*/L_S\calG]$.
	In fact, the functor $\mathrm{Corr}(\mathrm{vSt}_{/[*/L_S\calG]})\to \mathrm{Corr}(\mathrm{vSt})$ is naturally lax symmetric monoidal\footnote{See \cite[Definition 4.1.3.(b)]{HM24}. Another way of checking that the forgetful functor $\mathrm{Corr}(\mathrm{vSt}_{/X})\to \mathrm{Corr}(\mathrm{vSt}_{/Y})$ for any map $X\to Y$ of v-stacks is lax symmetric monoidal is to realize it as the right adjoint of the base change map $\mathrm{Corr}(\mathrm{vSt}_{/Y})\to \mathrm{Corr}(\mathrm{vSt}_{/X})$ (as explained in \ref{lem_ula_conv}), which is clearly monoidal, so the claim follows from \cite[Corollary 7.3.2.7]{Lur17}.
	}, i.e., it is a map of $\infty$-operads in the sense of \cite[Definition 2.1.2.7]{Lur17}, so it preserves associative algebras essentially by definition, compare with \cite[Definition 2.1.3.1]{Lur17}. Note that the corresponding monoid object in the sense of \cite[Definition 4.1.2.5]{Lur17} is nothing other than the \v{C}ech nerve of the natural map $[*/L_S^+\calG] \to [*/L_S\calG]$.

	Let us try to understand more closely what the map $\scrH^\otimes_S\colon \mathrm{Assoc}^\otimes \to \mathrm{Corr}^\otimes(\mathrm{vSt})$ induced by the above monoidal structure on $\Hk_{\calG,S}$ looks like. We send an object $\langle n \rangle$ to the fiber product $\Hk_{\calG,S}^n$ over $S$ and the morphism $\alpha \colon \langle n \rangle \to \langle m \rangle$ to the correspondence \begin{equation}
		\Hk_{\calG,S}^n \leftarrow \Hk_{\calG,S}^{\alpha} \to \Hk_{\calG,S}^m,
	\end{equation} where the middle term is the $m$-indexed product of the convolution Hecke stacks in the sense of \Cref{sec:convolution-fusion-definition-convolution-hecke-stack} with superscripts ranging over the ordered fibers of $\alpha$, the left map is the natural projection and the right map is the product of the natural multiplication. Note that the left map is a torsor for a power of $L_S^+\calG$, and thus pro-smooth, whereas the right map is fibered in powers of $\Gr_{\calG,S}$ and hence it is ind-proper. One can also write down the image under $\scrH^\otimes_S$ of arbitrary $n$-morphisms of the $1$-category $\mathrm{Assoc}^\otimes$, which are in bijection with sequences of composable morphisms. 
	
	In order to be able to apply the functor $\scrD_{\et}^{\otimes}$, we have to replace the convolution Hecke stacks by finite-dimensional truncations so that the maps to the right become $\ell$-fine, but here we will ignore this subtlety and refer to \cite[Definition 5.1.2, Subsection 5.1.7]{XZ17} for a detailed treatment.
	We denote by $\scrD^\otimes_\mathrm{\acute{e}t}(\Hk_{\calG,S})$ the monoidal $\infty$-category obtained from composing $\scrH_S^\otimes$ and $\scrD_{\et}^\otimes$ (after taking appropriate truncations, so that this becomes legitimate). This clearly refines the convolution product $\ast$, as seen by taking one of the two active maps $\langle 2 \rangle \to \langle 1 \rangle $.
	
	\begin{lemma}
		The full subcategory $\scrD_{\mathrm{ula}}(\Hk_{\calG,S})$ is stable under convolution.
	\end{lemma}
	\begin{proof}
		The reason is that ula sheaves are preserved under smooth pullback by \cite[Proposition IV.2.13]{FS21}, exterior products by \cite[Lemma IV.2.14]{FS21}, and proper pushforward by \cite[Proposition IV.2.11]{FS21}. These are exactly the operations involved in the convolution product.
	\end{proof}
	
	By \cite[Proposition 2.2.1.1]{Lur17} on full subcategories of $\infty$-operads, we have a monoidal $\infty$-category $\scrD_{\mathrm{ula}}(\Hk_{\calG,S})$ giving rise to convolution.
	
	\begin{lemma}\label{lem_ula_conv}
		Given a map $f\colon T\to S$, the pullback functor $f^*$ is monoidal, i.e., it enhances essentially uniquely to a $\bbE_1$-monoidal map $f^{*,\otimes}\colon \scrD^{\otimes}_{\mathrm{\acute{e}t}}(\Hk_{\calG,S}) \to \scrD^{\otimes}_{\mathrm{\acute{e}t}}(\Hk_{\calG,T})$ (and similarly for ula sheaves).
	\end{lemma}
	\begin{proof}
		Recall that we have adjoint functors between $\mathrm{Corr}(\mathrm{vSt}_S)$ and $\mathrm{Corr}(\mathrm{vSt}_T)$ given by base change and forgetting the base, respectively. Indeed, the data of a correspondence $X\times_S T \leftarrow Z \to Y$ over $T$ is equivalent to that of $X \leftarrow Z \to Y$ over $S$ (note that $Z$ lives canonically over $T$ through its morphism to $Y$). We have already noticed that the left adjoint $-\times_S T$ is symmetric monoidal, while the right adjoint (the forgetful functor) is only lax symmetric monoidal. The latter lax monoidal structure is given at the level of the active multiplication map $\langle 2\rangle\to \langle1\rangle$ by the correspondence $X \times_S Y\leftarrow X\times_T Y \to X \times_T Y$. The adjunction unit is a natural transformation between functors of $\infty$-operads, so upon evaluation it yields a map of associative algebras $\scrH_{S}^{\otimes} \to \scrH_T^{\otimes}$ (because the inclusion of maps of $\infty$-operads into the functor category is full, compare with \cite[Definition 2.1.3.1]{Lur17}). Applying now the lax monoidal functor $\scrD_{\mathrm{\acute{e}t}}$, we essentially get by construction that the pullback functor $f^{\ast}\colon \scrD_{\mathrm{\acute{e}t}}(\Hk_{\calG,S}) \to \scrD_{\mathrm{\acute{e}t}}(\Hk_{\calG,T})$ is equipped with a $\bbE_1$-monoidal structure. Pullback also clearly preserves ula sheaves with respect to the base, so the last claim is clear.
	\end{proof}
	
	Ultimately, what is happening in the previous lemma is that we are showing a sort of functoriality for the \v{C}ech nerves. This could have been achieved by appealing to \cite[Lemma 8.10]{Zhu25} instead, but the proof above is conceptually simpler and self-contained.
	
	\begin{proposition}\label{prop_Z_monoidal}
		The functor $\scrZ \colon \scrD_{\mathrm{ula}}(\Hk_{G,C}) \to \scrD_{\mathrm{ula}}(\Hk_{\calG,k})$ is monoidal, i.e., it enhances essentially uniquely to a $\bbE_1$-monoidal map $\scrZ^\otimes \colon \scrD_{\mathrm{ula}}^{\otimes}(\Hk_{G,C}) \to \scrD_{\mathrm{ula}}^{\otimes}(\Hk_{\calG,k})$.
	\end{proposition}
	
	\begin{proof}
		Recall that $\scrZ=i^*Rj_*$, where $j$ and $i$ denote the inclusion of the generic and special fibers of $\Hk_{\calI, O_C}$. We have seen that both pullback functors $j^*$ and $i^*$ are monoidal, thanks to \Cref{lem_ula_conv}. We claim that on ula objects, $j^*$ induces an isomorphism of $\infty$-operads. This can be checked at the level of underlying $\infty$-categories by \cite[Remark 2.1.3.8]{Lur17}, and that statement is \cite[Proposition 6.12]{AGLR22}.
	\end{proof}
	
	There is a more general version of the Hecke stack that can be obtained by not taking $\Spd O$ as the base, but allowing products with itself over $\Spd k$.
	
	\begin{definition}
		Let $S_i \to \Spd O$, $i=1,\dots,d$ be finitely many v-sheaves over $O$. We define the Hecke stack $\Hk_{\calG,S}$ with $S=S_1\times \dots \times S_d$ as the classifying stack of modifications of $\calG$-bundles over the completion of the relative curve $\calY_S$ at the union of the $d$ Cartier divisors specified by the $d$ projections $S\to S_i$, see \cite[Definition VI.1.6]{FS21}.
	\end{definition}
	
	A similar variant exists for the convolution Hecke stacks, where one allows compositions of several modifications instead of modifying simultaneously at several divisors. We are now able to recall the fusion interpretation from \cite[Section VI.9]{FS21} that refines the convolution product and induces symmetry constraints on perverse sheaves. Recall that a perverse t-structure on $\scrD_{\mathrm{\acute{e}t}}(\Hk_{G,C})$ was defined by Fargues--Scholze in \cite[Definition/Proposition VI.7.1]{FS21}.
	
	During the rest of this subsection and the next one, we are going to abbreviate the categories $\scrD_{\mathrm{ula}}(\Hk_{\calG, S})$ by $\scrC_S$, where $S$ is some v-sheaf over $(\Spd O)^n$. If $S$ is the product of the v-sheaf associated with Huber rings $(R_i,R_i^\circ)$ over $O$, then we will write $\scrC_{R_1\times \dots \times R_n}$ for $\scrC_S$ so as to highlight each of the factors. The full subcategory of perverse sheaves will be abbreviated by $\scrP_S$ and $\scrP_{R_1\times \dots \times R_n}$, respectively.
	
	\begin{proposition}\label{prop_fusion}
		The full subcategory $\scrP_C \subset \scrC_C$ of perverse sheaves is stable under convolution and it extends to a symmetric monoidal $\infty$-category.
	\end{proposition}
	
	\begin{proof}
		Stability under convolution can be found in \cite[Proposition VI.8.1]{FS21} and the symmetric monoidal structure follows from \cite[Definition/Proposition VI.9.4]{FS21}. We explain the second part, which will prove useful later on. 
		We have to prove that the bifunctor
		\begin{equation}
			\scrP_C \times \scrP_C \to \scrP_C
		\end{equation}
		is monoidal, where the left side carries the monoidal structure. This will endow $\scrP_C$ with the structure of a braided monoidal category (but we will not check explicitly that it is symmetric). To get the braiding, we extend the map into two commutative triangles in $\mathrm{Cat}_\infty$:
		\begin{equation}\label{eqn_diagram_fusion}\begin{tikzcd}
				& \scrP_{C^2}^{\neq} \\
				\scrP_C \times \scrP_C& \scrP_{C^2}\\
				& \scrP_C
				\arrow[from=2-1, to=2-2]
				\arrow[from=2-1, to=1-2]
				\arrow[from=2-1, to=3-2]
				\arrow[from=2-2, to=1-2]
				\arrow[from=2-2, to=3-2]
			\end{tikzcd}
		\end{equation}
		where the vertical arrows are the pullbacks to the obvious strata of $(\mathrm{Spd}C)^2$ given by the diagonal and its complement, and the middle map is induced by the fusion correspondence
		\begin{equation}
			\Hk_{G,C}^2 \leftarrow \Hk_{G,C} \tilde{\times} \Hk_{G,C} \to \Hk_{G,C^2}.
		\end{equation} The vertical maps are clearly monoidal and the upper one is fully faithful by \cite[Proposition VI.9.3]{FS21}, so it suffices to see that the upper diagonal map is monoidal.
		Indeed, the loop groups $L_{C^2}^+G$ and $L_{C^2}G$ naturally factor into a product away from the diagonal, so we get an induced map $\scrH_C^\otimes \times \scrH_C^\otimes \leftarrow \scrH_{C^2}^{\otimes,\neq} $ of functors $\mathrm{Assoc}^\otimes \to \mathrm{Corr}^\otimes(\mathrm{vSt})$ regarded as a correspondence to the left by functoriality of endomorphism objects.
		We note here that the discussion in \Cref{lem_ula_conv} on the monoidal structure of the adjunction between base change and the forgetful functor implies that for any associative algebra $A$ in $\mathrm{Corr}(\mathrm{vSt}_{C^2})$, e.g., $\scrH_C^\otimes \times \scrH_C^\otimes$, the natural map $A\to A_{\neq}$ is a morphism of associative algebra.
		Indeed, the unit of the adjunction is lax-symmetric monoidal for general reasons.
		Here, $A_{\neq}$ denotes the restriction of $A$ to the complement of the diagonal, but seen as an object over $C^2$.
		This yields our desired monoidal map upon applying $\scrD^\otimes$ and restricting to the monoidal subcategories of perverse sheaves.
	\end{proof}
	
	For the reader's convenience, let us explain more informally what is happening in the above proof. Let $\alpha \colon \langle n \rangle \to \langle m \rangle$ be a $1$-morphism in $\mathrm{Assoc}^\otimes$. Notice that we have a composition of two correspondences, namely the fusion and the diagonal ones:
	\begin{equation}
		\Hk_{G, C}^n \leftarrow \Hk^\alpha_{G,C^n} \rightarrow \Hk_{G,C^n}^m \leftarrow \Hk_{G,C}^m
	\end{equation}
	where the first two maps are the natural pro-smooth projection and ind-proper multiplication, and the last is a diagonal closed immersion. It is clear that the fiber product is the usual correspondence defining the monoidal structure on $\scrD_{\mathrm{ula}}(\Hk_{\calG,S})$. Now the advantage of the first correspondence lies in the fact that, after excluding the partial diagonals, the stack $\Hk_{\calG, S}^\alpha$ decomposes as a product of regular Hecke stacks, so that the order of the modifications (in other words the ordering on the fibers of $\alpha$) no longer matters. If we restrict to the full subcategory $\scrP_C \subset \scrC_C$ of perverse sheaves, then pullback $\scrP_{C^n} \to \scrP_{C^n}^{\neq}$ away from the union of the partial diagonals of $(\Spd C)^n$ is fully faithful, see \cite[Proposition VI.9.3]{FS21}. This yields the various symmetry constraints, as desired.

	\subsection{Associative center}
	
	Let $\scrC$ be a monoidal $\infty$-category. One may attach to $\scrC$ another monoidal $\infty$-category called its associative center and denoted $\scrZ(\scrC)$. Observe that the $\infty$-category $\mathrm{End}(\mathrm{\scrC})=\mathrm{Fun}(\scrC,\scrC)$ is left-tensored over $\scrC$ via the latter's monoidal structure. We define $\scrZ(\scrC):=\mathrm{End}_{\scrC\times \scrC}(\scrC)$ of $\scrC$-bilinear endomorphisms in the sense of \cite[Definition 4.6.2.7]{Lur17}. Since these are monoidal $\infty$-categories with tensor structure given by composition, and $\scrC$-bilinearity is stable under composition, we see that the full subcategory $\scrZ(\scrC)$ inherits a monoidal structure. It comes equipped with a natural monoidal map $\mathrm{ev}_1\colon \scrZ(\scrC) \to \scrC $ given by evaluation at the monoidal unit. Also note that this definition coincides by \cite[Theorem 4.4.1.28, Theorem 5.3.1.30]{Lur17} with the center of an associative algebra of $\mathrm{Cat}_\infty$ in the sense of \cite[Definition 5.3.1.12]{Lur17}.

	\begin{theorem}\label{thm_assoc_center}
		The monoidal functor $\scrZ: \scrC_C \to \scrC_k$ lifts monoidally to the center $\scrZ(\scrC_k)$.
	\end{theorem}
	
	\begin{proof}
		According to \cite[Definition 5.3.1.12]{Lur17}, this amounts to showing that the left action morphism $\scrZ_{\mathrm{lc}}:=\ast \circ (\scrZ,\mathrm{id})\colon \scrC_C\times \scrC_k \to \scrC_k$ is monoidal, where the left side is an associative algebra in $\mathrm{Cat}_\infty$ by multiplying coordinates separately.\footnote{Indeed, in \cite[Definition 5.3.1.12]{Lur17} we set $\mathcal{O}^\otimes=\mathbb{E}_1^\otimes$, $\mathcal{D}^\otimes=N(\mathrm{Fin}_\ast)$ and $\mathcal{C}^\otimes=\mathrm{Cat}_\infty$ with the cartesian symmetric monoidal structure (in particular, we are in the situation of \cite[Remark 5.3.1.13]{Lur17}). The morphism $q'$ in \cite[Definition 5.3.1.12]{Lur17} expresses the fact that the cartesian symmetric monoidal structure on $\mathrm{Alg}_{\mathcal{O}}(\mathcal{C})$ exhibits it as left-tensored over itself. Then considering the universal property of centers in \cite[Definition 5.3.1.6]{Lur17} (with $q$ given by the previous $q'$) unravels to the given claim.} We consider the following union of two commutative squares in $\mathrm{Cat}_\infty$:
		\begin{equation}\label{eq_diagram_central}\begin{tikzcd}
				\scrC_C\times \scrC_k& \scrC_{C\times k}\\
				\scrC_{O_C} \times \scrC_{k}& \scrC_{O_C\times k}\\
				\scrC_{k} \times \scrC_k & \scrC_{k^2} 
				\arrow[from=2-2, to=1-2]
				\arrow[from=2-1, to=1-1]
				\arrow[from=1-1, to=1-2]
				\arrow[from=2-1, to=2-2]
				\arrow[from=2-1, to=3-1]
				\arrow[from=2-2, to=3-2]
				\arrow[from=3-1, to=3-2]
			\end{tikzcd}
		\end{equation}
		where the vertical maps are pullback functors and hence clearly monoidal by \Cref{lem_ula_conv}, and the horizontal maps are given by the fusion product (and therefore are not a priori monoidal). Note also that $\scrC_{k^2}=\scrC_k$ as we take products over $k$ itself and so the lower horizontal map is simply convolution. Since the left upper morphism is an equivalence, we recover $\scrZ_{\mathrm{lc}}$ by taking an inverse and composing across the left lower edge of the diagram. Since the right upper map is fully faithful by \Cref{lem_fully_faithful_generic_pullback} below, it suffices to monoidally enhance the upper horizontal map. But this follows as in the case of the fusion map $\scrP_C^2 \to \scrP_{C^2}^{\neq}$ of perverse sheaves away from the diagonal: indeed, the loop groups $L^+_{C\times k}\calG$ and $L_{C\times k}\calG$ split as a direct product of the loop groups over $C$ and $k$, so we get an equivalence by functoriality of endomorphism objects (again, this relies on the discussion in \Cref{lem_ula_conv} regarding monoidality of pullback functors).
	\end{proof}
	
	Again for the reader's convenience, we repeat our explanation of our reasoning in terms of $1$-morphisms of $\mathrm{Assoc}^\otimes$. We have to see that morphisms $\alpha\colon \langle n \rangle \to \langle m \rangle$ in $\mathrm{Assoc}^\otimes$ are naturally intertwined with $\scrZ_{\mathrm{lc}}$, i.e., that the diagram below
	\begin{equation}\begin{tikzcd}
			\scrC_C^n\times \scrC_k^n& \scrC_C^m\times \scrC_k^m\\
			\scrC_k^n & \scrC_k^m
			\arrow[from=1-2, to=2-2]
			\arrow[from=1-1, to=2-1]
			\arrow[from=1-1, to=1-2]
			\arrow[from=2-1, to=2-2]
		\end{tikzcd}
	\end{equation} commutes, where the vertical maps are powers of $\scrZ_{\mathrm{lc}}$ and the horizontal ones are induced by $\alpha$.
	
	Notice that the composition across the right arises from the composition of correspondences
	\begin{equation}
		\Hk_{\calG,O_C}^n \times \Hk_{\calG,k}^n \leftarrow \Hk_{\calG,O_C^n \times k^n}^{\gamma[m]\circ\alpha[2]} \to \Hk_{\calG,O_C^n\times k^n}^m
	\end{equation}
	where $\gamma: \langle 2 \rangle \to \langle 1 \rangle$ is active carrying the usual order, $\gamma[m]\colon \langle 2m \rangle \to \langle m \rangle$ denotes its concatenation, and similarly for $\alpha[2]\colon \langle 2n \rangle \to \langle 2m \rangle$. The composition across the left arises instead from the correspondence
	\begin{equation}
		\Hk_{\calG,O_C}^n \times \Hk_{\calG,k}^n
		\leftarrow \Hk_{\calI, O_C^n \times k^n}^{\alpha\circ \gamma[n]} \to \Hk_{\calG, O_C^n \times k^n}^m.
	\end{equation}
	Indeed, we can invoke the monoidal equivalence $\scrC_{C} \simeq \scrC_{O_C}$ proved in \Cref{prop_Z_monoidal} and \cite[Proposition 6.12]{AGLR22}, apply the monoidal functor $\scrD$ to the previous correspondences, and then compose with the pullback $i^*$. 
	
	In order to verify that these maps are naturally isomorphic, we must be able to swap the contribution of each $O_C$-factor adjacent to a $k$-factor. 
	Thanks again to the equivalence $\scrC_{O_C} \simeq \scrC_C$ of \cite[Proposition 6.12]{AGLR22} and the fully faithful embeddings $\scrC_{O_C^n\times k^n} \subset \scrC_{C^n \times k^n}$ of \Cref{lem_fully_faithful_generic_pullback} proved below, we are reduced to comparing the maps after taking the pullback functor $j^*$. But since $\Spd C$ and $\Spd k$ map disjointly to $\Spd O$, both convolution Hecke stacks become isomorphic to $\Hk_{G,C^n}^{\alpha}\times \Hk_{\calG,k^n}^{\alpha}$, so the result is clear. 
	
	The following lemmas were used in the proof of \Cref{thm_assoc_center}:
	
	\begin{lemma}\label{lem_fully_faithful_generic_pullback}
		The natural map $j^*\colon \scrC_{O_C^n\times k^m} \to \scrC_{C^n\times k^m}$ is fully faithful for any $n,m\geq 0$.
	\end{lemma}
	
	\begin{proof}
		We must show that the unit $A= Rj_*j^*A$ for every ula sheaf. This follows from the ula base change, see \cite[Corollary IV.2.29]{FS21}, and the next lemma.
	\end{proof}
	
	Recall our shorthand notation for the various functors and categories defined over products of $O_C$, $C$, and $k$. In order to avoid cumbersome notation below involving $\mathrm{Spd}$ and lots of brackets, we apply this convention now to the point functor, so that $*_{O_C^n}:=(\Spd O_C)^n$
	
	\begin{lemma}\label{lemma_nearby_cycles_prod_points_new}
		If $j\colon *_{C^n}:=(\Spd C)^n \to (\Spd O_C)^n=: *_{O_C^n}$, then $Rj_*\La=\La$.
	\end{lemma}
	
	\begin{proof}For reasons that will become clear during our induction argument, we replace the exponent $n$ by a finite set $J$ during our proof.
		If $\lvert J\rvert =1$, this follows already from \cite[Theorem 4.7]{GL24} applied to the kimberlite $*_{O_C}$, since its reduction equals $*_{k}$ and hence nearby cycles are per definition algebraic, so they can be calculated via the étale site for kimberlites, which is trivial.
		
		If $\lvert J\rvert =2$, then we first compute the stalk of $Rj_*\La$ at $*_{ k\times C}$. We know that partially compactly supported cohomology vanishes by \cite[Theorem IV.5.3]{FS21}, so $R\Gamma(*_{O_C \times C}, j_!\La)=0$, compare with \cite[Proposition V.4.2, Remark V.4.3]{FS21}. This means that our sought stalk is given by $R\Gamma(*_{C^2},\La)$ which coincides with $\La$ thanks to \cite[Theorem 19.5]{Sch17}. It remains to compute the stalk at the reduction $*_k$ of the kimberlite $*_{O_C^2}$, so we apply \cite[Theorem 4.7]{GL24} once again.
		
		Finally, in the general case, we can stratify $*_{O_C^J}$ by locally closed subsets of the form $*_{C^K}$ where $K\subset J$. We prove the equality $Rj_*\La=\La$ on the analytic strata (i.e., with $K$ being non-empty) by descending induction on the cardinality of $K$. If $K=J$, there is nothing to show. Otherwise, consider the open set $*_{O_C^{J\setminus K} \times C^K}$ and observe again by \cite[Theorem IV.5.3]{FS21} that $R\pi_*i_!\La=0$ where \begin{equation}
			*_{O_C^{J \setminus K'} \times C^{K'}} \xrightarrow{i} *_{O_C^{J \setminus K} \times C^K} \xrightarrow{\pi}  *_{O_C^{J \setminus K'} \times C^K},
		\end{equation} and $K \subset K'$ has singleton complement. This implies the claim regarding the stratum $*_{C^K}$ again thanks to \cite[Theorem 19.5]{Sch17}. As for the non-analytic point $*_k$ of $*_{O_C^J}$, we invoke \cite[Theorem 4.7]{GL24} again for the last time.   
	\end{proof}
	
	Next, we prove that the symmetry constraints that appear in the full subcategory $\scrP_C  \subset \scrC_C$ of perverse sheaves are compatible with the braidings in the associative center $\scrZ(\scrC_k)$. While $\scrP_C$ is symmetric monoidal, $\scrZ(\scrC_k)$ is not. Instead, the associative center carries a structure over the $\infty$-operad $\bbE_2^\otimes$ of little squares, see \cite[Definition 5.1.0.2]{Lur17}. This arises more formally as the tensor product in $\mathrm{Op}_\infty$ of $\bbE_1^\otimes$ with itself, see \cite[Theorem 5.1.2.2]{Lur17}. Here, we identify $\bbE_1^\otimes$ with $\mathrm{Assoc}^\otimes$ via \cite[Example 5.1.0.7]{Lur17}. Our assertion that associative centers carry an $\bbE_2^\otimes$-structure is \cite[Remark 5.3.1.13]{Lur17}, which explains that they can be regarded as associative algebras in the category of associative algebras of $\mathrm{Cat}_{\infty}$, the extra associative structure arising by bilinearity.
	
	\begin{theorem}\label{thm_E2_mor}
		The composite $\scrP_C \subset \scrC_C\xrightarrow{\scrZ} \scrZ(\scrC_k)$ is an $\bbE_2$-monoidal map.
	\end{theorem} 
	
	\begin{proof}
		Our goal is verifying that the monoidal map $\scrP_C \to \scrZ(\scrC_k)$ actually respects the extra monoidal structures on both sides in the $\infty$-category of associative algebras. By the fact that $\mathbb{E}_2$-algebras are $\mathbb{E}_1$-algebras in $\mathbb{E}_1$-algebras (\cite[5.1.2.2]{Lur17}), this amounts to checking
		that the following commutative square 
		\begin{equation}\begin{tikzcd}
				\scrP_C\times \scrP_C& \scrP_C \times \scrC_{k}\\
				\scrP_C & \scrC_k
				\arrow[from=1-2, to=2-2]
				\arrow[from=1-1, to=2-1]
				\arrow[from=1-1, to=1-2]
				\arrow[from=2-1, to=2-2]
			\end{tikzcd}
		\end{equation}
		in $\mathrm{Cat}_\infty$ is actually a commutative square in $\mathrm{Alg}_{\bbE_1^\otimes}(\mathrm{Cat}^\times_\infty)$, where the maps are the obvious ones induced by convolution or $\scrZ$ and their monoidal enhancements were defined in \Cref{prop_fusion} and \Cref{thm_assoc_center}.

		Let us recapitulate how the braiding isomorphisms were constructed. For $\scrP_C$, we saw during \Cref{prop_fusion} how to define a monoidal structure on the left vertical map via a pair of commuting triangles.
		In the special fiber, we saw during \Cref{thm_assoc_center} how to define a monoidal structure on the right vertical map (actually, we took the larger category $\scrC_C$ instead of just $\scrP_C$) via a pair of commuting squares.
		
		We must now perform these constructions at once in such a way that they are intertwined. Indeed, we have the following pair of commuting triangles
		\begin{equation}\begin{tikzcd}
				& \scrC_{O_C^2}^{\neq} \\
				\scrC_{O_C} \times \scrC_{O_C}& \scrC_{O_C^2}\\
				& \scrC_{O_C}
				\arrow[from=2-1, to=2-2]
				\arrow[from=2-1, to=1-2]
				\arrow[from=2-1, to=3-2]
				\arrow[from=2-2, to=1-2]
				\arrow[from=2-2, to=3-2]
			\end{tikzcd}
		\end{equation}
		which relate to the previously constructed diagrams via natural pullback functors and passing to certain full subcategories. More precisely, restricting to $C^2$ and to perverse sheaves recovers the diagram \eqref{eqn_diagram_fusion}, while restricting to $O_C \times k$ recovers the diagram \eqref{eq_diagram_central} up to composing across the upper left and the lower left corners. Now, the upper vertical map is not fully faithful, and so we need to restrict to a full subcategory of sheaves where that happens. It suffices to take the category $\scrE_{O_C^2}$ of sheaves which are perverse over $C^2$ by \Cref{lem_fully_faithful_integers} below. 
	\end{proof}
	
	\begin{lemma}\label{lem_fully_faithful_integers}
		Denote by $\scrE_{O_C^n}$ the $\infty$-category given as the fiber product $ \scrC_{O_C^n} \times_{ \scrC_{C^n}} \scrP_{C^n}$. Then, the pullback functor $\scrE_{O_C^n}\to \scrE_{O_C^n}^{\neq}$ is fully faithful.
	\end{lemma}
	
	\begin{proof}
		We have a commutative diagram
		\begin{equation}\begin{tikzcd}
				\scrE_{O_C^n}& \scrE_{O_C^n}^{\neq} \\
				\scrP_{C^n} & \scrP_{C^n}^{\neq}
				\arrow[from=1-2, to=2-2]
				\arrow[from=1-1, to=2-1]
				\arrow[from=1-1, to=1-2]
				\arrow[from=2-1, to=2-2]
			\end{tikzcd}
		\end{equation} 
		The left arrow is fully faithful, because it is base changed from $\scrC_{O_C^n} \subset \scrC_{C^n}$ along $\scrP_{C^n} \to \scrC_{C^n}$ as proved in \Cref{lem_fully_faithful_generic_pullback}. The bottom arrow is fully faithful by \cite[Lemma VI.9.3]{FS21}. To show full faithfulness of the right arrow, it suffices to handle the map $\scrC_{O_C^n}^{\neq} \to \scrC_{C^n}^{\neq}$. By the ula property, we are reduced to showing that the derived pushforward of the constant sheaf along $*_{C^n}^{\neq} \to *_{O_C^n}^{\neq}$ is constant, which is also a consequence of \Cref{lemma_nearby_cycles_prod_points_new}. In particular, the top arrow is fully faithful.
	\end{proof}
	
	This concludes the part of our work most intensely concerned with $\infty$-categorical matters. It is natural to wonder if these arguments can be equally carried out in equicharacteristic, e.g., in the setting of \cite{AR}. Certainly, it is possible to pass to the corresponding v-sheaves and perform our arguments again, but this seems unsatisfactory. Some care would be needed when replacing $O_C^n$ by the corresponding scheme-theoretic product, as it does not appear so well-behaved to us. On the other hand, using the nearby cycle construction of \cite[Definition 3.1]{CvdHS24} that goes via pushforward along the original curve $\mathbb{A}^1_k$ and then identifying the right component via $\mathbb{G}_{m,k}$-equivariance, we think our arguments should generalize in a straightforward manner.
	
	\subsection{Perversity}
	
	Recall that for every algebraically closed field $C$, we have a perverse t-structure on $\scrD_{\mathrm{\acute{e}t}}(\Hk_{\calG,C})$ given by strata dimension, see \cite[Section VI.7]{FS21} and \cite{AGLR22}. 
	This restricts to a t-structure on the full subcategory $\scrD_{\mathrm{ula}}(\Hk_{\calG,C})$ of ula sheaves, since $\La$ is a field\footnote{Otherwise, the truncation functors do not generally preserve perfect complexes, an issue that already arises for the natural t-structure.}. It would be possible to define a relative perverse t-structure as in \cite[Definition/Proposition VI.7.1]{FS21}, at least after restricting to ula sheaves, but we will not pursue this avenue here.

	Our main result is the perverse t-exactness of $\scrZ$ at Iwahori level.
	
	\begin{theorem}\label{thm_perversity_new}
		Assume $\calG=\calI$ is Iwahori. Let $B \subset G$ be an arbitrary Borel subgroup. The complex $\scrZ(V)$ is a Wakimoto-filtered perverse sheaf with graded isomorphic to $\scrJ(V|_{\hat{T}^\Gamma})$ for any $V\in \mathrm{Rep}(\hat{G})$.    
	\end{theorem}
	\begin{proof}
		By \Cref{thm_assoc_center}, we see that $\scrZ(V)$ lies in the essential image of the obvious evaluation functor coming from the associative center $\scrZ(\scrD_{\mathrm{ula}}(\Hk_{\calI, k})) $. By \Cref{prop_wak_category_central}, this implies that $\scrZ(V)$ lies in the full subcategory $\mathrm{Wak}$ for our choice of Borel subgroup $B \subset G$. It remains to see that there is a canonical isomorphism \begin{equation}
			\label{eq:1-equation-for-graded-of-central-sheaf_new}
			\mathrm{CT}_{B^-}(\scrZ(V))_{\bar \nu}\simeq V(w_0\bar \nu)[\langle 2\rho, \nu \rangle],
		\end{equation}
		where $w_0$ denotes the longest element of the finite absolute Weyl group of $G$. Indeed, we would then know by \Cref{coro_wakimoto_graded_ct} that $\scrZ(V)$ is a perverse sheaf, because the same would hold for its Wakimoto grading. But notice that constant terms of $\scrZ(V)$ can be calculated applying geometric Satake in the generic fiber, see \cite[Corollary 6.14, Equation (6.32)]{AGLR22}, which yields the desired answer. 
	\end{proof}
	
	\begin{remark}
		There appears to be a discrepancy between the isomorphism $\mathrm{grad}\circ\scrZ(V) \simeq V|_{\hat{T}^\Gamma}$ and \eqref{eq:1-equation-for-graded-of-central-sheaf_new} due to the appearance of the longest element $w_0$ in the latter formula. However, this is due to the fact that we were implicitly using an identification of $T$ with the universal Cartan of $G$, compare with \cite[Remark 1.3.7]{AR}. Conjugating the identification by $w_0$ will not change the $\hat{T}^I$-grading coming from geometric Satake, but will change the one coming from the Wakimoto filtration, thereby fixing the issue.
	\end{remark}
	
	Next, we deduce a few important consequences from this theorem. We start by proving that $\scrZ(V)$ is perverse for general parahorics. This is based on a suggestion of Achar to Cass--van den Hove--Scholbach, see \cite[Theorem 5.30]{CvdHS24}.
	
	\begin{corollary}
		Let $\calG$ be an arbitrary parahoric. Then, $\scrZ(V)$ is a perverse sheaf.
	\end{corollary}
	
	\begin{proof}
		Let $\pi:\Fl_\calI\rightarrow\Fl_\calG$ denote the natural projection. Pick a Borel subgroup $T\subset B\subset G$ such that the underlying euclidean roots of the affine roots vanishing on the facet $\bbf$ fixed by $\calG(O)$ are positive with respect to $B$. One can easily check that $t_{\bar \mu}$ is right $W_\bbf$-minimal for all $B$-dominant $\bar \mu$, compare with \cite[Lemma 5.28]{CvdHS24}, so the map $\Fl_{ \calI,t_{\bar \mu}}\to \Fl_{(\calI,\calG),t_{\bar \mu}}$ is an isomorphism under the same assumption. Let now $\bar \nu$ be an arbitrary coweight and write it as the difference $\bar\nu_1-\bar\nu_2$ of two $B$-dominant coweights. Collecting the previous facts, we deduce that $R\pi_* \scrJ^B_{\bar{\nu}}=\Delta_{\calI,t_{\bar \nu_2}} \ast \nabla_{(\calI,\calG),t_{\bar \nu_1}}$ is left $t$-exact by \Cref{fn:1-convolution-of-standard-costandard-objects-perverse}.
		
		Since nearby cycle functors commute with proper pushforward, $\scrZ^\calG=R\pi_*\circ \scrZ^\calI$ where the superscripts are clear from content. Thus $\scrZ^\calG(V)$ lies in non-negative perverse degrees by the $t$-exactness of Wakimoto filtration of $\scrZ^\calI$ and the above discussion. Similarly, after replacing $B$ by the opposite Borel, we can see that $\scrZ(V)$ lies in non-positive degrees.
	\end{proof}
	
	From now on, we always assume that the parahoric level $\calG=\calI$ is Iwahori. We say that a central perverse sheaf $A$ is convolution exact if its left (equivalently right) convolution functor $\ell_A \colon \scrD_{\mathrm{ula}}(\Hk_{\calI,k}) \to \scrD_{\mathrm{ula}}(\Hk_{\calI,k})  $ is t-exact for the perverse t-structure.
	
	\begin{corollary}\label{cor_Z_conv_exact}
		The central perverse sheaf $\scrZ(V)$ is convolution exact.
	\end{corollary}
	
	\begin{proof}
		Given an element $w$ of the affine Weyl group, we can find a Borel $B \subset G$ such that $\ell(t_{\bar\nu} w)=\ell(t_{\bar\nu}) +\ell(w)$ for all $\bar \nu \gg 0$ with respect to $B$. Indeed, we can consider a minimal gallery from the $\calI(O)$-stable alcove $\bba$ to its Weyl translate $w\bba$, and simply take $B$ as the Borel corresponding to a Weyl chamber containing the vector given as the difference of the barycenters of $\bba$ and $w\bba$. 
		
		Now consider the complexes $\scrJ_{\bar \nu} \ast \nabla_w$ for arbitrary $\bar \nu$, and notice that it equals the perverse sheaf $\Delta_{t_{-\bar \nu'}} \ast \nabla_{t_{\bar \nu''} w}$ if we choose $\bar \nu=\bar{\nu}''-\bar{\nu}' $ and $\bar{\nu}''\gg 0$. Here, we applied \Cref{lem_wak_basic_prop} and \Cref{lemma_convolution_stds_and_costds}. Now, \Cref{thm_perversity_new} states that the perverse sheaf $\scrZ(V)$ admits a filtration with subquotients isomorphic to a direct sum of $\scrJ_{\bar \nu}$, hence implying that $\scrZ(V)\ast \nabla_w$ is perverse for any $V$. By a dual argument, the same result holds for $\Delta_w$. Finally, we apply the fact that the iterated extensions of the non-positive shifts of $\nabla_w$ (resp.~non-negative shifts of $\Delta_w$) span the non-negative part ${}^p\scrD_{\mathrm{ula}}^{\geq 0}(\Hk_{\calI,k})$ (resp. the non-positive part ${}^p\scrD_{\mathrm{ula}}^{\leq 0}(\Hk_{\calI,k})$) of the perverse t-structure to deduce that $\ell_{\scrZ(V)}$ is indeed perverse t-exact. 
	\end{proof}
	
	In the following, we say that an endomorphism $\varphi: A \to A$ of an object $A$ in an abelian category $\scrC$ is unipotent if $(\varphi-1)^n=0$ for some positive integer $n$. 
	We say $\varphi$ is quasi-unipotent if a power of $\varphi$ is unipotent.
	Recall that $\scrZ(V)$ carries a natural $I_E$-action, where $E$ is the reflex field of the representation $V$ and $I_E\subseteq \Gamma_E$ the inertia subgroup, see \Cref{sec:background-central-functor-with-galois-action}.
	
	\begin{corollary} \label{unipotent_new}
		The $I_E$-action on the perverse sheaf $\scrZ(V)$ is given by quasi-unipotent automorphisms. Moreover, there exists a finite index subgroup $I'\subset I_E$ such that the action factors through its maximal pro-$\ell$ quotient. If $G$ is split, then $I'=I_E=I$ acts unipotently on $\scrZ(V)$.
	\end{corollary}
	
	\begin{proof}
		Since $I_E$ fixes a Borel subgroup $B \subset G$ defined over $F$, we conclude the $I_E$-action on $\scrZ(V)$ preserves the Wakimoto filtration and it acts on $\on{Grad}_{\bar \nu}(\scrZ(V))\simeq V(w_0\bar \nu)$, compare with \Cref{thm_perversity_new}, via its natural action on the given weight space. Since $V(w_0\bar \nu)$ equals the sum of the $V(w_0 \nu)$ for all lifts $\nu$ of $\bar \nu$, we see that $I_E$ acts on the Wakimoto sheaves by permuting those weight spaces. Let $F'$ be a splitting field of $G$ and note that its absolute Galois group $I'$ acts trivially on $V$. In particular, the $I'$-action on $\scrZ(V)$ is unipotent. Note, moreover, that both the pro-$p$ wild inertia, and the remaining prime-to-$\ell$ tame quotient must map trivially to an unipotent $\ell$-adic group, so the $I'$-action factors through its maximal pro-$\ell$ quotient.
	\end{proof}

	In particular, if $\La$ is an algebraic extension of $\bbQ_\ell$ and given an isomorphism between $\bbZ_\ell$ and the maximal pro-$\ell$ quotient of $I'$, we deduce the existence of a canonical nilpotent morphism
	\begin{equation}
		\mathbf{n}_V: \scrZ(V) \longrightarrow \scrZ(V)
	\end{equation}
	such that the action of $\gamma' \in I'$ on $\scrZ(V)$ 
	is given by $\on{exp}(t_\ell(\gamma') \mathbf{n}_V)$,
	where $t_\ell: I' \rightarrow \mathbb{Z}_\ell$ is the natural quotient map.

	\begin{corollary}
		\label{sec:perversity-graded-monoidal}
		The isomorphism of functors
		$\mathrm{Grad}\circ \scrZ(V)\simeq V|_{\hat{T}^I}$ is monoidal.
	\end{corollary}
	\begin{proof}
		We first explain how to construct the monoidal structure of the restriction functor $V\mapsto V_{|\hat{T}^I}$ geometrically using constant terms following \cite[Section 6]{Yu22} and \cite[Section 4]{ALRR24}. Namely, for any $\calA,\calB\in \scrP(\mathrm{Hk}_{G,C})$, we obtain isomorphisms
		\begin{align*}
			\mathrm{CT}_{B^-}(\calA*\calB)_{\nu} & \cong R\Gamma_c(S^-_{\nu},\calA*\calB)\\   
			&\cong R\Gamma_c(m^{-1}(S^-_\nu),\calA\tilde{\boxtimes}\calB)\\
			& \cong \bigoplus_{\nu_1+\nu_2=\nu}R\Gamma_c(S^-_{\nu_1}\tilde{\times}S^-_{\nu_2},\calA\tilde{\boxtimes}\calB)\\
			& \cong \bigoplus_{\nu_1+\nu_2=\nu}R\Gamma_c(S^-_{\nu_1},\calA)\otimes R\Gamma_c(S^-_{\nu_2},\calB))\\
			& \cong \bigoplus_{\nu_1+\nu_2=\nu} \mathrm{CT}_{B^-}(\calA)_{\nu_1}\otimes \mathrm{CT}_{B^-}(\calB)_{\nu_2}.
		\end{align*}
		The second isomorphism above follows directly from the construction of the convolution. Then we are in a similar situation considered in \cite[Lemma~6.1]{Yu22} and \cite[Corollary~4.16]{ALRR24}, and may thus adapt arguments in \textit{loc.cit} to deduce the remaining isomorphisms. This monoidal structure coincides with the natural one on $V\mapsto V_{|\hat{T}^I}$ under geometric Satake. However, strictly speaking, this construction is not quite complete, because on the one hand \cite{Zhu17,Yu22} works with the Witt vector affine Grassmannian instead of the $B_{\mathrm{dR}}^+$-affine Grassmannian, and moreover it would not be immediate that this monoidal structure is compatible with that of \cite{FS21}. In order to fill this gap, we must use the equivalence  between the Satake categories over $C$ and $\overline{k}$ for the split group $G_C$ in \cite[VI.6.7]{FS21} and a theorem of Bando, see \cite[Theorem 1.1]{Ban22} showing that it is monoidal, i.e., that pulling back the Satake category of \cite{Zhu17,Yu22} to the context of \cite{FS21} yields the same monoidal structure on perverse sheaves.
		
		As nearby cycles commute with constant terms \cite[Proposition IV.6.12]{FS21}, we get an equivalence $\scrZ_T \circ \mathrm{CT}_{B^-}\simeq \mathrm{CT}_{B^-}\circ \scrZ_G$, where the indices $T$ and $G$ denote the underlying group of the Hecke stack for which we take nearby cycles. In particular, we obtain a monoidal structure on the functor $\mathrm{CT}_{B^-}\circ \scrZ_G$ by composing the above monoidal structure on $\mathrm{CT}_{B^-}$ with the one on $\mathscr{Z}_T$ (which clearly coincides with the restriction along $\hat{T}^I\subseteq \hat{T}$ under geometric Satake for $T$ resp.\ $\mathcal{T}$). Looking back to the construction of the monoidal structure on Wakimoto graded pieces in \Cref{prop_grad_monoidal}, it made inductive use of the isomorphisms from \Cref{coro_Grad_composes_Zisom_H} and \Cref{coro_wakimoto_graded_ct}. These resulted as well from decomposing twisted products of semi-infinite orbits, and so this monoidal structure on $\mathrm{CT}_{B^-}\circ \mathscr{Z}_G$ must coincide with the previous one above, which was constructed using geometric Satake.		
	\end{proof}

	\subsection{Highest weight arrows}
	\label{sec:high-weight-arrows}
	Let $\mu\in \bar{\bbX}_\bullet$ be a dominant coweight with respect to $B$. For a $\hat{G}_\Lambda$-representation $V$ with a single highest weight $\mu$, we see that $\scrZ(V)$ is supported on the $\mu$-admissible locus $\calA_{\calI,\mu}$, cf.\ \cite[Theorem 6.16]{AGLR22}, which equals the union of the $\calI(O)$-orbit closures of the translations $t_{\bar \nu}$ associated with weights $\nu$ of $V$. We are going to define a canonical map
	\begin{equation}
		\mathfrak{f}_{V}\colon \scrZ(V) \to \on{gr}_{\bar \mu} \scrZ(V)
	\end{equation}
	called the \textit{highest weight arrow}, which geometrizes the projection onto the $\bar{\mu}$-weight space.

	First, observe that we have the adjunction unit \begin{equation}\scrZ(V) \to Rj_{\bar \mu,*}j_{\bar \mu}^* \scrZ(V).\end{equation} But the restriction of $\scrZ(V)$ to the $I$-orbit $\Fl_{\calI, \bar \mu}$ is isomorphic to the local system with value 
	$R\Gamma(\Fl_{\calI,\bar \mu},\scrZ(V))$. 
	On the other hand, we know by \Cref{lem_schubert_cell_in_semi_infinite_orbit} that $\Fl_{\calI,\bar \mu}$ coincides with the intersection $\calA_{\calI,\mu}\cap \calS_{t_{\bar\mu}}=\Fl_{\calI,\leq \bar\mu}\cap \calS_{t_{\bar\mu}}$. Therefore, \Cref{coro_wakimoto_graded_ct} tells us that $j^*_{\bar \mu} \scrZ(V)\simeq \on{Grad}_{\bar \mu} \scrZ(V)[\langle 2\rho, \mu \rangle]$ 
	in natural fashion. In particular, we get $Rj_{\bar \mu,*}j^*_{\bar \mu} \scrZ(V)\simeq \on{gr}_{\bar \mu} \scrZ(V)$ and we obtain the desired highest weight arrow.
	
	\begin{proposition}
		The highest weights arrows are symmetric monoidal, i.e., for $V$ (resp. $W$) a representation of $\hat{G}_{\Lambda}$ with a single highest weight $\mu$ (resp. $\nu$),  there are natural identifications $\mathfrak{f}_{V} \ast \mathfrak{f}_{W} \simeq \mathfrak{f}_{V\otimes W} \simeq \mathfrak{f}_{W} \ast \mathfrak{f}_{V}$ in the sense that the diagram
		\begin{equation}
			\begin{tikzcd}[row sep=huge]
				\scrZ(V) * \scrZ(W) \arrow[r,"\sim"]	
				\arrow[d,"\mathfrak{f}_V *\mathfrak{f}_W"]
				&
				\scrZ(V\otimes W)
				\arrow[d,"\mathfrak{f}_{V\otimes W}"]
				&
				\scrZ(W) * \scrZ(V)
				\arrow[l,"\sim"']
				\arrow[d,"\mathfrak{f}_W *\mathfrak{f}_V"]
				\\
				\on{gr}_{\bar \mu} \scrZ(V)	* \on{gr}_{\bar \nu} \scrZ(W)
				\arrow[r,"\sim"]
				&
				\on{gr}_{\overline{ \mu +\nu}} \scrZ(V\otimes W)
				&
				\on{gr}_{\bar \nu} \scrZ(W)	* \on{gr}_{\bar \mu} \scrZ(V)
				\arrow[l,"\sim"']
			\end{tikzcd}
		\end{equation}
		is commutative, 
		where the horizontal isomorphisms in the first row stem from \Cref{prop_Z_monoidal}, and the isomorphisms in the second row are given by 
		\begin{equation}
			\on{gr}_{\bar \mu} \scrZ(V)	* \on{gr}_{\bar \nu} \scrZ(W)
			\simeq 
			\on{gr}_{\overline{ \mu+\nu}}( \scrZ(V)	*  \scrZ(W))
			\simeq 
			\on{gr}_{\overline{ \mu+\nu}}( \scrZ(V\otimes W))
		\end{equation}
		with the first isomorphism given by \Cref{gr-convol}. 
	\end{proposition}
	
	\begin{proof}
		By  \Cref{sec:perversity-graded-monoidal} we know that the composition $\on{Grad} \circ \scrZ$ identifies with the restriction functor from $\mathrm{Rep}_\La (\hat{G})$ to $\mathrm{Rep}_\La (\hat{T}^I)$ as a tensor functor, so it is symmetric monoidal. Indeed, the monoidal structure of $\on{Sat}$ comes from the monoidality of constant terms in the generic fiber $\Gr_G$ which is compatible with the one in the special fiber $\Fl_{\calI}$, which was used in \Cref{prop_grad_monoidal}. Finally, we just have to remark that the adjunction unit is naturally symmetric monoidal as are the isomorphisms $Rj_{\bar \mu,*}j^*_{\bar \mu} \scrZ(V)\simeq \on{gr}_{\bar \mu} \scrZ(V)$.
	\end{proof}

	We also have the relation of $\mathfrak{f}_V$ with the monodromy operator.  
	
	\begin{lemma}
		\label{sec:high-weight-arrows-1-monodromy-is-zero-on-highest-weight-vector}
		Let $V$ be a representation of $\hat{G}_{\Lambda}$, then we have
		$$
		\mathfrak{f}_V \circ \mathbf{n}_V =0. 
		$$
	\end{lemma}
	
	\begin{proof}
		By definition, $\mathfrak{f}_V$ is the quotient map of $\scrZ(V)$ towards the final subquotient of the Wakimoto filtration, upon which $I'$ acts trivially by geometric Satake, see \Cref{unipotent_new}.
	\end{proof}

	Moreover, we also have that 
	$\mathbf{n}_V$ is monoidal with respect to the monoidal structure in $\mathrm{Rep}(\hat{G})$. Note that here the tensor product of two nilpotent operators $\mathbf{n}_A$ and $\mathbf{n}_B$ of objects $A$ and $B$ of a monoidal category is given by $\mathbf{n}_A\otimes 1 +1\otimes \mathbf{n}_B$.

	\begin{lemma}\label{D_P_for_n_V}
		The nilpotent endomorphisms $\mathbf{n}_V$ for $V \in \mathrm{Rep}_\La (\hat{G})$ form a nilpotent monoidal endomorphism $\mathbf{n}$ of $\scrZ\colon \mathrm{Rep}_\La (\hat{G})\to \scrP(\Hk_{\calI,k})$ i.e. the following diagram commutes
		$$
		\begin{tikzcd}[row sep=huge]
			\scrZ(V\otimes W) \arrow [r,"\simeq"] \arrow[d,"\mathbf{n}_{V\otimes W}",swap] & \scrZ(V)\ast \scrZ(W) \arrow[d,"\mathbf{n}_V\otimes \mathrm{id}_W+\mathrm{id}_V\otimes \mathbf{n}_W"] \\
			\scrZ(V\otimes W) \arrow[r,"\simeq "] &\scrZ(V)\ast \scrZ(W). 
		\end{tikzcd}
		$$
	\end{lemma}
	
	\begin{proof}
		Recall the $I'$-action on $\scrP(\Hk_{\calI,k})$ constructed in \Cref{unipotent_new}. It is enough to observe that the monoidal structure of $\scrZ\colon \mathrm{Rep}_\La (\hat{G})\to \scrP(\Hk_{\calI,k})$ induced from \Cref{prop_Z_monoidal} is $I'$-equivariant, but this follows directly from the construction of $\scrZ$.
	\end{proof}

	\subsection{Mixed variant}
	
	In this subsection, we are going to upgrade our previous work to the setting of mixed sheaves. We consider a $p$-adic field $F$ with ring of integers $O$, a finite residue field $k$ of cardinality $q$, and an absolute Galois group $\Gamma$. We continue to fix a quasi-split and residually split $F$-group $G$ with an Iwahori $O$-model $\calI$. In this subsection, we assume furthermore that $\La$ is an algebraic extension of $\bbQ_\ell$ and contains a preferred choice of square-root $\sqrt{q}$.
	
	We need to introduce the $\Gamma$-equivariant derived category of étale sheaves on our preferred spaces. Note that the Deligne topos $X\times_s \eta$ for a finite type $k$-scheme $X$ with compatible $\Gamma$-action defined in \cite{SGA7.2}, see also \cite[Definition A.1.3]{HZ23}, is the same as the étale topos of the stack $[\Gamma \backslash X_{\bar k}]$. Indeed, the latter identifies by descent with the category of étale sheaves on $X_{\bar k}$ with continuous compatible $\Gamma$-action as in \cite[Definition A.1.2]{HZ23}, so the conclusion follows from \cite[Lemma A.1.4]{HZ23}.  We usually consider the stable derived category $\scrD_{\mathrm{\acute{e}t}}([\Gamma\backslash X_{\bar k}])$ which is equivalent to the stable derived category $\scrD_{\mathrm{\acute{e}t}}(X\times_s \eta)$ of the Deligne topos, compare with \cite[Construction A.1.6 and Definition B.1.1]{HZ23}. Recall that we have a decisive notion of a mixed complex $A \in \scrD_{\mathrm{\acute{e}t}}([\Gamma \backslash X_{\bar k}] )$ of weight $\leq w$ (resp. $\geq w$) in the sense of \cite[Definition 2.4.4]{HZ23}. The condition $\leq w$ is defined by requiring that $\calH^i(\sig^*A)$ have weights bounded by $w$ in the sense of Deligne, where $\sig \colon X \to [\Gamma \backslash X_{\bar k}]$ is induced by a section of the morphism $\Gamma \to \Gal_k$. The weight bound is ultimately independent from $\sig$, see \cite[Section 2.4]{HZ23} for a discussion. The condition $\geq w$ is defined in terms of $\leq w$ and Verdier duality for $X_{\bar k}$.
	
	Again, we can define the mixed standard functor from $\scrD_{\mathrm{\acute{e}t}}([\Gamma \backslash S_{\bar k}])$ towards $\scrD_{\mathrm{\acute{e}t}}([\Gamma \backslash\Hk_{\calI,\bar k}])$
	\begin{equation}
		\del_w^{\mathrm{mix}}\colon M \mapsto j_{w!}M\langle \ell(w)\rangle,
	\end{equation}
	where $\langle d \rangle$ denotes the shift-twist operator $[d](\frac{d}{2})$, and the mixed costandard functor
	\begin{equation}
		\nab_w^{\mathrm{mix}} \colon M \mapsto Rj_{w*}M\langle \ell(w)\rangle.
	\end{equation}
	both of which preserve mixed perverse sheaves by Weil II. \Cref{djfejkjd;kjkel;} and \Cref{lem_invertible_conv_stand} generalize to the current setting, so that we can define the mixed Wakimoto functor $\scrJ_{\bar \nu}^{\mathrm{mix}} \colon\scrD_{\mathrm{\acute{e}t}}([\Gamma \backslash S_{\bar k}]) \to \scrD_{\mathrm{\acute{e}t}}([\Gamma \backslash\Hk_{\calI,\bar k}]) $ mapping a weighted complex $M$ of $\La$-modules to the object representing
	\begin{equation}
		\calF \mapsto \on{colim}\Hom(\nab^{\mathrm{mix}}_{t_{\bar\nu_1}}(M),\calF \ast \nab^{\mathrm{mix}}_{t_{\bar\nu_2}}(\La)),
	\end{equation}
	where $\bar \nu_1, \bar \nu_2 \in \bar \bbX_\bullet^+$ run over all those elements such that $\bar \nu=\bar \nu_1-\bar \nu_2 $. Again this sends a mixed weighted $\La$-module to a mixed perverse sheaf, since the mixedness property is preserved under derived pushforward and pullback, whereas perversity was already verified in \Cref{wersd;kfdl;}. We define the full subcategory \begin{equation}\mathrm{Wak}^\mathrm{mix}\subset \scrD_{\mathrm{\acute{e}t}}([\Gamma\backslash \Hk_{\calI,\bar k}])
	\end{equation} of mixed Wakimoto complexes, as the span under cones and shifts of the essential image of $\scrJ^{\mathrm{mix}}_{\bar \nu}$ for all $\bar \nu \in \bar \bbX_{\bullet}$. We also refer the reader to \cite[Definition 5.7]{CvdHS24} in the motivic equicharacteristic setting, which discusses weights in greater detail. Similarly, we can generalize \Cref{prop_wakimoto_constant_terms} to the mixed setting, in such a way that it allows us to determine the Wakimoto grading of such an object. The full subcategory $\scrP(\mathrm{Wak}^{\mathrm{mix}})$ consists of the objects in $\mathrm{Wak}^\mathrm{mix}$ whose gradeds are all perverse.
	
	Note that by \cite[Section 8]{AGLR22}, the functor of nearby cycles upgrades to the mixed setting
	\begin{equation}
		R\Psi^{\mathrm{mix}}:=(i^{\mathrm{mix}})^\ast R(j^{\mathrm{mix}})_\ast \colon \scrD_{\mathrm{\acute{e}t}}([\Gamma \backslash\Hk_{G,C}]) \to  \scrD_{\mathrm{\acute{e}t}}([\Gamma \backslash\Hk_{\calI,\bar k}]) 
	\end{equation}
	and composition with the functor $\mathrm{Rep}_\Lambda(^LG)\to \scrP(\Hk_{G,F})$ defines the mixed central functor $\scrZ^{\mathrm{mix}}(-)$. 
	\begin{theorem}
		The mixed central functor $\scrZ^{\mathrm{mix}}$ lands in $\scrP(\mathrm{Wak}^{\mathrm{mix}})$. Concretely, the Wakimoto graded pieces of $\scrZ^{\mathrm{mix}}(V)$ are canonically isomorphic to $\scrJ_{\bar \nu}^{\mathrm{mix}}(V(w_0 \bar \nu))$,\ $\bar \nu\in \bar \bbX_{\bullet}$.
	\end{theorem} 
	
	\begin{proof}
		The arguments of \Cref{thm_perversity_new} apply in this case as well.
	\end{proof}
	
	Recall that there exists a unique exhaustive and separated filtration $\mathrm{Fil}_{i}^{\mathrm{M}}\scrZ(V)$ (called the monodromy filtration) on the perverse sheaf $\scrZ(V)$ such that $\mathbf{n}_V$ is a filtered operator of degree $-2$ inducing isomorphisms $\mathbf{n}_V^i \colon \mathrm{Gr}_{i}^{\mathrm{M}}\scrZ(V) \simeq \mathrm{Gr}_{-i}^{\mathrm{M}}\scrZ(V) $.
	This filtration descends by functoriality to the corresponding mixed object $\scrZ^{\mathrm{mix}}(V)$. Indeed, the unicity of $\mathbf{n}_V$ implies that it defines a $\Gamma$-equivariant morphism $\mathcal{Z}(V)\to \mathcal{Z}(V)(-1)$, where the Tate twist accomodates a trivialization of the tame inertia.
	On the other hand, the mixed perverse sheaf $\scrZ^{\mathrm{mix}}(V)$ admits a filtration $\mathrm{Fil}_{i}^{\mathrm{W}}\scrZ^{\mathrm{mix}}(V) $ in mixed perverse sheaves whose weights are at most $i$ and whose gradeds $\mathrm{Gr}_{i}^{\mathrm{W}}\scrZ^{\mathrm{mix}}(V)$ are purely of weight $i$, see \cite[Théorème 5.3.5]{BBDG18} and \cite[Theorem 2.6.8]{HZ23}. We say following \cite{HZ23} that $\scrZ^{\mathrm{mix}}(V)$ is monodromy-pure of weight $0$ if these two filtrations coincide. We have the following local weight-monodromy conjecture:
	
	\begin{conjecture}\label{wt_monodr_conjecture}
		The mixed perverse sheaf $\scrZ^{\mathrm{mix}}(V)$ is monodromy-pure of weight $0$.
	\end{conjecture}
	
	For finite-type schemes over a field, it is known that nearby cycles send pure sheaves of weight $0$ to monodromy-pure sheaves of weight $0$, by a theorem of Gabber \cite[Theorem 5.1.2]{BB93}. In mixed characteristic, this was partially generalized by Hansen--Zavyalov \cite{HZ23} assuming the existence of an étale cover by rigid-analytic tubes that admit an étale map to a disk.
	
	
	\begin{proposition}\label{prop_wtmon_minuscule}
		If $G$ is split and every non-zero weight of $V$ is minuscule, then \Cref{wt_monodr_conjecture} holds true for $\scrZ(V)$. 
	\end{proposition}
	
	\begin{proof}
		By semi-simplicity of the Satake category in characteristic $0$, we may assume $V=V_\mu$ is the simple representation with highest weight $\mu$. 
		In particular, we know by the proof of \cite[Theorem 7.21, 7.23]{AGLR22}, that the local model $M_{\calI,\mu}$ -defined as the $v$-sheaf closure of the Schubert cell for $\mu$- is representable by a flat projective scheme $M_{\calI,\mu}^{\mathrm{sch}}$ over $O$. By functoriality, it also maps to the local model $\calG/\calP_\mu^-$ at hyperspecial level $\calG$, which is smooth over $O$. Since the transition map is an isomorphism in the generic fiber, we deduce by pull-back an étale cover of $G/P_\mu^-$ by rigid-analytic tubes admitting étale maps to a disk. Therefore, we can apply \cite[Theorem 4.4.4]{HZ23}.
	\end{proof}
	
	\section{Coherent functor}
	In this section, we assume that $G$ is split, that $\Lambda$ is an algebraic extension of $\bbQ_\ell$, and that $\calI$ is the Iwahori $O$-model obtained as the dilatation of a split model $G_O$ along the closed subgroup $B_k \to G_k$. Consider the Springer resolution \begin{equation}\label{Springer_resolution}
		p_\mathrm{Spr}:\hat{\calN}_{\mathrm{Spr}}=\hat{G} \times^{\hat{B}} \mathrm{Lie}\,\hat{U}\rightarrow \hat{\calN} \subset \mathrm{Lie}\, \hat{G}
	\end{equation}
	of the nilpotent cone $\hat{\calN}$ defined over $\La$.
	Let $\mathrm{Coh}([\hat{G}\backslash \hat{\calN}_{\mathrm{Spr}}]) $ denote the abelian category of  of coherent sheaves on the quotient stack $[\hat{G}\backslash \hat{\calN}_{\mathrm{Spr}}]$ .  Observe that there are natural functors 
	$$  
	\Rep_\La (\hat{G}) \rightarrow  \mathrm{Coh}([\hat{G}\backslash \hat{\calN}_{\mathrm{Spr}}]),~  V\mapsto \mathcal{O}\boxtimes V
	$$    
	
	and 
	$$
	\Rep_\La (\hat{T}) \rightarrow  \mathrm{Coh}([\hat{G}\backslash \hat{\calN}_{\mathrm{Spr}}]),~ \nu\mapsto \calO(\nu),
	$$
	where $\mathcal{O}(\nu)$ denotes the line bundle which is the pullback along the natural projection $[\hat{G}\backslash \hat{\calN}_{\mathrm{Spr}}]\rightarrow [\hat{B}\backslash\mathrm{pt}]$ of the line bundle corresponding to $\nu$. We aim to construct a monoidal functor \begin{equation}
		\scrF \colon \mathrm{Perf}([\hat{G} \backslash \hat{\calN}_{\mathrm{Spr}}]) \to \scrD_{\mathrm{\acute{e}t}}(\Hk_{\calI})
	\end{equation}
	of monoidal, stable $\infty$-categories. Here, the domain of $\scrF$ is the category of perfect complex on a smooth Artin $\La$-stack, thus equivalently, the $\infty$-derived category  (in fact the bounded derived category cf.\ (\ref{equation_coh_VS_perf})) of coherent sheaves, and the target of $\scrF$ is the $\infty$-derived category of \'etale $\La$-sheaves on the perfect Artin $k$-stack. The functor $\scrF$ is supposed to extend both the Wakimoto functor $\scrJ$ and the central functor $\scrZ$ in the sense that the composition of $\scrF$ with the functor $V \mapsto V \boxtimes \calO$ on $\mathrm{Rep}_{\Lambda}(\hat{G})$ resp.~ the functor $\nu \mapsto \calO(\nu)$ on $\mathrm{Rep}_\Lambda(\hat{T})$ is equivalent to $\scrZ$ resp.~$\scrJ$.
	\subsection{Generalities on coherent sheaves}
	
	Throughout this section, we continue to assume $\La$ is an algebraic extension of $\bbQ_\ell$ and we let $X=Y/H$ be the quotient stack of a finitely presented quasi-affine $\La$-scheme acted upon by a reductive group $H$ over $\La$. Let us recall how to define the derived category $\scrD_{\mathrm{qc}}(X)$ of quasi-coherent sheaves on $X$. Recall that the category $\mathrm{Mod}_Y$ of $\calO_Y$-module sheaves is Grothendieck abelian in the sense of \cite[Definition 1.3.5.1]{Lur17}. By \cite[Definition 1.3.5.8]{Lur17}, this abelian category induces a stable $\infty$-category $\scrD(\mathrm{Mod}_Y)$ of $\calO_Y$-modules on $Y$. It is naturally endowed with a t-structure in the sense of \cite[Definition 1.2.1.4]{Lur17} defined by non-vanishing degrees of its cohomology functors, see \cite[Definition 1.3.5.16]{Lur17}. Hence, we can define $\scrD_{\mathrm{qc}}(Y)$ (resp.~$\scrD_{\mathrm{coh}}(Y)$)  as the full subcategory spanned by complexes whose cohomologies are quasi-coherent (resp.~coherent) $\calO_Y$-modules. We now define $\scrD_{\mathrm{qc}}(X)$ (resp.~$\scrD_{\mathrm{coh}}(X)$) as the limit of the simplicial object in $\infty$-categories $[n]\mapsto \scrD_{\mathrm{qc}}(Y^{n}_X)$ (resp.~$[n]\mapsto \scrD_{\mathrm{coh}}(Y^n_X)$). The resulting $\infty$-categories are stable and carry natural t-structures, whose hearts will be denoted $\mathrm{QCoh}(X)$, resp.~$\mathrm{Coh}(X)$. 
	
	We will decorate the derived categories by the superscripts $^b,^+,^-$ to denote the full subcategories of bounded, left-bounded, and right-bounded complexes. The full subcategory $\mathrm{Perf}(Y)$ of perfect complexes is spanned by bounded complexes with finite Tor-amplitude (i.e., those which are represented by finite complexes of vector bundles as $Y$ is quasi-affine), and we define $\mathrm{Perf}(X)$ again by descent.
	
	Notice that the Grothendieck abelian category $\mathrm{QCoh}(X)$ induces a stable $\infty$-category $\scrD(\mathrm{QCoh}(X))$ again by an application of \cite[Definition 1.3.5.8]{Lur17} with a natural t-structure. It will be often useful to relate this to $\scrD_\mathrm{qc}(X)$. There is an induced t-exact functor $\scrD(\mathrm{QCoh}(X)) \to \scrD_{\mathrm{qc}}(X)$ and under our assumptions, we get:
	
	\begin{proposition}\label{prop_isom_der_coh}
		The functor $\scrD(\mathrm{QCoh}(X)) \to \scrD_{\mathrm{qc}}(X)$ is an equivalence.
	\end{proposition}
	\begin{proof}
		Essential surjectivity can be tested at the triangulated level, i.e., by taking homotopy categories. Similarly, full faithfulness amounts to checking isomorphism of homotopy groups of mapping spaces, which can be expressed in terms of Ext groups by \cite[Notation 1.1.2.17]{Lur17}, so we can also verify it at the triangulated level. Since $\scrD_{\mathrm{qc}}(X)$ is compactly generated by \cite[Theorem B]{HR17}, the claim now follows from \cite[Theorem 1.2]{HNR19}.
		Alternatively, one can invoke \cite[Proposition A.1.2]{Man22}, which applies as $\mathrm{QCoh}(X)$ is Grothendieck, $\scrD(\mathrm{QCoh}(X))$ is left-complete (this uses that reductive groups in characteristic $0$ have a semisimple representation theory, and in particular $\mathrm{QCoh}(X)$ has finite cohomological dimension), and the morphisms for the \v{C}ech nerve for the covering $Y\to X$ are flat and affine.
	\end{proof}
In order to understand right-bounded complexes in the affine case, the following lemma is decisive.

\begin{lemma}\label{lem_enough_proj}
	If $Y$ is affine, then the abelian category $\mathrm{Coh}(X)$ has enough projectives. In particular, the t-exact functor $\scrD^-(\mathrm{Coh}(X))\to \scrD^-_{\mathrm{coh}}(X)$ is an equivalence.
\end{lemma}

\begin{proof}
	Let $R$, resp.~$A$ be the ring of global sections on $Y$ resp.~$H$. The category of finitely generated $R$-modules (which is equivalent to $\mathrm{Coh}(Y)$) has enough projectives by considering the collection of free $R$-modules. Notice that the functor of taking $H$-invariants on $R$-modules is exact by assumption on $H$. We can deduce that the $H$-equivariant free $R$-module $V\otimes M$ with $V$ being a finite dimensional representation of $H$ is projective in $\mathrm{Coh}(X)$. This collection of projectives is enough, as each coherent sheaf on $X$ is surjected upon by the $n$-fold sum of the regular representation $A\otimes R$ for $n\gg 0$, and we can find a finite representation $V \subset A$ which completes the job, by finiteness of the underlying $R$-module of the initial coherent sheaf on $X$.
\end{proof}


If the stack $X$ is smooth, then we actually get an equality 
\begin{equation}\label{equation_coh_VS_perf}
	\scrD^b_{\mathrm{coh}}(X)=\mathrm{Perf}(X)    
\end{equation}
of full subcategories. This motivates our construction of the AB functor via the following equivariant analogue of the localization theorem originally due to Thomason--Trobaugh \cite{TT90} and Neeman \cite{Nee92}.

\begin{proposition}\label{prop_thomason_quotient}
	Let $U \subset X$ be an open immersion with closed complement $Z$. Then $\mathrm{Perf}(U)$ is the idempotent-completion of the quotient $\mathrm{Perf}(X)/\mathrm{Perf}(X)_Z$,
	\footnote{The quotient $\mathrm{Perf}(X)/\mathrm{Perf}(X)_Z$ 
		denotes the Verdier quotient as defined in \cite[Definition A.1.4]{CDH+20} or \cite[Theorem I.3.3]{NS18}. Another way to phrase it is that $\mathrm{Perf}(U)$ is the Karoubi quotient of $\mathrm{Perf}(X)$ 
		by $\mathrm{Perf}(X)_Z$ in the sense of \cite[Definition A.3.6]{CDH+20}.}
	where the denominator indicates the full subcategory spanned by complexes supported in $Z$.
\end{proposition}

\begin{proof}
	This is \cite[Theorem 3.4, Equation (3.6)]{KR18} for the underlying triangulated categories, which implies the statement in general. Let us explain how one obtains the result. First, it is clear that $\scrD_{\mathrm{qc}}(U)$ is a localization of $\scrD_{\mathrm{qc}}(X)$ with kernel $\scrD_{\mathrm{qc}}(X)_Z$, because restriction admits a right adjoint given by pushforward with unit being an equivalence. Finally, since each of the categories involved are compactly generated by \cite[Theorem B]{HR17} with compact objects given exactly by perfect complexes by \cite[Lemma 4.4]{HR17}, we can apply the localization theorem, see \cite[Theorem 3.12]{HR17}, to obtain the claim.
\end{proof}



\subsection{Coherent sheaves on the Springer variety}
Recall that the variety $\hat{G}/\hat{U}$, which is a $\hat{T}$-torsor over $\hat{G}/\hat{B}$, is quasi-affine\footnote{By the construction of quotients via fixed vectors in representations, any quotient of an affine scheme of finite type over a field by a unipotent group scheme is quasi-affine. For details on $\hat{G}/\hat{U}$ see \cite[Subsection 6.2.1]{AR}.} , so it embeds openly in the spectrum $\hat{\calX}$ of its global sections $\calO(\hat{G}/\hat{U})$. In turn, these admit the following explicit description as a graded $\Lambda$-algebra
\begin{equation}
	\calO(\hat{G}/\hat{U})=\bigoplus_{\mu \in \bbX_\bullet^+} V_\mu
\end{equation}
where $V_\mu$ denotes the highest weight representation of highest weight $\mu$ and multiplication is given by the obvious maps $V_{\mu_1} \otimes V_{\mu_2} \to V_{\mu_1+\mu_2}$, see \cite[Lemma 1.5.1,~Section 1.5.2,~Lemma 6.2.1]{AR}. In addition, the above $\Lambda$-algebra is finitely generated cf.\ \cite[Proposition 1.2.1]{Ful93}. 

Similarly, we can define the following $\hat{T}$-torsor
\begin{equation}
	\hat{\calN}_{\mathrm{Spr}}^{\mathrm{qaf}}=\hat{G}\times^{\hat{U}}\mathrm{Lie}(\hat{U})
\end{equation}
over the Springer resolution, which is a quasi-affine scheme with an action of $\hat{G}^\prime:=\hat{G}\times\hat{T}$. 
The Lie algebra $\hat{\frakg}$ of $\hat{G}$ acts naturally via derivations on the structure sheaf of $\hat{G}/\hat{U}$, see \cite[Equation (6.2.9)]{AR} and we can associate to it the so-called infinitesimal universal stabilizer $\hat{\calN}_{\mathrm{Spr}}^{\mathrm{af}} $ as the closed subscheme of $\hat{\frakg} \times \hat{\calX}$ given by the image of the derivation map cf.\ \cite[Section 6.2.2]{AR}. Note that, even though the intersection of $\hat{\calN}_{\mathrm{Spr}}^{\mathrm{af}}$ with $ \hat{\frakg} \times \hat{G}/\hat{U}$ is exactly $\hat{\calN}^{\mathrm{qaf}}_{\mathrm{Spr}}$, it is not generally true that $\hat{\calN}_{\mathrm{Spr}}^{\mathrm{af}}$ coincides with the scheme-theoretic closure of the locally closed immersion $ \hat{\calN}^{\mathrm{qaf}}_{\mathrm{Spr}} \to \hat{\frakg} \times \hat{\calX}$. 
The latter is an integral variety admitting $\hat{\calN}_{\mathrm{Spr}}^{\mathrm{qaf}}$ as a dense open subset, 
with ideal of definition given by the kernel of $\calO(\hat{\frakg})\times \calO(\hat{\calX}) \to \calO(\hat{\calN}^{\mathrm{qaf}}_{\mathrm{Spr}})$.

We have two distinguished collections of generators for the derived category of $[\hat{G}\backslash \hat{\calN}_{\mathrm{Spr}}]$.

\begin{lemma}
	The derived category $\scrD^b_{\rm{coh}}([\hat{G}\backslash \hat{\calN}_{\mathrm{Spr}}])$ is spanned by the set of the line bundles $\calO(\nu)$ for $\nu \in \bbX_\bullet$, or by the set of the vector bundles $V\otimes \calO(\nu)$ for $V\in {\rm{Rep}}\,\hat{G}$ and $\nu \in \bbX_\bullet^+$.   
\end{lemma}

\begin{proof}
	This is \cite[Lemma 21]{AB09}. See also \cite[Lemma 6.2.8]{AR}.
\end{proof}

Originally, it was claimed in \cite[Lemma 20]{AB09} that the triangulated category $\scrD^b_{\rm{coh}}([\hat{G}\backslash \hat{\calN}_{\mathrm{Spr}}])$ is a Verdier quotient of ${\mathrm{Perf}}([\hat{G}'\backslash \hat{\calN}_{\mathrm{Spr}}^{\mathrm{af}}])$. Upon expanding the argument in \cite[Proposition 6.2.8]{AR}, we noticed that it seemed to rely on density of $\hat{\calN}_{\mathrm{Spr}}^{\mathrm{qaf}} \subset \hat{\calN}_{\mathrm{Spr}}^{\mathrm{af}} $, which unfortunately fails in general. Instead, we will argue below via \cite[Remark 6.3.10]{AR}.

\subsection{Construction of the AB functor}
As in this whole section, we assume that $\calI$ is the standard Iwahori attached to the fixed Borel $B$ of the pinned split group $G$. We recall also the notation $G'=G\times T$ and $\hat{G}'=\hat{G}\times \hat{T}$. First, we start with the functor
\begin{equation}
	\scrZ':=\scrZ \times \scrJ \colon \on{Rep}_\Lambda (\hat{G}') \to \scrD_{\mathrm{ula}}(\Hk_{\calI})
\end{equation}
which has a natural monoidal structure\footnote{Even in the $\mathbb{E}_1$-monoidal sense: The functors $\scrZ$ and $\scrJ$ are $\mathbb{E}_1$-monoidal, and $\scrZ$ is central, see \Cref{thm_assoc_center}. This implies the existence of $\scrZ'$ by the definition of $\mathbb{E}_1$-centers, see \cite[Definition 5.3.1.2]{Lur17}.}and factors through the full subcategory of $ \scrP(\mathrm{Wak})$ consisting of Wakimoto-filtered perverse sheaves. However, this is still not good enough, because the convolution of Wakimoto-filtered perverse sheaves is \textit{not} {symmetric} in general.

In order to fix this, we consider the (non-full!) subcategory $\scrC$ of $ \scrP(\Hk_{\calI})$ whose objects are those in the image of $\scrZ'$ and whose morphisms commute with the images along $\scrZ'$ of the symmetry isomorphisms of $\Rep_\La(\hat{G}')$. This is a symmetric monoidal category
by definition, see \cite[Lemma 6.3.3]{AR}. Consider the following $\La$-algebra
\begin{equation}
	A=\mathrm{Hom}_{\mathrm{Ind}(\scrC)}(1_\scrC,\scrZ'(\calO(\hat{G}')))
\end{equation}
where the multiplication is induced by that of the group $\hat{G}'$ and the monoidal structure of $\scrZ'$, and $\calO(\hat{G}^\prime)$ is a $\hat{G}^\prime$-representation via left multiplication (see the discussion before \cite[Proposition 6.3.5]{AR}). By \cite[Proposition 6.3.5]{AR}, this defines an identification between $\scrC$ and the category of free $A$-modules with $\hat{G}'$-equivariant structure of the form $V\otimes_\Lambda A$ where $V$ is a finite dimensional $\Lambda$-representation of $\hat{G}'$.

Next, we construct a $\La$-algebra homomorphism
\begin{equation}
	\calO(\hat{\calN}_{\mathrm{Spr}}^{\mathrm{af}}) \to A
\end{equation}
that is equivariant with respect to the $\hat{G}'$-module structures. Via the $\hat{G}^\prime$-equivariant embedding  $\hat{\calN}_{\mathrm{Spr}}^{\mathrm{af}}\to \frakg \times \hat{\calX}$, we start by handling each of these two factors separately (following closely the respective part in \cite[Section 6.3]{AR}).

For any $\hat{G}$-representation $V$, we extend it to a $\hat{G}'$-representation $V'=V\boxtimes 1$ by letting $\hat{T}$ act trivially and consider the logarithm of the monodromy $n_V$ acting on $\scrZ(V)=\scrZ'(V')$. The collection of these endomorphisms defines a map of $\La$-algebras $\calO(\hat{\frakg}) \to A$ which is $\hat{G}'$-equivariant. For details we refer to \cite[Example 6.3.1]{AR} and \Cref{D_P_for_n_V}.

Next, we need to define a map of $\hat{G}'$-modules $V_\nu \boxtimes -\nu\to A $ and the natural source for this is the highest weight arrow $f_\nu$ provided by the Wakimoto filtration, see \Cref{sec:high-weight-arrows}. It defines a morphism in $\scrC$ by the already checked compatibilities, so applying the description of $\scrC$ in terms of $A$ yields a map $V_\nu' \otimes A \to 1\boxtimes \nu\otimes A$ which corresponds to our goal after twisting by $\nu$ and restricting the domain on the left.

In total, we have thus constructed a $\hat{G}'$-module homomorphism
\begin{equation}\label{eq_equiv_hom_lie_afflag}
	\calO(\hat{\frakg} \times \hat{\calX}) \to A.
\end{equation}
However, we are still left with the task of showing that this factors over the coordinate ring of the affine enlargement $\hat{\calN}_{\mathrm{Spr}}^{\mathrm{af}}$, which is a closed subscheme of $\hat{\frakg}\times \hat{\calX}$ of the Springer bundle.

\begin{lemma}
	The $\hat{G}'$-equivariant map \eqref{eq_equiv_hom_lie_afflag} factors uniquely through a $\hat{G}'$-equivariant map
	\begin{equation}     \calO(\hat{\calN}_{\mathrm{Spr}}^{\mathrm{af}}) \to A.
	\end{equation}
\end{lemma}

\begin{proof}
	Here, we follow \cite[Lemma 6.3.7]{AR}. We know that the highest weight arrow is equivariant with respect to the monodromy operator. Passing to the logarithm, we see as in \Cref{sec:high-weight-arrows-1-monodromy-is-zero-on-highest-weight-vector} that $f_\nu \circ \mathbf{n}_\nu=0$. This equality holds true in the auxiliary category $\scrC$ (in fact, the monodromy action on $\scrZ$ factors through $\scrC$). 
	Comparing with \cite[Example 6.3.1]{AR}, we conclude from this identity that the definition ideal of $\hat{\calN}_{\mathrm{Spr}}^{\mathrm{af}}$ inside $\frakg \times \hat{\calX}$ vanishes under the map to $A$.
\end{proof}

So far, we have arrived at a functor \begin{equation}
	\tilde{F}\colon\mathrm{Coh}_{\mathrm{fr}}([\hat{G}'\backslash \hat{\calN}_{\mathrm{Spr}}^{\mathrm{af}}]) \to \scrC,
\end{equation} where $\mathrm{Coh}_{\mathrm{fr}}$ denotes the full subcategory of $\mathrm{Coh}([\hat{G}'\backslash \hat{\calN}_{\mathrm{Spr}}^{\mathrm{af}}])$ spanned by the objects $V\otimes \mathcal{O}$ for $V\in \mathrm{Rep}_{\Lambda}(\hat{G})$. Now we are going to show that this functor passes to the actual Springer resolution $\hat{\calN}_{\mathrm{Spr}}$.

\begin{lemma}\label{lem_gr_circ_F_restr}
	The functor $\tilde{F}$ composed with the Wakimoto grading functor $\mathrm{gr}$ from \Cref{sec:wakimoto-sheaves-graded-wakimoto-functor} identifies with the pullback functor of coherent sheaves along the morphism
	\[
	[\hat{T}\backslash e]\to  [\hat{G}\backslash \hat{\calN}_{\mathrm{Spr}}]\cong [\hat{G}'\backslash \hat{\calN}_{\mathrm{Spr}}^{\mathrm{qaf}}]\subseteq [\hat{G}'\backslash \hat{\calN}_{\mathrm{Spr}}^{\mathrm{qaf}}].
	\]
\end{lemma}
Here, $e:=\Spec(\Lambda)\to \hat{\calN}_{\mathrm{Spr}}=\hat{G}\times^{\hat{B}} \mathrm{Lie}(\hat{U})$ denotes the point $[(1,0)]$.
\begin{proof}
	We follow the proof in \cite[Lemma 6.3.8]{AR}. It suffices to understand the corresponding $\hat{T}$-equivariant map of $\La$-algebras $\calO(\hat{\calN}_{\mathrm{Spr}}^{\mathrm{af}})\to \La$. But the monodromy acts trivially on the Wakimoto grading as we saw in \Cref{unipotent_new}, and the highest weight arrow is projects to $V_\la$ to the highest weight space $V_\la(\la)$. Hence the sought homomorphism is just evaluation at the origin $e$.
\end{proof}

\begin{proposition}\label{prop_constr_AB_funct_F}
	There is a unique monoidal functor of stable $\infty$-categories up to equivalence \begin{equation}
		\scrF\colon \mathrm{Perf}([\hat{G}\backslash \hat{\calN}_{\mathrm{Spr}}]) \to \scrD_{\mathrm{cons}}(\Hk_{\calI})
	\end{equation} extending $\tilde{F}$.
\end{proposition}

\begin{proof}
	This is \cite[Proposition 6.3.9, Remark 6.3.10]{AR} in the triangulated setting and we follow their argument. 
	
	Since $\mathrm{Coh}_{\mathrm{fr}}([\hat{G}'\backslash \hat{\calN}_{\mathrm{Spr}}^{\mathrm{af}}])$ consists of compact, projective generators of $\mathrm{Coh}([\hat{G}'\backslash \hat{\calN}_{\mathrm{Spr}}^{\mathrm{af}}])$ by \Cref{lem_enough_proj}, left Kan extension of the composition $\mathrm{Coh}_{\mathrm{fr}}([\hat{G}'\backslash \hat{\calN}_{\mathrm{Spr}}^{\mathrm{af}}])\overset{\tilde{F}}{\to} \mathcal{C}\to \mathcal{D}_\et(\Hk_\calI)$ yields the exact ``left-derived'' functor  $L\tilde{F}\colon \scrD^{\leq 0}(\mathrm{Coh}([\hat{G}'\backslash \hat{\calN}_{\mathrm{Spr}}^{\mathrm{af}}]))\to \scrD_{\et}(\Hk_\calI)$ as in \cite[Theorem 1.3.3.2]{Lur17}. This functor formally extends to an exact functor on $\scrD^-(\mathrm{Coh}([\hat{G}'\backslash \hat{\calN}_{\mathrm{Spr}}^{\mathrm{af}}]))$. Thanks to the equivalence from \Cref{prop_isom_der_coh} and after restricting to perfect complexes, we get a functor
	\begin{equation}
		\tilde{\scrF}: \mathrm{Perf}([\hat{G}'\backslash \hat{\calN}_{\mathrm{Spr}}^{\mathrm{af}} ])\to \scrD_{\mathrm{cons}}(\Hk_{\calI}).
	\end{equation}
	This functor is monoidal because it can be written as the composition of the monoidal functor $\scrC^b(\tilde{F})$, where $\scrC^b$ denotes the associated $\infty$-category of complexes, followed by the restricted realization functor $\scrC^b(\scrP(\mathrm{Wak}) ) \to \scrD_{\mathrm{cons}}(\Hk_{\calI})$, which is monoidal because $\scrP(\mathrm{Wak})$ is a full subcategory of $\scrD_{\mathrm{cons}}(\Hk_{\calI})$ closed under convolution.
	\Cref{prop_thomason_quotient} implies that the category $\mathrm{Perf}([\hat{G}\backslash \hat{\calN}_{\mathrm{Spr}}])$ is the idempotent-completion of the quotient of $\mathrm{Perf}([\hat{G}'\backslash \hat{\calN}^{\mathrm{af}}_{\mathrm{Spr}}])$ by the full subcategory of those perfect complexes supported on the complement. 
	Since $\scrD_{\mathrm{cons}}(\Hk_{\calI})$ is idempotent-complete, we are reduced to showing that such a perfect complex lies in the kernel of $\tilde{\scrF}$. 
	Since the image of $\tilde{\scrF}$ lies in $\mathrm{Wak}$ (this reduces to the same statement for $\tilde{F}$ as $\mathrm{Wak}$ is idempotent-complete), we can check acyclicity after passing to gradeds by first taking Wakimoto filtrations termwise and then inducting. But the grading functor corresponds at the coherent level to restriction to the origin of $\hat{\calN}_{\mathrm{Spr}}$ by \Cref{lem_gr_circ_F_restr}, hence the desired vanishing holds. 
\end{proof}
\section{Iwahori--Whittaker averaging}

We continue to work with a pinned split $F$-group $G$ with a fixed maximal torus $T$ and a Borel $B$ containing $T$ (in particular, we do not regard them as being defined over $O$, unless indexed by $O$). We let $\calI$ be the standard Iwahori $O$-model of $G$, i.e., such that $\calI(O)$ fixes the origin of the apartment $\scrA(G,T)$ induced by the pinning and is contained in the $B$-dominant Weyl chamber.

We let $\calI^{\mathrm{op}}$ denote the parahoric $O$-model 
opposite to $\calI$ with respect to the origin of the apartment and the Borel, and simply call it the \textit{opposite Iwahori}. In other words, $\calI^{\mathrm{op}}(O)$ fixes the alcove opposite to the one fixed by $\calI(O)$. We define likewise the \textit{pro-$p$ Iwahori} $\calI_{\mathrm{u}}$ as the unique smooth affine $O$-model of $G$ with connected geometric fibers whose $O$-valued points are the kernel of $\calI(O)\to \calI_k^{\mathrm{red}}(k) $, where the $\calI_k^{\mathrm{red}}$ is the reductive quotient of the special fiber of $\calI$. We have a natural action of $L^+\calI_\mathrm{u}^\mathrm{op}$ on $\Fl_\calI$ be left multiplication.

Our next task is to choose a Whittaker datum. Assume $\Lambda$ is an algebraic field extension of $\bbQ_\ell(\zeta)$ where $\zeta \in \bar \bbQ_\ell$ is a choice of a primitive $p$-th root of unity. We get the Artin--Schreier étale sheaf $\calL_{\mathrm{AS}}$ on $\bbG_{a,k}$: this is the rank $1$ direct summand of the pushforward $\pi_*\La$ of the constant sheaf along the Artin--Schreier cover $\pi\colon \bbG_{a,k} \to \bbG_{a,k}$ arising as the $\zeta$-eigenspace for the Galois action of $\bbZ/p\bbZ$. It is a character sheaf in the sense of Lusztig--Yun \cite{LY20}, i.e., we have isomorphisms $m^*\calL_{\mathrm{AS}} \simeq \calL_{\mathrm{AS}}\boxtimes \calL_{\mathrm{AS}}$ and $e^*\calL_{\mathrm{AS}}\cong \La$ with respect to the multiplication $m$ and unit $e$ of $\bbG_{a,k}$, that satisfy associativity constraints (this is equivalent to the corresponding $\infty$-enhancement, because $\scrP(\bbG_{a,k})$ is an abelian category). 
Besides, it satisfies the following cohomological vanishing cf.\ \cite[Theorem 2.7]{Del77} \begin{equation}
	R\Gamma(\bbG_{a,k}, \calL_{\mathrm{AS}})=R\Gamma_c(\bbG_{a,k}, \calL_{\mathrm{AS}})=0,
\end{equation}
which will turn out to be important later on.

Let $U^{\mathrm{op}}$ be the unipotent radical of the opposite Borel $B^{\mathrm{op}}$. Consider the homomorphism $U_k^{\mathrm{op}} \to \bbG_{a,k}$ induced by the 
sum of the negative simple roots and let $\chi: L^+\calI^{\mathrm{op}}_u \to \bbG_{a,k}$ be the homomorphism resulting from pre-composing the first one with the natural projection $L^+\calI^{\mathrm{op}}_u \to U_k^{\mathrm{op}}$. Taking the pullback of $\calL_{\mathrm{AS}}$ along $\chi$, we get a character sheaf on $\calL_{\calI\calW} \in \scrD_{\mathrm{cons}}(L^+\calI_u^{\mathrm{op}})$. Indeed, this is the character sheaf attached to the cover $\pi_\chi \colon (L^+\calI_u^{\mathrm{op}})_{\mathrm{AS}} \to L^+\calI_u^{\mathrm{op}}$ deduced from $\pi\colon \bbG_{a,k} \to \bbG_{a,k}$ by pullback along $\chi$.

\begin{definition}
	The derived category $\scrD_{\mathrm{\acute{e}t}}(\Hk_{\calI\calW})$ of \textit{Iwahori-Whittaker sheaves} is the $\zeta$-isotypical component of the stable $\infty$-category $\scrD_{\mathrm{\acute{e}t}}([(L^+\calI_u^{\mathrm{op}})_{\mathrm{AS}}\backslash \Fl_{\calI}])$.
\end{definition}

In the above definition, we use the fact that the $\La$-linear stable $\infty$-category $\scrD_{\mathrm{\acute{e}t}}([(L^+\calI_u^{\mathrm{op}})_{\mathrm{AS}}\backslash \Fl_{\calI}])$ has a $\bbZ/p\bbZ$-action coming from $\bbZ/p\bbZ \simeq \mathrm{ker}(\pi_\chi)$ and that it decomposes as a direct sum of full subcategories where $\bbZ/p\bbZ$ acts via a $\La^\times$-valued character, since $\La$ has characteristic $0$ and $\mu_p\subseteq \La$.
Note that no underlying stack $\mathrm{Hk}_{\calI\calW}$ seems to exist, but we find this shorthand notation useful, and hope it does not cause any confusion to the reader. We could also define $\scrD_{\mathrm{\acute{e}t}}(\Hk_{\calI\calW})$ as the $\infty$-category of $(L^+\calI_{\mathrm{u}}^{\mathrm{op}},\calL_{\calI\calW})$-equivariant étale sheaves on $\Fl_{ \calI}$,
obtained by twisting by the character sheaf $\calL_{\calI\calW}$ the cosimplicial diagram obtained by applying $\scrD_{\mathrm{\acute{e}t}}$ to the \v{C}ech complex of $\Fl_{ \calI} \to [L^+\calI_{\mathrm{u}}^{\mathrm{op}}\backslash \Fl_{ \calI}]$. In the end, it turns out that all of this is unnecessarily complicated, because:

\begin{proposition}
	The forgetful functor $\scrD_{\mathrm{\acute{e}t}}(\Hk_{\calI\calW})\to \scrD_{\mathrm{\acute{e}t}}(\Fl_{\calI})$ is fully faithful.
\end{proposition}
\begin{proof}
	This is essentially \cite[Proposition A.5]{ARW16} and follows from the fact that $L^+\calI_u^{\mathrm{op}}$ is pro-unipotent and hence so is its Artin--Schreier cover. Thus, we can apply \cite[Proposition VI.4.1]{FS21}.
\end{proof}
The category of Iwahori--Whittaker sheaves inherits a perverse t-structure from its fully faithful embedding into $\scrD_{\mathrm{\acute{e}t}}(\Fl_{\calI})$, so that one can consider its heart $\scrP(\Hk_{\calI\calW})$, called the category of Iwahori--Whittaker perverse sheaves. 		
The $\infty$-category $\scrD_{\mathrm{\acute{e}t}}(\Hk_{\calI\calW})$ does not appear to be monoidal, but it is a right module of $\scrD_{\mathrm{\acute{e}t}}(\Hk_{\calI})$ in the sense of \cite[Definition 4.2.1.13]{Lur17}. Indeed, we invoke the natural isomorphism
\begin{equation}
	[(L^+\calI_{\mathrm{u}}^{\mathrm{op}})_{\mathrm{AS}}\backslash \Fl_{ \calI}] \simeq [*/(L^+\calI_{\mathrm{u}}^{\mathrm{op}})_{\mathrm{AS}}]\times_{[*/LG]} [*/L^+\calI] 
\end{equation} to identify our stack as a homomorphism object in $\mathrm{Corr}(\mathrm{vSt}_{/[*/LG]})$ with a natural right module structure under the endomorphism object $\Hk_{\calI}\cong [*/L^+\calI]\times_{[*/LG]} [*/L^+\calI]$. Taking the symmetric monoidal forgetful functor $\mathrm{Corr}(\mathrm{vSt}_{[\ast/LG]})\to \mathrm{Corr}(\mathrm{vSt})$ and applying $\scrD_{\mathrm{\acute{e}t}}^{\otimes}$, we deduce a right module structure on $\scrD_{\mathrm{\acute{e}t}}([L^+\calI_{\mathrm{u}}^{\mathrm{op}}\backslash \Fl_{ \calI}])$ under $\scrD_{\mathrm{\acute{e}t}}(\Hk_{\calI})$ and this module structure is preserved under passing to direct summands.
We begin our study of this $\infty$-category by classifying Iwahori--Whittaker i.e. $(L^+\calI_{\mathrm{u}}^{\mathrm{op}},\calL_{\calI\calW})$-equivariant, local systems on $L^+\mathcal{I}^{\mathrm{op}}_u$-orbits in $\Fl_\calI$. Note that the latter are exactly the $L^+\calI^{\mathrm{op}}$-orbits and hence are in bijection with the Iwahori--Weyl group $W$ as in \Cref{sec:geometry-affine-flag-1-stratification-of-affine-flag-variety}. We will denote the corresponding $L^+\calI^{\mathrm{op}}$-orbit of $w$ by $\Fl_{\calI,w}^{\mathrm{op}}:=L^+\calI^{\mathrm{op}}wL^+\calI/L^+\calI$.

\begin{lemma}
	\label{sec:iwah-whitt-aver-support-of-iwahori-whittaker-sheaves}
	The orbit $\Fl^{\mathrm{op}}_{\calI,w} $ carries an Iwahori--Whittaker local system if and only if $w$ has minimal length in its $W_{\mathrm{fin}}$-left coset, i.e., $\ell(w_{\mathrm{fin}}w)\geq \ell(w)$ for all $w_{\mathrm{fin}}\in W_{\mathrm{fin}}$. If the latter condition holds, then the rank $1$ Iwahori--Whittaker local system on $\Fl^{\mathrm{op}}_{\calI,w} $ is unique up to isomorphism.
\end{lemma}

\begin{proof}
	The Iwahori--Whittaker equivariant condition forces the stabilizer of the point $\dot{w}$ to be contained in the kernel of $\chi$, and conversely such a containment would allow us to pullback the Artin--Schreier sheaf along $\chi$ to the desired orbit. This inclusion happens if and only if $w(\alpha_s)$ is a positive affine root where $\alpha_s$ is the positive simple affine root attached to any positive simple reflection $s \in W_{\mathrm{fin}}$. But this is equivalent to $sw>w$, i.e., that $w$ is the minimal length representative of its $W_{\mathrm{fin}}$-left coset.
\end{proof}

Since the set of left $W_\mathrm{fin}$-cosets of the Iwahori--Weyl group is enumerated by $\bbX_\bullet$, we will call $\calL_{\nu}$ the unique Iwahori--Whittaker local system supported on the $L^+\mathcal{I}_u^{\mathrm{op}}$-orbit of $w_\nu$, the minimal length element in $W_\mathrm{fin}t_{\nu}$, according to the preceding statement. 
We also obtain the standard Iwahori--Whittaker equivariant sheaf
\begin{equation}
	\Delta_\nu^{\calI\calW} := (j^{\mathrm{op}}_{w_\nu})_!\calL_{\nu}[\ell(w_\nu)],
\end{equation}
where $j^{\mathrm{op}}_{w_\nu}$ is the inclusion of the $L^+\calI^{\mathrm{op}}$-orbit and likewise the costandard Iwahori--Whittaker equivariant sheaf
\begin{equation}
	\nabla_\nu^{\calI\calW} := R(j^{\mathrm{op}}_{w_\nu})_*\calL_{\nu}[\ell(w_\nu)],
\end{equation}
both of which are supported on $\Fl_{\calI,\leq w_\nu}^{\mathrm{op}}$ and are perverse because orbits of solvable groups are affine, so we can invoke Artin vanishing, compare with \cite[Corollaire 4.1.10]{BBDG18}. We also have access to $\mathrm{IC}$ sheaves $\mathrm{IC}_\nu^{\calI\calW}$ by taking the image of the natural map $\Delta_\nu^{\calI\calW} \to \nabla_\nu^{\calI\calW}$. Recall that in \cite[Subsection 3.2]{BGS96} a sufficient criterion for the existence of projective covers, injective hulls and tilting modules was given. We call an abelian category satisfying these axioms a highest weight category, see also \cite[Subsection 1.12.3]{BR18}.

\begin{proposition}\label{prop_IW_category_is_highest_weight}
	The category $\scrP(\Hk_{\calI\calW})$ of Iwahori--Whittaker equivariant perverse sheaves is a highest weight category, whose underlying poset equals $\bbX_\bullet \simeq W_{\mathrm{fin}}\backslash W$ ordered by the quotient Bruhat order.
\end{proposition}

\begin{proof}
	The first part is a standard consequence of arguments by Beilinson--Ginzburg--Soergel, see \cite[Theorems 3.2.1 and 3.3.1]{BGS96}. As for the second claim, it suffices to identify the closure relations. It can be easily checked that the opposite Schubert variety $\Fl_{\calI,\leq w_\nu}^{\mathrm{op}}$ equals the $(G,\calI)$-Schubert variety $\Fl_{ (G,\calI),\leq \nu}$ in the notation of \cite[Section 3]{AGLR22} (up to the order of action), which in turn coincides with $\Fl_{ \calI,\leq w_0w_\nu}$, where $w_0 \in W_{\mathrm{fin}}$ is the longest element (there is a notational clash here, because $w_\nu$ evaluated at $\nu=0$ is simply the identity). Indeed, they have the same dimension and there is an obvious inclusion $\Fl_{\calI,\leq w_\nu}^{\mathrm{op}} \subset \Fl_{ (G,\calI),\leq \nu}$, because $L^+G \supset L^+\calI^{\mathrm{op}}W_{\mathrm{fin}}$. The closure relations follow then from the usual combinatorics of flag varieties as in \cite[Section 3]{AGLR22}. Indeed, the variety $\Fl_{ (G,\calI),\leq \nu_1}$ is contained in $\Fl_{ (G,\calI),\leq \nu_2}$ if and only if $\nu_1\leq \nu_2$ for the quotient Bruhat order.
\end{proof}

Notice that $\Fl_{\calI,\leq 1}^{\mathrm{op}}=\Fl_{ \calI,\leq w_0}=G/B \subset \Fl_\calI$, with opposite cells indexed by $W_{\mathrm{fin}}$. This implies the equalities $\Delta_0^{\calI\calW}=\nabla_0^{\calI\calW}=\mathrm{IC}_0^{\calI\calW}$ and we will denote this simple Iwahori-Whittaker perverse sheaf by $\Xi$. It allows us to define the Iwahori--Whittaker averaging functor
\begin{equation}
	\mathrm{av}_{\calI\calW}\colon \scrD_{\mathrm{\acute{e}t}}(\Hk_{\calI}) \to \scrD_{\mathrm{\acute{e}t}}(\Hk_{\calI\calW})
\end{equation}
given by $A \mapsto \Xi\ast A$. To get a better understanding of this functor, we start by the following calculation:

\begin{lemma} \label{Lemma_IW_averaging_kills}
	The sheaf $\mathrm{av}_{\calI\calW}(\mathrm{IC}_w)$ equals $\mathrm{IC}_\nu^{\calI\calW}$ if $w= w_\nu$ for some $\nu\in \bbX_\bullet$ and vanishes otherwise.
\end{lemma}

\begin{proof}
	If $sw<w$ for some simple reflection $s\in W_{\mathrm{fin}}$, we conclude that $\mathrm{IC}_w$ is equivariant for the left action of $L^+\calP_s$, where $\calP_s$ is the minimal standard parahoric with respect to $\calI$ and the simple reflection $s$. Notice that by \cite[Lemma 2.5]{BGMRR19} we have a natural isomorphism
	\begin{equation}\Xi \ast^{\calI} \mathrm{IC}_w \simeq (R\pi_{s,*}\Xi)\ast^{\calJ_s} \mathrm{IC}_w
	\end{equation} where the exponent stands for the fact that the first convolution is induced by contracting the $L^+\calI$-action, and the second one by contracting the $L^+\calJ_s$-action. Here, $\pi_s:\Fl_\calI\to \Fl_{\calP_s}$ is the projection. In particular, it suffices to check the vanishing of $R\pi_{s,*}\Xi$. Note that $\Xi$ is supported on $U^{\mathrm{op}}w_0 \subset G/B$ as the complement cannot support a non-zero Iwahori-Whittaker sheaf by \Cref{sec:iwah-whitt-aver-support-of-iwahori-whittaker-sheaves}. We can now see that the fiber of $\pi_s$ over the image of $\mathrm{supp}\,\Xi$ is isomorphic to $U_{a}^{\mathrm{op}}$, where $a$ is the positive root associated with $s$. Since $\chi$ does not vanish on $U_a^{\mathrm{op}}$, it follows by proper base change that $R\pi_{s,*}\Xi$ identifies with the cohomology $R\Gamma(\bbG_{a,k},\calL_{\mathrm{AS}})$ of the Artin--Schreier sheaf, i.e., it vanishes. 
	
	If $w=w_\nu$ is the minimal length representative of $W_{\mathrm{fin}}t_\nu$, we can check that the map $\Fl_{ \calI,1}^{\mathrm{op}}\tilde{\times}\Fl_{\calI,w} \to \Fl_{\calI,w}^{\mathrm{op}}$ is an isomorphism. Indeed, both are affine spaces with the same dimension, and thus the given map is universally bijective by basic properties of Tits systems. In particular, we conclude that $\mathrm{av}_{\calI\calW}(\mathrm{IC}_{w})$ identifies with $\mathrm{IC}_{\nu}^{\calI\calW}$.
\end{proof}

\begin{proposition}\label{aviw_is_t_exact}
	The functor $\mathrm{av}_{\calI\calW}$ is perverse t-exact.
\end{proposition}

\begin{proof}
	Since each half of the t-structure on $\scrD_{\mathrm{\acute{e}t}}(\Hk_{\calI})$ is spanned under extensions by appropriate shifts of the standard or costandard sheaves, it will suffice by symmetry to show that 
	\begin{equation} \label{equation_av_IW_of_std_objects}
		\mathrm{av}_{\calI\calW}(\Delta_w)=\Delta_\nu^{\calI\calW}
	\end{equation}
	for all $w\in W$, where $\nu \in \bbX_\bullet$ is in the same left $W_{\mathrm{fin}}$-coset. When $w=w_\nu$ is the minimal length representative, this follows from the same argument of the previous lemma for IC sheaves. In general, consider an injection of perverse sheaves $\Delta_{w_{\nu}} \to \Delta_w$ as in \cite[Lemma 3b]{AB09}, whose cone is spanned under extensions by $\mathrm{IC}_y \ast \Delta_{w_\nu}$ for non-trivial $y \in W_{\mathrm{fin}}$. The latter sheaves vanish under $\mathrm{av}_{\calI\calW}$ by the previous lemma and we get the desired conclusion.
\end{proof}

We are now ready to prove the main result regarding Iwahori--Whittaker averaging. 
\begin{theorem}\label{avIW_factos_through_pervas}
	The averaging functor restricted to perverse sheaves factors through a fully faithful functor \begin{equation}\mathrm{av}_{\calI\calW}^{\mathrm{as}}\colon \scrP_{\mathrm{as}}(\Hk_{\calI})\to \scrP(\Hk_{\calI\calW})\end{equation}
	where the left side is the \textit{anti-spherical category} which isthe Serre quotient of $\scrP(\Hk_{\calI})$ obtained by modding out the sheaves $\mathrm{IC}_{w}$ for all $w \in W$ which are not minimal in their left $W_{\mathrm{fin}}$-cosets.
\end{theorem}

\begin{proof}
	Let $\calG_{\geq 1} \to G$ be the dilatation along the identity subgroup in $G_k$ so that $L^+\calG_{\geq 1} = L^{\geq 1}G_O\subset L^+G_O$ is the first congruence subgroup of $G$. We have a natural map $\alpha\colon [L^+\calG_{\geq 1} \backslash \Fl_{ \calI}] \to \Hk_{\calI}$ and similarly a forgetful functor between stable $\infty$-categories $\beta^*\colon \scrD_{\mathrm{\acute{e}t}}(\Hk_{\calI\calW}) \to \scrD_{\mathrm{\acute{e}t}}([L^+\calG_{\geq 1} \backslash \Fl_{ \calI}] )$ because the kernel of $\chi$ contains $L^{\geq 1}G_O$. We consider the induced functor 
	\begin{equation}
		\mathrm{ind}_{\calI\calW}:= {}^pH^{\minus\mathrm{rk}G}\circ R\alpha_*\circ \beta^*\colon \scrP(\Hk_{\calI\calW}) \to \scrP(\Hk_{\calI})
	\end{equation}
	and claim that its composition with the quotient map $\scrP(\Hk_{\calI}) \to \scrP_{\mathrm{as}}(\Hk_{\calI})$, which we denote by $\mathrm{ind}_{\calI\calW}^{\mathrm{as}}$, yields a right inverse to $\mathrm{av}^{\mathrm{as}}_{\calI\calW}$. First notice that $ \mathrm{ind}_{\calI\calW}(\Xi)$ is an extension of negative shifts of $\mathrm{IC}_w$ for $w\in W_{\mathrm{fin}}$, with the local system $\mathrm{IC}_1$ appearing with multiplicity $1$ (see \cite[Lemma 6.4.8]{AR}). If $w$ is non-trivial and $\mathcal{F}\in \calP(\Hk_\calI)$, we can show that $\mathrm{IC}_w[n] \ast \calF$ is $L^+\calP_s$-equivariant for some simple reflection $s$, hence its perverse cohomology groups die under the quotient map $\scrP(\Hk_{\calI}) \to \scrP_{\mathrm{as}}(\Hk_{\calI})$. If $w=1$ and $n\neq 0$, then $\mathrm{IC}_1[n]\ast \calF$ sits in non-zero perverse degree. In total, this yields an equivalence of functors $\mathrm{ind}_{\calI\calW}^{\mathrm{as}} \circ \mathrm{av}_{\calI\calW}^{\mathrm{as}} \simeq \mathrm{id}$. This implies that $\mathrm{av}_{\calI\calW}^{\mathrm{as}}$ is injective on $\mathrm{Ext}$ groups. To see that $\mathrm{av}_{\calI\calW}^{\mathrm{as}}$ is fully faithful, we argue by induction on the length of the objects being considered: the simple case is a consequence of \Cref{Lemma_IW_averaging_kills}, implying bijectivity of simple objects along $\mathrm{av}_{\calI\calW}^{\mathrm{as}}$; the induction step follows from the 5-lemma and the injectivity on $\mathrm{Ext}$-groups.
\end{proof}

\section{Tilting modules}	

We proved in \Cref{prop_IW_category_is_highest_weight} that the category of Iwahori-Whittaker perverse sheaves admits a highest weight category structure. It then makes sense to discuss \textit{tilting objects} in this category. The aim of the current section is to show that the functor
\begin{equation}
	\scrZ_{\calI\calW}:=\mathrm{av}_{\calI\calW}\circ\scrZ: \mathrm{Rep}_\Lambda(\hat{G})\rightarrow \scrP(\Hk_{\calI\calW})
\end{equation}
lands on the full subcategory of tilting objects. This is related to the 'parabolic-singular' Koszul duality phenomenon studied by Beilinson--Ginzburg--Soergel \cite{BGS96} for finite flag varieties and Bezrukavnikov--Yun \cite{BY13} for Kac--Moody flag varieties. We continue to assume that $\Lambda$ is an algebraic field extension of $\mathbb{Q}_\ell(\zeta_\ell)$ for a fixed primitive $p$-th root of unity $\zeta \in \bar{\mathbb{Q}_\ell}$.

We recall the following useful property.
\begin{proposition}\label{prop_stalk_and_costalk_of_tilting_objects}
	For any $\calF\in \scrD_{\mathrm{cons}}(\Hk_{\calI\calW})$, $\calF$ is tilting if and only if $(j_{w_\nu}^{\mathrm{op}})^*\calF$ and $R(j_{w_\nu}^{\mathrm{op}})^!\calF$ are both concentrated in degree $\minus\ell(w_\nu)$ for all $\nu\in\mathbb{X}_\bullet$.
\end{proposition}
\begin{proof}
	This follows from \cite[Proposition 1.3]{BBM04}.
\end{proof}
\subsection{Multiplicities of tilting objects}Let $\calF\in \scrP(\Hk_{{\calI\calW}})$ be a tilting object. Recall that the multiplicity of the standard (resp.~costandard) objects $\Delta^{\calI\calW}_\nu$ (resp.~$\nabla^{\calI\calW}_\nu$) in $\calF$ is well-defined and we denote it by $(\mathcal{F}:\Delta^{\calI\calW}_\nu)$(resp. $(\calF: \nabla^{\calI\calW}_\nu$)). It follows by orthogonality of $\Delta_\nu^{\calI\calW}$ and $\nab_\nu^{\calI\calW}$ that
\begin{align}
	(\mathcal{F}:\Delta^{\calI\calW}_\nu)= \dim\mathrm{Hom} (\calF, \nabla^{\calI\calW}_\nu),\\
	(\mathcal{F}:\nabla^{\calI\calW}_\nu)= \dim\mathrm{Hom} ( \Delta^{\calI\calW}_\nu,\calF).
\end{align}
where the Hom spaces are taken inside $\scrD_{\mathrm{\acute{e}t}}(\Hk_{\calI\calW})$, compare with \cite[Theorem 3.2.1]{BGS96}. We use the same notation for $\calF\in \calD_\et(\Hk_{\calI\calW})$ as well.

\begin{proposition}
	For any $V\in \mathrm{Rep}_\Lambda (\hat{G})$ and any $\mu\in \mathbb{X}_\bullet$, we have 
	\begin{align}
		&  \sum_{i\geq 0}  (-1)^i \dim (\Hom_{\scrD_{\mathrm{\acute{e}t}}(\Hk_{\calI\calW})} (\scrZ_{\calI\calW}(V)[i],\nabla_\mu^{\calI\calW}))=\dim (V(\mu))\\
		& \sum_{i\leq 0} (-1)^i  \dim (\Hom_{\scrD_\mathrm{\acute{e}t}(\Hk_{\calI\calW})} (\Delta_\mu^{\calI\calW},\scrZ_{\calI\calW}(V)[i])=\dim (V(\mu)),
	\end{align}
	where $V(\mu)$ denotes the $\mu$-weight space of $V$.
\end{proposition}
\begin{proof}
	The proof follows the strategy of \cite[Lemma 27]{AB09} and \cite[Proposition 6.5.4]{AR} by \Cref{coro_wakimoto_graded_ct}, and (\ref{equation_av_IW_of_std_objects}).
\end{proof}

\begin{corollary}\label{coro_mult_support_of_Z_IW}
	For any $V\in\ \mathrm{Rep}_\Lambda (\hat{G})$,
	\begin{itemize}
		\item [(1)]if $\scrZ_{{\calI\calW}}(V)$ is tilting, then
		\begin{equation}\label{equation_dimension_estimate_tilting_sheaves}
			(\scrZ_{{\calI\calW}}(V):\Delta_\mu^{\calI\calW})=(\scrZ_{{\calI\calW}}(V):\nabla_\mu^{\calI\calW})=\dim (V_\mu);
		\end{equation}
		\item[(2)] if $V$ is the highest weight representation of highest weight $\mu$, then $\iwcentral(V)$ is supported on $\Fl^{\mathrm{op}}_{\calI,w_\mu}$.
	\end{itemize}

\end{corollary}
\begin{proof}
	Statements follow from \cite[Remark 6.5.5, Corollary 6.5.6]{AR}.
\end{proof}

The work that now follows will eventually lead to proving that $\mathscr{Z}_{\mathcal{I}\mathcal{W}}(V)$ is tilting for almost all groups $G$, which we explain below in \Cref{thm_tilting}. But first, we handle minuscule representations, and for that we require the next lemma.

\begin{lemma}\label{lemma_multiplicity_ZIW}
	For any $V\in\Rep(\hat{G})$, $\nu\in\mathbb{X}_\bullet$, $x\in W_\mathrm{fin}$, and $n\in\mathbb{Z}$, we have isomorphisms
	\begin{align*}
		\mathrm{Ext}^n(\Delta_\nu^{\calI\calW},\scrZ_{\calI\calW}(V))\simeq  	\mathrm{Ext}^n(\Delta_{x(\nu)}^{\calI\calW},\scrZ_{\calI\calW}(V)),\\
		\mathrm{Ext}^n(\scrZ_{\calI\calW}(V),\nabla_\nu^{\calI\calW})\simeq  	\mathrm{Ext}^n(\scrZ_{\calI\calW}(V),\nabla_{x(\nu)}^{\calI\calW}).
	\end{align*}
\end{lemma}
\begin{proof}
	The proof is similar to the arguments in \cite[Lemma 6.5.11]{AR} and we sketch it here. Without loss of generality, we assume $\nu$ to be dominant so that $w_\nu=t_\nu$, as the statement only depends on its $W_{\mathrm{fin}}$-orbit. We can find $y\in W_\mathrm{fin}$ with minimal length such that $t_\nu=w_{x(\nu)}y^{-1}$ is a minimal length decomposition,  and $xy(\nu)=\nu$. Then
	$$
	\Delta_{\nu}^{\calI\calW}=\Delta_0^{\calI\calW}*\Delta_{w_\nu}\simeq \Delta_0^{\calI\calW}*\Delta_{w_{x(\nu)}}*\Delta_{y^{-1}}\simeq \Delta_{x(\nu)}^{\calI\calW}*\Delta_{y^{-1}},
	$$
	by \Cref{aviw_is_t_exact} and \Cref{djfejkjd;kjkel;}. Thus
	\begin{align*}
		\mathrm{Ext}^n(\Delta_\nu^{\calI\calW},\scrZ_{\calI\calW}(V)) & \simeq 	\mathrm{Ext}^n(\Delta_{\nu}^{\calI\calW}*\Delta_{y},\scrZ_{\calI\calW}(V)*\Delta_{y})\\
		&\simeq 	\mathrm{Ext}^n(\Delta_{x(\nu)}^{\calI\calW},\scrZ_{\calI\calW}(V)*\Delta_y) \\
		& \simeq 	\mathrm{Ext}^n(\Delta_{x(\nu)}^{\calI\calW}, \Delta_0^{\calI\calW}*\scrZ(V)*\Delta_y)\\
		&\simeq 	\mathrm{Ext}^n(\Delta_{x(\nu)}^{\calI\calW}, \Delta_0^{\calI\calW}*\Delta_y*\scrZ(V))\\
		&\simeq	\mathrm{Ext}^n(\Delta_{x(\nu)}^{\calI\calW},\scrZ_{\calI\calW}(V))
	\end{align*}
	again by \Cref{aviw_is_t_exact} and \Cref{djfejkjd;kjkel;}. The second isomorphism is proved analogously.
\end{proof}

\begin{proposition}\label{Z_IW_is_tilting_for_minuscule_coweights}
	Let $V$ be a simple representation of $\hat{G}$ with highest weight $\mu$ being a minuscule dominant coweight of $G$, then $\scrZ_{\calI\calW}(V)$ is tilting.
\end{proposition}
\begin{proof}
	Because of \Cref{prop_stalk_and_costalk_of_tilting_objects}, it suffices to show that $(j_{w_\nu}^{\mathrm{op}})^*(\scrZ_{\calI\calW}(V))$ and $(j_{w_\nu}^{\mathrm{op}})^!(\scrZ_{\calI\calW}(V))$ are both perverse sheaves. Since $\mu$ is minuscule, the only weights we have to check are in the $W_{\mathrm{fin}}$-orbit of $\mu$. By adjunction and \Cref{lemma_multiplicity_ZIW}, we are reduced to showing the statements above for $\nu=\mu$. Note that the support of $\scrZ(V)$ equals the $\mu$-admissible locus $A_{I,\mu}$ by \cite[Theorem 6.16]{AGLR22}, whose open $L^+\calI$-orbits are indexed by $W_{\mathrm{fin}}\mu$. We deduce that the support of $\scrZ_{\calI\calW}(V)$ equals $\Fl_{\mathcal{I},\leq w_0t_\mu}$ and hence the locally closed immersion $j_{w_\mu}^{\mathrm{op}}$ is actually open and dense. In particular, it is clear that $(j_{w_\mu}^{\mathrm{op}})^*(\scrZ_{\calI\calW}(V))$ and $(j_{w_\mu}^{\mathrm{op}})^!(\scrZ_{\calI\calW}(V))$ are both perverse.
\end{proof}

Now, we deduce \Cref{thm_tilting} by propagating the result via convolution.

\begin{proposition}\label{lem_prop_tilting_tensor_prod}
	Let $V,W\in\Rep(\hat{G})$ such that $\iwcentral (V)$ and $\iwcentral (W)$ are both tilting. Then so is $\iwcentral (V\otimes W)$.
\end{proposition}
\begin{proof}
	Our proof is similar to \cite[Proposition 6.5.7]{AR}. It suffices to prove that $(j_{w_\nu}^{\mathrm{op}})^*\iwcentral (V\otimes W)$ and $(j_{w_\nu}^{\mathrm{op}})^!\iwcentral (V\otimes W)$ are both perverse for any $\nu\in\mathbb{X}_\bullet$. Since $\iwcentral (V\otimes W)$ is perverse, $(j_{w_\nu}^{\mathrm{op}})^*\iwcentral (V\otimes W)$ concentrates in perverse degress $\leq 0$, and $(j_{w_\nu}^{\mathrm{op}})^!(\iwcentral (V\otimes W))$ concentrates in perverse degrees $\geq 0$. Note that if $\iwcentral(V)$ is tilting, the object $\iwcentral(V\otimes W)\cong \iwcentral(V)*\scrZ(W)$ admits a filtration with subquotients $\Delta_\mu^{\calI\calW}*\scrZ(W)$. By \Cref{aviw_is_t_exact} and \Cref{thm_assoc_center}, $\Delta_\mu^{{\calI\calW}}*\scrZ(W)\cong \iwcentral(W)*\Delta_{w_\mu}$. Since $\iwcentral(W)$ is tilting, then $\Delta_\mu^{{\calI\calW}}*\scrZ(W)$ admits a filtration with subquotients $\mathrm{av}_{{\calI\calW}}(\Delta_{w_\nu}*\Delta_{w_\mu})$. \Cref{lemma_convolution_stds_and_costds} and \Cref{aviw_is_t_exact} imply that $(j_{w_\nu}^{\mathrm{op}})^*(\iwcentral(V\otimes W))$ concentrates in non-negative perverse degrees. The statement for $(j_{w_\nu}^{\mathrm{op}})^!(\iwcentral(V\otimes W))$ is proved similarly.
\end{proof}

Now, we need to describe when minuscule representations form a class of Karoubi generators for the symmetric monoidal category $\mathrm{Rep}(\hat{G})$. It will be enough for us to restrict to adjoint $G$. The argument below was partially suggested to us by Jeremy Taylor.

\begin{lemma}\label{lem_minusc_gener}
	Assume $G$ is adjoint. The following are equivalent:
	\begin{enumerate}
		\item \label{lem_minusc_gener_cat} $\mathrm{Rep}(\hat{G})$ is generated by minuscule representations under sums, retracts and tensor products;
		\item \label{lem_minusc_gener_rep}$\hat{G}$ admits a faithful minuscule representation;
		\item \label{lem_minusc_gener_dyn} Every simple adjoint factor of $G$ contains a minuscule coweight (i.e., it is not of type $E_8$, $F_4$ nor $G_2$).
	\end{enumerate}
\end{lemma}

\begin{proof}
	Recall that the generation property when including quasi-minuscules was observed by Ngô--Polo, see \cite[Lemma 10.3]{NP01}, and we could read this off their proof. It is clear that (\ref{lem_minusc_gener_cat}) implies (\ref{lem_minusc_gener_dyn}). Also, (\ref{lem_minusc_gener_rep}) implies (\ref{lem_minusc_gener_cat}) because we are working over a characteristic $0$ field $\Lambda$ and hence we can invoke the Peter--Weyl theorem: explicitly, taking global sections along the faithful representation $\rho \colon \hat{G}\to \mathrm{GL}(V)$ over $\Lambda$, we see that the regular $\hat{G}$-representation is spanned by minuscule representations under sums, retracts and tensor products, so by semisimplicity the same holds for any representation of $G$.
	Finally, we show that (\ref{lem_minusc_gener_dyn}) implies (\ref{lem_minusc_gener_rep}). It is enough to assume $G$ is almost simple and let $\mu$ be a minuscule coweight. The representation $\hat{G} \to \mathrm{GL}(V_\mu)$ has finite central kernel, which is shared by any non-trivial representation whose weights differ from $\mu$ by an element of the coroot lattice. Varying the minuscule $\mu$, we get every single element in $\bbX_\bullet$ by \cite[Exercices 24--25, p.~225]{Bou68}, which implies that the direct sum of all minuscule representations is faithful.
\end{proof}

\begin{definition}\label{def_enough_minusc}
	We say that $G$ has {\it enough minuscules} if its adjoint quotient satisfies the equivalent conditions of \Cref{lem_minusc_gener}.
\end{definition}

Now, we can prove the tilting property under a mild assumption involving exceptional groups. As explained during the introduction, this stems from the lack of the rotation $\mathbb{G}_{m,k}^{\mathrm{pf}}$-action in our $p$-adic context.

\begin{theorem}\label{thm_tilting}
	If $G$ has enough minuscules, then $\mathscr{Z}_{\mathcal{I}\mathcal{W}}(V)$ is tilting for all $V \in \mathrm{Rep}(\hat{G})$.
\end{theorem}

\begin{proof}
	First of all, we perform a reduction to the adjoint case. Recall that the adjoint map $M_{\calI,\mu}\to M_{\calI_{\mathrm{ad}},\mu_{\mathrm{ad}}}$ is an isomorphism, so if $V$ is simple (which we may and do assume), we can naturally identify the central sheaf $\scrZ(V)$ with the central sheaf $\scrZ(V_{\hat{G}_{\mathrm{sc}}})$ constructed in $\Hk_{\calI_{\mathrm{ad}}}$, see also \cite[Section VI.11]{FS21} to see that geometric Satake is compatible with adjoint quotients. The natural functor $\mathscr{P}(\mathrm{Hk}_{\mathcal{I}\mathcal{W}})\to \mathscr{P}(\mathrm{Hk}_{\mathcal{I}\mathcal{W}_{\mathrm{ad}}})$ also becomes an equivalence when restricted to a single connected component of the Hecke stack, identifying standard and costandard objects in the obvious way. In particular, the assertion can be read off the adjoint case.
	
	Now, if $G$ is adjoint, we apply \Cref{Z_IW_is_tilting_for_minuscule_coweights}, \Cref{lem_prop_tilting_tensor_prod} and \Cref{lem_minusc_gener} in combination to arrive at the desired result.
\end{proof}
\section{Regular quotient}
During this section, we assume $\La=\bar \bbQ_\ell$ is algebraically closed. Consider the Serre subcategory $\scrP_{>0}(\Hk_{\calI}) \subset \scrP(\Hk_{\calI})$ generated by IC sheaves with positive dimensional support and denote by $\Pi^0$ the natural quotient functor  
\begin{equation}
	\label{eq:2-functor-pi-0}
	\Pi^0: \scrP(\Hk_{\calI})\rightarrow \scrP_0(\Hk_{\calI}):=\scrP(\Hk_{\calI})/\scrP_{>0}(\Hk_{\calI})
\end{equation}
to the Serre quotient. 

Therefore the simple objects in $\scrP_0(\Hk_{\calI})$ are precisely given by the $\mathrm{IC}_\tau$ where $\tau \in \Omega_{\bba}$ stabilizes the base alcove. In particular, if $G$ is semi-simple,  $\scrP_0(\Hk_{\calI})$ has only finitely many simple objects.
\begin{proposition}\label{prop_monoid_perv_reg}
	The monoidal structure on $\scrP(\Hk_{\calI})$ given by perverse truncated convolution $^pH^0((-)\ast(-))$ descends to an exact monoidal structure $\circledast$ on $\scrP_0(\Hk_\calI)$.
\end{proposition}
\begin{proof}
	The proof follows the idea in \cite[Proposition 6.5.14]{AR} and we sketch it here. Let $A_1=\IC_w$ for some $w\in W$ with $\ell(w)> 0$. Then there exists a simple reflection $s$ such that $\ell(sw)<\ell(w)$. Let $\calJ_s$ be the minimal parahoric containing $\calI$ associated with $s$. Then $A_1$ is $\calJ_s$-equivariant. It follows that $A_1\ast A_2$ is also $L^+\calJ_s$-equivariant for any $A_2 \in \scrP(\Hk_{\calI})$, and so are its perverse cohomology sheaves. But a $L^+\calJ_s$-equivariant perverse sheaf has equivariant composition factors, hence lies in $\scrP_{>0}(\Hk_{\calI})$. If we have instead $A_2=\mathrm{IC}_w$ for $\ell(w)>0$, then we exploit $L^+\calJ_s$-equivariance on the {\it right}, compare also with the proof of \Cref{sec:geometry-affine-flag-4-convolution-is-affine}: it descends to $\Hk_{(\calI,\calJ_s)}:= L^+\calI\backslash LG/L^+\calJ_s$, and either by arguing with $L^+\calI\backslash LG$ or by applying the inverse map of $LG$, we conclude as in the previous case that $^pH^0(A_1 \ast A_2)$ also lies in $\scrP_{>0}(\Hk_{\calI})$.
	Varying $w$, we conclude that the monoidal structure given by $^pH^0(\ast)$ on $\scrP(\Hk_{\calI})$ descends to a monoidal structure $\circledast$ on $\scrP_0(\Hk_{\calI})$. In order to check exactness of $\circledast$, we must see that for arbitrary $x,y \in W$ the perverse cohomology sheaves in non-zero degree of a convolution product $\IC_x\ast \IC_y$  lie in $\scrP_{>0}(\Hk_{\calI})$. The only remaining case to analyze is when
	both elements have length $0$, but in this case $\IC_x\ast \IC_y=\IC_{xy}$.
	%
	%
\end{proof}

We have the following important result:

\begin{lemma}\label{properties_of_Z_0_and_n_0}
	The functor $\scrZ_0:=\Pi^0\circ \scrZ\colon \mathrm{Rep}(\hat{G}) \to \scrP_0(\Hk_{\calI})$ is monoidal and central.
\end{lemma}

\begin{proof}
	By \Cref{prop_monoid_perv_reg}, we can construct the monoidality and centrality isomorphisms by applying those of \Cref{prop_Z_monoidal} and \Cref{thm_assoc_center} and then projecting towards $\scrP_0(\Hk_{\calI})$, compare with \cite[Lemma 6.5.15]{AR}.
\end{proof}

Note that for every $\hat{G}$-representation $V$, we have a nilpotent operator $\mathbf{n}^0_V \colon \scrZ_0(V) \to \scrZ_0(V)$ arising from the logarithm of the monodromy of $\scrZ(V)$.
Denote by $\scrP_0^c(\Hk_{\calI})$ the full subcategory of $\scrP_0(\Hk_{\calI})$ whose objects are the subquotients of $\scrZ_0(V), V\in\Rep(\hat{G})$. The exactness of $\circledast$ and monoidality of $\scrZ_0$ imply that $\scrP_0^c(\Hk_{\calI})$ is closed under the monoidal structure. By definition, the functor $\scrZ_0$ naturally factors through a functor $\scrZ_0^c:\Rep(\hat{G})\rightarrow \scrP_0^c(\Hk_{\calI})$.
\begin{proposition}\label{prop_tannakian_for_perv_I0}
	There exists a closed subgroup $H\subset \hat{G}$ such that we have
	\begin{itemize}
		\item [(1)] an equivalence of monoidal categories
		\begin{equation}
			\Phi^0:(\scrP_0^c(\Hk_{\calI}),\circledast)\simeq (\Rep(H),\otimes).
		\end{equation}
		\item[(2)] a nilpotent element $n_0\in\hat{\frakg}$ such that $H\subset Z_{\hat{G}}(n_0)$. 
		\item[(3)] an isomorphism of functors $\alpha:\Phi^0\circ\scrZ^c_0\simeq \mathrm{For}^{\hat{G}}_H$, carrying the monodromy operators $\mathbf{n}_V^0$ to the natural action of $n_0$ on $V$. 

	\end{itemize}
\end{proposition}

\begin{remark} \label{remark_regular_quotient}
	If $G$ has enough minuscules, then $\scrP_0^c(\Hk_{\calI})=\scrP_I^0$, and $H=Z_{\hat{G}}(n_0)$. We do not need this in the proof of the main theorem and will postpone the discussion of this fact to Section10 (cf.  \ \Cref{H_equals_the_full_centralizer}).
\end{remark}

\begin{proof}
	The above proposition is the mixed characteristic analogue of a particular case of \cite[Proposition 1, Theorem 3]{Bez04} and \cite[Proposition 6.5.18]{AR}. We sketch the proof here and refer further details to \textit{loc.cit}. Note that we can regard the regular representation $\calO(\hat{G})$ of the dual group as a ring object in $\mathrm{Ind}(\Rep(\hat{G}))$. 
	Then $\scrZ_0(\calO(\hat{G}))$ is a ring object in $\mathrm{Ind}(\scrP_0^c(\Hk_{\calI}))$. Zorn's lemma implies that there exists a maximal left ideal subobject $ \scrI \subset \scrZ_0(\calO(\hat{G}))$, whose quotient will be denoted by $\calO(H)$. The centrality of $\scrZ_0$ (cf. \ \Cref{properties_of_Z_0_and_n_0}) implies that $\calO(H)$ is also a ring object. Thus, we define $\calO(H)$-$\mathrm{Mod}$ as the category of left $\calO(H)$-modules in $\mathrm{Ind}(\scrP_0^c(\Hk_{\calI}))$. Clearly $\calO(H)$ is a simple object in the abelian category $\calO(H)$-$\mathrm{Mod}$. Hence, its endomorphism ring $K:=\mathrm{End}_{\calO(H) }(\calO(H))$ is a division algebra, and $V\mapsto V\otimes K$ defines an equivalence between the category of right finite $K$-modules and the full subcategory in $\calO(H)$-$\mathrm{Mod}$ generated by $\calO(H)$ under finite direct sums and subquotients. Now, we deduce that 
	\begin{equation}\label{endo_of_ind_objects}
		K\simeq \Hom_{\mathrm{Ind}(\scrP_0^c(\Hk_{\calI}))}(\delta_0,\calO(H))\simeq \bar{\mathbb{Q}}_\ell.
	\end{equation}
	Here, the first isomorphism is obtained by restriction to $\delta_0\subset \calO(H)$, and it implies that $K$ is of countable $\bar \bbQ_{\ell}$-dimension. For the second isomorphism, we observe that every division algebra of countable dimension over $\bar{\mathbb{Q}}_{\ell}$ is isomorphic to $\bar \bbQ_\ell$.
	Now, we construct a monoidal fiber functor to invoke the Tannakian formalism.
	\begin{lemma}\label{prep_Tannakian}
		\begin{itemize}
			\item[(1)] For any $A\in \scrP_0^c(\Hk_{\calI})$, there exists a finite-dimensional vector space $V$ such that $\calO(H)\circledast A\simeq \calO(H)\otimes V$ is an isomorphism of $\calO(H)$-modules, where we endow $V$ with the trivial $\calO(H)$-action.
			\item [(2)] The functor $\Phi_G:\scrP_0^c(\Hk_{\calI})\rightarrow \mathrm{Vect}_{\La}$ defined by $A\mapsto \Hom(\calO(H),\calO(H)\circledast A)$
			is an exact, faithful, and monoidal functor. In addition, $\Phi_G\circ\scrZ^c_0\simeq \mathrm{For}^{\hat{G}}:\Rep(\hat{G})\rightarrow \mathrm{Vect}_{\bar{\mathbb{Q}}_\ell} $.
		\end{itemize}
	\end{lemma}
	\begin{proof}			
		This statement is a mixed-characteristic analogue of \cite[Lemma 6.5.20]{AR}. To prove statement $(1)$ in the above lemma, we first note that there is a canonical isomorphism $\scrZ_0(\calO(\hat{G}))\circledast \scrZ_0(V)\simeq \scrZ_0(\calO(\hat{G}))\otimes_{\La}V$ of $\scrZ_0(\calO(\hat{G}))$-modules for any $V\in \Rep(\hat{G})$. Quotienting out the maximal left idea $\scrI$, we conclude that $\calO(H)\circledast \scrZ_0(V)\simeq \calO(H)\otimes V$. The general situation follows from taking subquotients from both sides and we thus settle statement $(1)$.

		The exactness of $\Phi_G$ follows from that of $\scrZ_0$ and statement $(1)$. Also, \Cref{endo_of_ind_objects} and statement $(1)$ imply that $\calO(H)\circledast A\cong \calO(H)\otimes \Phi_G(A)$ for any $A\in\scrP_0^c(\Hk_{\calI})$. Then for any $A_1,A_2\in\scrP_0^c(\Hk_{\calI})$, we have 
		\begin{align*}
			\Phi_G(A_1\circledast A_2) & \simeq \Hom(\calO(H), \calO(H)\otimes \Phi_G(A_1\circledast A_2)) \simeq \Hom(\calO(H),\calO(H)\circledast (A_1\circledast A_2))\\
			& \simeq \Hom (\calO(H), \Phi_G(A_1)\otimes (\calO(H)\circledast A_2) )\simeq\Phi_G(A_1)\otimes \Phi_G(A_2).
		\end{align*}
		Finally, it suffices to check that $\Phi_G$ sends non-zero objects to non-zero objects since it is exact. This can be checked on all simple objects $\Pi^0(\IC_\tau)$. The faithfulness then follows from the dualizability, in fact invertibility, of $\Pi^0(\IC_\tau)$ and the monoidal structure of $\Phi_G$.
	\end{proof}
	
	\Cref{prep_Tannakian} allows us to apply the Tannakian formalism and obtain an equivalence of monoidal categories
	\begin{equation}
		\Psi:\scrP_0^c(\Hk_{\calI})\simeq \mathrm{Comod}_{\calA(H)},
	\end{equation}
	where $\scrA(H)$ is a $\La$-bialgebra and $\mathrm{Comod}_{\calA(H)}$ is the category of $\La$-finite $\calA(H)$-comodules. In addition, the composition of this equivalence with the natural forgetful functor $\mathrm{Comod}_{\calA(H)}\rightarrow \mathrm{Vect}_{\La}$ equals $\Phi_G$. By the Tannakian construction, the functor $\Psi\circ \scrZ^c_0:\Rep(\hat{G})\rightarrow\mathrm{Comod}_{\calA(H)}$ induces a surjective morphism of bialgebras $\calO(G)\rightarrow \calA(H)$. It follows from \cite[Lemma 3]{Bez04} that $\calA(H)$ is commutative and $\Spec \calA(H)$ is the desired group scheme $H$.

	Now we construct $n_0$. Recall our construction of the nilpotent endomorphism $\mathbf{n}_V^0$ of $\scrZ_0(V)$ for any ${V\in\Rep(\hat{G})}$. By naturality and compatibility with the monoidal structure as in \Cref{D_P_for_n_V}, we deduce a tensor endomorphism of the functor $\Phi_G\circ\scrZ_0^c\simeq\mathrm{For}^{\hat{G}}$. 
	In particular, this gives rise to an element $n_0$ of $\hat{\frakg}$ by the Tannakian formalism. On the other hand, $\Phi_G\circ\scrZ^c_0\simeq \mathrm{For}^{H}\circ \Psi\circ\scrZ^c_0\simeq \mathrm{For}^{H}\circ\mathrm{For}^{\hat{G}}_H $, and $(\mathbf{n}_V^0)_V$ induces an automorphism of $\mathrm{For}^{\hat{G}}_H$. Hence, $H\subset Z_{\hat{G}}(n_0)$.
\end{proof}

\begin{proposition}
	If $G$ has enough minuscules, then the nilpotent element $n_0$ is regular.
\end{proposition}

\begin{proof}
	Our argument is similar to the one in \cite[Section 6.5.8]{AR} and uses weight theory. Recall that in \Cref{wt_monodr_conjecture}, we posited that the mixed sheaves $\scrZ_{\mathrm{mix}}(V)$ ought to be monodromy-pure of weight $0$, as this is the case in equicharacteristic due to a theorem of Gabber whose proof was written up by Beilinson--Bernstein \cite{BB93}, compare with \cite[Theorem 5.1.2]{BB93}. By \Cref{prop_wtmon_minuscule}, we know that this holds for minuscule representations. Note that the functor $\Pi^0\colon \calP(\Hk_\calI)\to \calP_0(\Hk_{\calI})$ in (\ref{eq:2-functor-pi-0}) admits a mixed variant $\Pi^0_{\mathrm{mix}}$, namely the quotient of mixed perverse sheaves by the ones with positive dimensional support.  We claim that the images under $\Pi^0_{\mathrm{mix}}$ are monodromy-pure of weight $0$. In other words, we want to show that the weight filtration obtained on $\scrZ_0(V)$ via push-pull coincides with the monodromy filtration induced by $n_0$. It suffices to prove this when $G$ is adjoint, and then we can check that both filtrations are monoidal on $V$, see \cite[Lemma 4.1.2]{BB93}, and also respect splittings, so we can propagate the claim starting from the minuscule case by \Cref{lem_minusc_gener}.
	
	Now, we can check whether $n_0$ is regular by calculating the dimension of $\hat{\frakg}^{n_0}$. Reading off the weight filtration $\scrZ_0(\hat{\frakg})$ on the Iwahori--Hecke algebra, one sees that its $i$-th graded has dimension equal to that of the sum of the weight spaces $\hat{\frakg}(\nu)$ with $\langle 2\rho,\nu\rangle=i$. Since the weights of $\hat{\frakg}$ are roots of $\hat{G}$, its non-zero gradeds are even integers, and hence $\mathrm{dim}(\hat{\frakg}^{n_0}) =\mathrm{dim}(\hat{\frakg}(0))=\mathrm{rk}(G)$.
\end{proof}

\section{Proof of the AB equivalence}	

At this point, we consider the composition of the two functors
\begin{equation}\label{equation_AB-equivalence}
	\scrF_{\calI\calW}:=\mathrm{av}_{\calI\calW}\circ \scrF\colon \mathrm{Perf}([\hat{G}\backslash \hat{\calN}_{\mathrm{Spr}}])\to \scrD_{\mathrm{cons}}(\Hk_{\calI\calW})
\end{equation}
that we have extensively studied thus far. Our goal is to prove the Arkhipov--Bezrukavnikov equivalence below:

\begin{theorem}\label{thm_A-B}
	If $G$ has enough minuscules, then the functor $\scrF_{\calI\calW}$ is an equivalence.
\end{theorem}

Thus, for any $\calG\in \scrP(\Hk_{\mathcal{IW}})$, there exits  $\calF\in \mathrm{Perf}([\hat{G}\backslash \hat{\calN}_{\mathrm{Spr}}])$ such that $\scrF_{\calI\calW}(\calF)=\calG$. One can immediately draw the following conclusion by \Cref{aviw_is_t_exact} and \Cref{avIW_factos_through_pervas}:

\begin{corollary}\label{coro_av_IWas_is_an_equiv}
	If $G$ has enough minuscules, then the functor $\mathrm{av}_{\calI\calW}^{\mathrm{as}}$ from \Cref{avIW_factos_through_pervas} is an equivalence of abelian categories.
\end{corollary}

The strategy behind the proof of the theorem is as usual based on generators and relations. We start with the following lemma.

\begin{lemma}
	The $\infty$-category $\scrD_{\mathrm{cons}}(\Hk_{\calI\calW})$ is spanned by $\mathrm{av}_{\calI\calW}(\scrJ_\nu)$ for all $\nu \in \bbX_\bullet$ under cones and extensions.
\end{lemma}

\begin{proof}
	As in \Cref{sec:k_0-i-equivariant-1-description-of-k-0-of-i-equivariant-sheaves,wersd;kfdl;}, we can check that $\mathrm{av}_{\calI\calW}(\scrJ_\nu)$ has the same class in the Grothendieck group as $\Delta_\nu^{\calI\calW}$. Taking its Euler characteristic, we deduce that it is supported on $\Fl_{\calI,\leq w_0w_\nu}$ and has generic rank $1$. A standard induction argument now implies the spanning assertion.
\end{proof}

\begin{lemma}
	For any $V \in \Rep_\La \hat{G}$, the map
	\begin{equation}
		\Hom(\calO_{\hat{\calN}_{\mathrm{Spr}}},V\otimes \calO_{\hat{\calN}_{\mathrm{Spr}}}) \to \Hom(\Xi, \scrZ_{\calI\calW}(V))
	\end{equation}
	induced by $\scrF_{\calI\calW}$ is injective.
\end{lemma}

\begin{proof}
	Since $\mathrm{av}_{\calI\calW}^{\mathrm{as}}$ is fully faithful, it suffices to check the injectivity on the anti-spherical category $\scrP_{\mathrm{as}}(\Hk_{\calI})$. We can also further reduce to verifying injectivity after passing to the quotient $\scrP_0(\Hk_{\calI})$ defined in the previous section. Now, we use the regular orbit $\hat{G}/Z_{\hat{G}}(n_0)\simeq \calO_r \subset \hat{\calN}$, together with the compatible isomorphism $\scrP_0(\Hk_{\calI})\cong \Rep_\La (H)$ for a certain subgroup $H\subset Z_{\hat{G}}(n_0)$. In terms of these data, the homomorphism of $\Hom$-groups identifies with $ V^{Z_{\hat{G}}(n_0)} \to V^H$, which is clearly injective.
\end{proof}

We deduce our last key calculation:
\begin{corollary}
	For any $V \in \Rep_\La \hat{G}$, any $n \in \bbZ$ and $\nu \in \bbX_+$, the natural map
	\begin{equation}\label{eq_FIW_shift_hom}
		\mathrm{Ext}^n(\calO,V\otimes \calO(\nu)) \to \mathrm{Ext}^n(\Xi, \scrZ_{\calI\calW}(V)\ast \scrJ_\nu)
	\end{equation}
	is injective.
\end{corollary}

\begin{proof}
	The left side identifies with $(V\otimes H^n(\hat{\calN}_{\mathrm{Spr}},\calO(\nu)))^{\hat{G}}$. The higher cohomology of $\calO_{\hat{\calN}_{\mathrm{Spr}}}(\nu)$ vanishes, meaning we only need to consider the right side when $n=0$. Since there exists an equivariant embedding $\calO_{\hat{\calN}_{\mathrm{Spr}}}(\nu) \to W\otimes\calO_{\hat{\calN}_{\mathrm{Spr}}}$ for a certain $W\in \Rep_\Lambda \hat{G}$, the claim reduces to the preceding lemma.
\end{proof}

Finally, we can prove our main theorem, the AB equivalence.

\begin{proof}[Proof of \Cref{thm_A-B}]
	Applying \cite[Lemma 6.2.8.(2)]{AR} and the 5-lemma, fully faithfulness will follow from seeing that the injection \eqref{eq_FIW_shift_hom} is bijective. Since both sides are finite dimensional $\La$-vector spaces, it will be enough to check their dimensions match. Furthermore, once we know $\scrF_{\calI\calW}$ is fully faithful, we conclude it is an equivalence as its image spans the Iwahori--Whittaker category.
	
	Let us compute the dimension of the right side. After convolution on the right with $\scrJ_{-\nu}(\La)$, it vanishes if $n \neq 0$ by the tilting property of $\scrZ_{\calI\calW}(V)$ , see \Cref{thm_tilting}, and noting that $\mathrm{Ext}(\Delta^{\mathcal{IW}}_\mu,\nabla^\mathcal{IW}_\nu[i])=0$ for all $i\neq 0$ or $\mu\neq \nu$.  Otherwise, the dimension equal to that of the weight space $V(-\nu)$ by (\ref{equation_dimension_estimate_tilting_sheaves}). As for the left side, we have already checked its vanishing if $n\neq 0$ and it has otherwise dimension equal to that of the weight space $V(\nu)$, as one checks via the cohomology of $\calO(\nu)$ on the Springer resolution $\hat{\calN}_{\mathrm{Spr}}$, compare with \cite[Subsection 6.6.3]{AR}.
\end{proof}

\section{Exotic $t$-structure on the Springer resolution}
The equivalence in \Cref{thm_A-B} allows us to transport the perverse $t$-structure on $\scrD_{\mathrm{cons}}(\Hk_{\calI\calW})$ to a $t$-structure which we call the \textit{exotic t-structure} on $\mathrm{Perf}([\hat{G}\backslash \hat{\calN}_{\mathrm{Spr}}])$, at least when $G$ has enough minuscules. The exotic t-structure has been intrinsically studied in \cite{Bez06,MR16}. In this section, we discuss the exotic t-structure obtained via our object $\Fl_\calI$, and explain how it will be used to prove the assertions in \Cref{remark_regular_quotient} for groups with enough minuscules.

Recall the partial order $\leq $ on $\bbX_{\bullet}$ given by $\nu\leq \mu$ if and only if $\mu\minus \nu$ is a linear combination of positive roots. Note that the $\infty$-category $\mathrm{Perf}([\hat{G}\backslash \hat{\calN}_{\mathrm{Spr}}])$ has finite cohomological dimension by either \cite[Corollary 3.2.2]{BGS96} or \cite[Theorem 1.4.2]{DG13}, i.e., for any objects $\calA,\calB$, the vector space $\oplus_i\mathrm{Ext}^i(\calA,\calB)$ is finite-dimensional. Then \cite[Lemma 7.1.2]{AR} (see also \cite[Lemma 5]{Bez06}) implies that the line bundles $\calO(\nu)$ form an \textit{exceptional collection} indexed by $\nu \in \bbX_\bullet$ in the sense of \cite[Section 2.1.2]{Bez06} and generate $\scrD^b_{\mathrm{coh}}([\hat{G}\backslash \hat{\calN}_{\mathrm{Spr}}])$ under shifts and cones.

Choose a refinement of the Bruhat partial order $\leq$ on $\bbX_{\bullet}$ to a total order $\leq'$. Now, we can define the \textit{exotic exceptional collection} 
\begin{equation}
	\{\nabla^\mathrm{ex}_\nu\colon \nu\in \bbX_{\bullet}\}
\end{equation}
of $\scrD^b_\mathrm{coh}([\hat{G}\backslash \hat{\calN}_{\mathrm{Spr}}])$ as the collection of objects produced by \textit{mutation} of $\{\calO(\nu)\vert \nu\in\bbX_{\bullet})\}$ in the sense of \cite[Section 7.1.2]{AR}. By \cite[Proposition 3]{Bez06}, it in turn gives rise to the dual exotic exceptional collection 
\begin{equation}
	\{\Delta_\nu^\mathrm{ex}\colon \nu\in \bbX_{\bullet}\}
\end{equation}
in the sense of \textit{loc.cit}. 
Define $^{\mathrm{ex}}\scrD_\mathrm{coh}^{b,\geq 0}([\hat{G}\backslash \hat{\calN}_{\mathrm{Spr}}])$ (resp. $^{\mathrm{ex}}\scrD_\mathrm{coh}^{b,\leq 0}([\hat{G}\backslash \hat{\calN}_{\mathrm{Spr}}])$) as the full subcategory generated under extensions by objects $\nabla_\nu^\mathrm{ex}[n]$ (resp. $\Delta_\nu^\mathrm{ex}[n]$) with $\nu\in \bbZ_{\geq 0}$ and $n\in \bbZ_{\leq 0}$ (resp. $n\in \bbZ_{\geq 0}$). Then \cite[Proposition 4]{Bez06} shows that the above pair of full subcategories forms a bounded $t$-structure and we call it the {exotic t-structure}. We denote the heart of this $t$-structure by $\mathrm{ExCoh}([\hat{G}\backslash\hat{\calN}_{\mathrm{Spr}}])$.
\begin{proposition} \label{prop_F_IW_is_t_exact}
	If $G$ has enough minuscules, there are isomorphisms
	\begin{equation}
		\scrF_{\calI\calW}(\nabla_\nu^\mathrm{ex})\cong \nabla_\nu^{\calI\calW},
	\end{equation}
	\begin{equation}
		\scrF_{\calI\calW}(\Delta_\nu^\mathrm{ex})\cong \Delta_\nu^{\calI\calW},
	\end{equation}for any $\nu\in \bbX_{\bullet}$.
\end{proposition}

\begin{proof}
	The proof follows the idea of \cite[Proposition 7.1.5]{AR}  in the equicharacteristic situation and we sketch it here. In view of \Cref{thm_A-B}, it amounts to prove that the collection $\{\nabla^{\calI\calW}_\nu \vert \nu\in \mathbb{X}_\bullet\}$ coincide with the collection of exceptional objects that come from the mutation of $\{\mathrm{av}_{\calI\calW}(\scrJ_\nu)\colon\nu\in \bbX_\bullet\}$ with respect to the Bruhat order on $\bbX_\bullet$. This is shown by the closure relation of affine Schubert varieties proved in \Cref{sec:geometry-affine-flag-1-stratification-of-affine-flag-variety}. The second isomorphism follows from the uniqueness of the dual exceptional collection.		
\end{proof}

We have the following immediate corollary.
\begin{corollary}\label{coro_F_IW_is_t_exact}Assume $G$ has enough minuscules. Then, the following hold:
	\begin{enumerate}
		\item The functor $\scrF_{\calI\calW}$ is $t$-exact with respect to the exotic $t$-structure on $\scrD_\mathrm{coh}^b([\hat{G}\backslash \hat{\calN}_\mathrm{Spr}])$ and the perverse $t$-structure on $\scrD_{\mathrm{cons}}(\Hk_{\calI\calW})$.  
		\item In addition, the functor $\scrF_{\calI\calW}$ restricts to an equivalence of abelian categories
		\begin{equation}
			^pH^0(\scrF_{\calI\calW})\colon\mathrm{ExCoh}([\hat{G}\backslash\hat{\calN}_{\mathrm{Spr}}])\xrightarrow{\sim} \scrP(\Hk_{\calI\calW}).
		\end{equation}
	\end{enumerate}
\end{corollary}

We have already seen in Section6 that the simple objects of $\scrP(\Hk_{\calI\calW})$ denoted $\mathrm{IC}^{\calI\calW}_\nu$ are in bijection with $\bbX_{\bullet}$. On the coherent side, the space $\mathrm{Hom}(\Delta_\nu^\mathrm{ex},\nabla_\nu^\mathrm{ex})$ is one-dimensional and the image of $\Delta_\nu^\mathrm{ex}$ under any non-zero map is a simple object in $\mathrm{ExCoh}([\hat{G} \backslash\hat{\calN}_{\mathrm{Spr}}])$. We denote this simple object by $L_\nu^\mathrm{ex}$. 
\begin{lemma} \label{lemma_props_of_std_cstd_simple_ex_sheaves}
	\begin{itemize}
		\item[(1)] The realization functor defined in \cite[Section 3.1]{BBDG18}
		$$
		\scrD^b(\mathrm{ExCoh}([\hat{G}\backslash\hat{\calN}_\mathrm{Spr}]))\rightarrow \mathrm{Perf}([\hat{G}\backslash\hat{\calN}_\mathrm{Spr}])
		$$
		is an equivalence of $\infty$-categories. 
		\item [(2)] For any $\nu\in \bbX_\bullet^+$, there are isomorphisms
		\begin{equation}
			\nabla_\nu^\mathrm{ex}\simeq \calO(\nu),
		\end{equation}
		\begin{equation}
			\Delta_{-\nu}^\mathrm{ex}\simeq \calO(-\nu)
		\end{equation}  
		\item[(3)] For any $\nu\in \bbX_\bullet$, there are isomorphisms
		\begin{equation}
			\nabla_\nu^\mathrm{ex}|_{\tilde{\calO}_{\mathrm{r}}}\simeq \calO(\nu^+)|_{\tilde{\calO}_{\mathrm{r}}},
		\end{equation}
		\begin{equation}
			\Delta_\nu^\mathrm{ex}|_{\tilde{\calO}_{\mathrm{r}}}\simeq \calO(\nu^-)|_{\tilde{\calO}_{\mathrm{r}}},
		\end{equation} where $\tilde{\calO}_{\mathrm{r}} \subset \hat{\calN}_{\mathrm{Spr}}$ is the isomorphic preimage of the regular orbit $\calO_\mathrm{r}\subset \hat{N}$ under the Springer map, and $\nu^+$ (resp.~$\nu^-$) is the dominant (resp.~anti-dominant) $W_\mathrm{fin}$-conjugate of $\nu$.
		\item[(4)] Assume $G$ has enough minuscules. Then, for any $\nu\in \bbX_\bullet$, there is an isomorphism
		\begin{equation}
			\scrF_{\calI\calW}(L_\nu^\mathrm{ex})\cong \mathrm{IC}_\nu^{\calI\calW}.
		\end{equation}
	\end{itemize}
\end{lemma}
\begin{proof}
	The first three properties appear in \cite[Corollary 7.1.6, Lemmas 7.2.1 and 7.2.2]{AR}. We will explain how they follow from \Cref{prop_F_IW_is_t_exact} for groups with enough minuscules and prove the last claim. It is well-known that $\scrD^b(\scrP(\Hk_{\calI\calW}))\cong \scrD_{\mathrm{cons}}(\Hk_{\calI\calW})$. Assertion  $(1)$ then follows from the equivalence (\ref{equation_AB-equivalence}).   The first isomorphism in assertion $(2)$ can be easily deduced from (\ref{equation_av_IW_of_std_objects}) and \Cref{prop_F_IW_is_t_exact}, and the second isomorphism follows analogously. The statement $(3)$ follows from a standard induction argument on the length of the minimal element
	$w \in W_\mathrm{fin}$ such that $\nu=w\nu^+$(resp. $\nu=w\nu^{\minus}$) using $(2)$ and \cite[Proposition 7.1.4]{AR}.
	Assertion $(4)$ follows directly from \Cref{prop_F_IW_is_t_exact} and \Cref{coro_F_IW_is_t_exact}.
\end{proof}

\begin{corollary} \label{Corollary_excoh_restricts_to_equiv_vbs}
	For any $A\in \mathrm{ExCoh}(\hat{\calN}_{\mathrm{Spr}})$, $A\vert_{\calO_\mathrm{r}}$ is a $\hat{G}$-equivariant vector bundle on $\calO_\mathrm{r}$.
\end{corollary}
\begin{proof}
	The result follows from $\hat{G}/Z_{\hat{G}}(n_0)\simeq {\calO}_\mathrm{r}$ and  \Cref{lemma_props_of_std_cstd_simple_ex_sheaves}.  
\end{proof}
\begin{proposition}\label{prop_rest_of_simple_ex_sheaves}
	For $\nu\in \bbX_\bullet$, we have
	\begin{equation}   L_\nu^\mathrm{ex}\vert_{\tilde{\calO}_\mathrm{r}}\simeq \left \{ \begin{aligned}  & \calO & \mathrm{ if }\,\ell(w_\nu)=0 \\
			& 0 & \mathrm{otherwise}
		\end{aligned}\right.
	\end{equation}
\end{proposition}
\begin{proof}
	This is \cite[Proposition 7.2.4]{AR} and we sketch it here for groups with enough minuscules. If $\ell(w_\nu)=0$, then $\nabla_\nu^{\calI\calW}\cong \mathrm{IC}_\nu^{\calI\calW}$, and we conclude the proof by \Cref{lemma_props_of_std_cstd_simple_ex_sheaves}. In general, there exists a unique $\mu\in\bbX_\bullet$ such that $\ell(w_\mu)=0$ and $\Fl_{\calI,w_\nu}$ and $\Fl_{\calI,w_\mu}$ belong to the same connected component. Then the proof of \Cref{Lemma_IW_averaging_kills} and \Cref{aviw_is_t_exact} imply that $\mathrm{IC}_{w_\mu}^{\calI\calW}$ is a composition factor of $\nabla_\mu^{\calI\calW}$. Then \Cref{lemma_props_of_std_cstd_simple_ex_sheaves} yields that $L_\mu^\mathrm{ex}$ is a composition factor of $\nabla_\nu^\mathrm{ex}$. Combining \Cref{lem_socle_top_stand_costand}, \Cref{lemma_props_of_std_cstd_simple_ex_sheaves}, \Cref{Corollary_excoh_restricts_to_equiv_vbs} and the previous discussion, both $L_\mu^\mathrm{ex}$ and $\nabla_\nu^\mathrm{ex}$ restricts to an equivariant $\hat{G}$-line bundle on $\tilde{\calO}_\mathrm{r}$. In particular,  as a composition factor of $\nabla_\nu^\mathrm{ex}$, $L_\nu^\mathrm{ex}$ restricts to $0$.
\end{proof}
Denote by $ \mathrm{Perf}([\hat{G}\backslash \hat{\calN}_{\mathrm{Spr}}])_{\mathrm{nr}}$ the full subcategory of perfect complexes supported on the complement of $\tilde{\calO}_\mathrm{r}$. 
\begin{lemma} \label{lemma_coh_cat_generated_by_ex_IC}
	The category $ \mathrm{Perf}([\hat{G}\backslash \hat{\calN}_{\mathrm{Spr}}])_{\mathrm{nr}}$ is generated by $\{L_\nu^\mathrm{ex}\vert \nu\in \bbX_\bullet, \ell(w_\nu)>0\}$ under cones and shifts.
\end{lemma}
\begin{proof}
	This is \cite[Lemma 7.2.7]{AR} and we could also prove it for groups with enough minuscules via \Cref{Corollary_excoh_restricts_to_equiv_vbs} and \Cref{prop_rest_of_simple_ex_sheaves}.
\end{proof}

For the rest of this section, we apply the previous discussion to study the relation between $\scrP_0(\Hk_{\calI})$ and its full subcategory $\scrP^c_0(\Hk_{\calI})$, culminating in the proof that they coincide assuming the existence of enough minuscules and so do $H \subset Z_{\hat{G}}(n_0)$ as promised in \Cref{remark_regular_quotient}.		
Recall we define the functor
\begin{equation}
	\Pi^0:\scrP(\Hk_{\calI})\rightarrow \scrP_0(\Hk_{\calI})
\end{equation}
in \S8. By definition, it factors through the anti-spherical category 
and we will denote by 
$\Pi^0_\mathrm{as}$ the resulting functor $\scrP_\mathrm{as}(\Hk_{\calI})\rightarrow \scrP_0(\Hk_{\calI})$. \Cref{thm_A-B} and \Cref{coro_av_IWas_is_an_equiv} show that 
\begin{equation}
	\scrF^\mathrm{as}:=\scrD^b(\Pi^0_\mathrm{as})\circ \scrF:\mathrm{Perf}([\hat{G}\backslash \hat{\calN}_{\mathrm{Spr}}])\rightarrow \scrD^b(\scrP_\mathrm{as}(\Hk_{\calI}))
\end{equation}
is an equivalence for groups $G$ with enough minuscules.
\begin{proposition}\label{prop_equivalence_F^r}
	Assume $G$ has enough minuscules. Then, there exists a unique t-exact equivalence of $\infty$-categories
	\begin{equation}
		\scrF^\mathrm{r}:\mathrm{Perf}([\hat{G}\backslash \calO_\mathrm{r}])\rightarrow \scrD^b(\scrP_0(\Hk_{\calI})),
	\end{equation}
	fitting into the commutative diagram
	\begin{equation}
		\begin{tikzcd}[row sep=huge]
			\mathrm{Perf}([\hat{G}\backslash \hat{\calN}_{\mathrm{Spr}}])\arrow[r,"\scrF^\mathrm{as}"] \arrow[d] & \scrD^b(\scrP_\mathrm{as}(\Hk_{\calI}))\arrow[d,"\scrD^b(\Pi^0_\mathrm{as})"]\\
			\mathrm{Perf}([\hat{G}\backslash \calO_\mathrm{r}])\arrow[r,"\scrF^\mathrm{r}"] & \scrD^b(\scrP_0(\Hk_{\calI})),
		\end{tikzcd}
	\end{equation}
	where the left vertical arrow is induced by restriction.
\end{proposition}
\begin{proof}
	The proof follows from the idea of \cite[Proposition 7.2.6]{AR}. We sketch the proof here and refer to \textit{loc.cit} for details.
	We first observe that $\scrD^b(\Pi^0_\mathrm{as})\circ \scrF^\mathrm{as}(L_\nu^\mathrm{ex})=\Pi^0(\mathrm{IC}_{w_{\nu}})$ by \Cref{lemma_props_of_std_cstd_simple_ex_sheaves} and \Cref{coro_av_IWas_is_an_equiv}. Then it follows from \Cref{lemma_coh_cat_generated_by_ex_IC} that
	$\scrD^b(\Pi^0_\mathrm{as})\circ \scrF^\mathrm{as}$ restricts to zero on $\mathrm{Perf}(\hat{G}\backslash \hat{\calN}_\mathrm{Spr})_{\mathrm{nr}}$. Note that $\mathrm{Perf}(\hat{G}\backslash \calO_{\mathrm{r}})$ is the quotient in $\mathrm{Cat}_\infty$ of $\mathrm{Perf}(\hat{\calN}_{\mathrm{Spr}})$ by the non-regular full subcategory by \Cref{prop_thomason_quotient} (idempotent completions are not necessary as the Springer variety and the regular orbit are smooth). On the other hand, as noticed before this proposition, the functor
	\begin{equation}
		\scrD^b(\Pi^0_\mathrm{as}):\scrD^b(\scrP_{\mathrm{as}}(\Hk_\calI))\rightarrow \scrD^b(\scrP_0(\Hk_{\calI}))
	\end{equation}			
	is a quotient map in $\mathrm{Cat}_\infty$ with kernel given by the full subcategory generated by $\Pi^{\mathrm{as}}(\IC_w)$ with $\ell(w)>0$. Thus $\scrF^\mathrm{r}$ is an equivalence. 
	
	By \Cref{coro_F_IW_is_t_exact}, the restriction functor $\scrF^\mathrm{as}: \mathrm{Perf}([\hat{G}\backslash\hat{\calN}_{\mathrm{Spr}}])\rightarrow \scrD^b(\scrP_\mathrm{as}(\Hk_{\calI}))$ is $t$-exact with respect to the exotic $t$-structure on the source and the tautological
	$t$-structure on the target. Then, to prove $\scrF^\mathrm{r}$ is $t$-exact, it suffices to show that every simple object in $\mathrm{Coh}([\hat{G}\backslash \calO_\mathrm{r}])$ is the restriction of an exotic coherent sheaf. The verification of the latter assertion can be argued entirely on the coherent side as in \cite[Proposition 7.2.6]{AR}.
\end{proof}

\begin{proposition}\label{H_equals_the_full_centralizer}
	With notations in \Cref{prop_tannakian_for_perv_I0}, we have $\scrP_0^c(\Hk_{\calI})=\scrP_0(\Hk_{\calI})$, and $H=Z_{\hat{G}}(n_0)$.  
\end{proposition}
\begin{proof}
	The proof is completely analogous to \cite[Proposition 7.2.8]{AR} by our previous preparations. By construction, the projective objects in
	$\mathrm{Coh}([\hat{G}\backslash\hat{\calN}_\mathrm{Spr}])$ map to $\scrP_0^c(\Hk_{\calI})$ under $\Pi_{\mathrm{as}}^0\circ \scrF^{\mathrm{as}}$. Also every coherent sheaf on the regular orbit is a quotient of a projective object in 	$\mathrm{Coh}([\hat{G}\backslash\hat{\calN}_\mathrm{Spr}])$ by \cite[Lemma 7.2.9]{AR}. Then it follows from \Cref{prop_equivalence_F^r} that the first assertion holds. Recall the equivalence $\Phi^0$ from \Cref{prop_tannakian_for_perv_I0} between $\mathrm{Rep}H$ and $\scrP_0(\Hk_{\calI})$. On the coherent side, we have an equivalence $\Psi: \mathrm{Coh}(\hat{G}\backslash \calO_\mathrm{r})\simeq \mathrm{Rep}(Z_{\hat{G}}(n_0))$ induced by the isomorphism $\calO_{\mathrm{r}}\simeq \hat{G}/Z_{\hat{G}}(n_0)$ by the definition of the regular orbit itself. The second statement follows by showing that $\mathrm{For}^{Z_{\hat{G}}(n_0)}_{H}\circ \Psi$ is equivalent to $ \Phi^0\circ F^\mathrm{r}$ and we refer to the end of the proof of \cite[Proposition 7.2.8]{AR} for details.		
\end{proof}

\section{Equivariant coherent sheaves on the nilpotent cone}
Recall the Springer resolution 
\begin{equation}
	p_\mathrm{Spr}:\hat{\calN}_{\mathrm{Spr}}=\hat{G} \times^{\hat{B}} \mathrm{Lie}(\hat{U})\rightarrow \hat{\calN}
\end{equation}
of the nilpotent cone of the dual group $\hat{G}$, defined over the coefficient field $\La=\bar \bbQ_\ell$. In this section, we study the category $\mathrm{Coh}([\hat{G}\backslash \hat{\calN}])$ by establishing a connection with a certain quotient of $\scrP(\Hk_{\calI})$ and proving main results of \cite{Bez09} in the mixed-characteristic setting for groups with enough minuscules.

Let $\scrP_{\mathrm{bas}}(\Hk_{\calI})$ denote the quotient of $\scrP(\Hk_\calI)$ by the Serre subcategory spanned by the IC sheaves of $\Fl_{ \calI,\leq w}$ for non-minimal $w$ in its $W_{\mathrm{fin}}$-double coset. Recall the anti-spherical category $\scrP_\mathrm{as}(\Hk_{\calI})$ in \Cref{avIW_factos_through_pervas}. The natural functor $\scrP(\Hk_{\calI})\rightarrow\scrP_{\mathrm{bas}}(\Hk_{\calI})	$
factors through the quotient  
\begin{equation}
	\Pi^{\mathrm{as}}_{\mathrm{bas}}:\scrP_\mathrm{as}(\Hk_{\calI})\rightarrow \scrP_{\mathrm{bas}}(\Hk_{\calI}).
\end{equation}
In the sequel, we will relate this category to equivariant coherent sheaves on the nilpotent cone.	
\begin{theorem}\label{perverse_description_of_coherent_sheaves}
	Assume $G$ has enough minuscules. Then, there exists a unique equivalence of $\infty$-categories:
	\begin{equation}
		\scrF_{\mathrm{bas}}\colon \scrD^b_\mathrm{coh}([\hat{G}\backslash\hat{\calN}])\rightarrow \scrD^b(\scrP_{\mathrm{bas}}(\Hk_{\calI})) ,
	\end{equation}
	making the following diagram commute
	\begin{equation}
		\begin{tikzcd}[row sep=huge]
			\scrD^b_\mathrm{coh}([\hat{G}\backslash\hat{\calN}_{\mathrm{Spr}}]) \arrow[r, "\scrF_\mathrm{as}"] \arrow [d,"Rp_{\mathrm{Spr}*}"] & \scrD^b(\scrP_\mathrm{as}(\Hk_{\calI})) \arrow [d,"\scrD^b({\Pi^{\mathrm{as}}_{\mathrm{bas}}})"] \\
			\scrD^b_\mathrm{coh}([\hat{G}\backslash\hat{\calN}])\arrow[r,"\scrF_\mathrm{bas}"] &\scrD^b(\scrP_{\mathrm{bas}}(\Hk_{\calI})),
		\end{tikzcd}
	\end{equation}
	where $\scrF_\mathrm{as}$ is the composition of $\scrF$ with the functor $\scrD^b(\scrP(\Hk_{\calI}))\rightarrow \scrD^b(\scrP_\mathrm{as}(\Hk_{\calI}))$ induced by the quotient functor $\scrP(\Hk_{\calI})\rightarrow \scrP_\mathrm{as}(\Hk_{\calI})$.
\end{theorem}
\begin{proof}
	The proof follows the idea of \cite[Theorem 1]{Bez09} and \cite[Theorem 7.3.1]{AR} in equicharacteristic and we sketch it here. Recall that $\scrD^b_\mathrm{coh}([\hat{G}\backslash\hat{\calN}_{\mathrm{Spr}}])$ is the bounded derived category of its abelian heart for the exotic t-structure, see \Cref{lemma_props_of_std_cstd_simple_ex_sheaves}. Let $\scrD$ be the Verdier quotient of $\scrD_\mathrm{coh}^b([\hat{G}\backslash\hat{\calN}_{\mathrm{Spr}}])$ by the full subcategory spanned by the $L_\nu^\mathrm{ex}$ with $\nu\not\in \bbX_{\bullet}^-$ under cones and extensions. Then $Rp_{\mathrm{Spr}*}$ factors as the composition of the quotient  $\Pi:\scrD^b_\mathrm{coh}([\hat{G}\backslash\hat{\calN}_{\mathrm{Spr}}])\rightarrow\scrD$ and a functor $\alpha:\scrD\rightarrow \scrD^b_\mathrm{coh}([\hat{G}\backslash{\calN}])$ since $Rp_{\mathrm{Spr*}}(L^\mathrm{ex}_\nu)=0$ for any $\nu\not\in \bbX_\bullet^-$, compare with \cite[Lemma 1]{Bez09} and \cite[Lemma 7.3.3]{AR}. 
	
	Similarly, let $\scrD'$ be the Verdier quotient of $\scrD^b(\scrP_\mathrm{as}(\Hk_{\calI}))$ by the full subcategory spanned by the IC sheaves of the form $\mathrm{IC}_{w_\nu}$ with $\nu \not\in \bbX_{\bullet}^-$. Then $\scrD^b(\Pi^{\mathrm{as}}_{\mathrm{bas}})$ factors through $\scrD'$ via a functor $\alpha' \colon \scrD'\to \scrD^b(\scrP_{\mathrm{bas}}(\Hk_{\calI}))$. We know by \cite[Theorem 3.2]{Miy91} that $\alpha'$ is an equivalence.  The equivalence $\scrF_\mathrm{as}$ induces an equivalence $\scrD\simeq \scrD'$ by  \Cref{coro_av_IWas_is_an_equiv} and \Cref{lemma_props_of_std_cstd_simple_ex_sheaves}. Hence, it suffices to show $\alpha$ is an equivalence. The essential surjectivity follows from \cite[Lemma 7]{Bez03}  
	and full faithfulness follows from the abstract \cite[Lemma 7.3.13]{AR} together with a few input calculations.
\end{proof}

In unpublished notes, Deligne introduces an analogue of the perverse $t$-structure \cite{BBDG18} on the derived category of coherent sheaves on a Noetherian scheme with a dualizing complex. This $t$-structure has been studied and extended by Arinkin--Bezrukavnikov \cite{AB10}. In this subsection, we compare the perverse $t$-structure on $\scrD^b_\mathrm{coh}([\hat{G}\backslash\calN])$ with the exotic $t$-structure transported from the equivalence of \Cref{perverse_description_of_coherent_sheaves} and prove \cite[Theorem 2, Corollary 1]{Bez09} in our setting.

The following lemma is due to Bezrukavnikov \cite{Bez03}.
\begin{lemma}\label{lemma_perverse_coh_t_strucutre}
	The perverse coherent $t$-structure corresponding to the perversity function $p(O)=\mathrm{codim} (O)/2$ is the unique $t$-structure which has all $p_{\mathrm{Spr}*}(\calO(\nu))$ lie in its heart.
\end{lemma}
\begin{proof}
	This is \cite[Corollary 3]{Bez03}.
\end{proof}
Note that we have an exotic t-structure on $\scrD^b_\mathrm{coh}([\hat{G}\backslash \hat{\calN}])$ inherited from the derived category of the Springer variety in virtue of our realization of the former as a Verdier quotient of the latter in \Cref{thm_perverse_coh_compare_with_exotic_perverse}.

\begin{corollary}\label{thm_perverse_coh_compare_with_exotic_perverse}
	The exotic $t$-structure on $\scrD^b_\mathrm{coh}([\hat{G}\backslash \hat{\calN}])$ identifies with the perverse coherent $t$-structure with perversity function $p(O)=\mathrm{codim} (O)/2$.
\end{corollary}
\begin{proof}
	The statement follows directly from \Cref{perverse_description_of_coherent_sheaves} in light of \Cref{lemma_perverse_coh_t_strucutre}.
\end{proof}

\bibliography{biblio.bib}
\bibliographystyle{alpha}

\end{document}